\documentclass[12pt]{article}
\usepackage{amsmath,latexsym,amssymb,amsthm,enumerate,amsmath,amscd}
\usepackage{amsthm,amsmath,amsfonts,amssymb,amscd,latexsym,stmaryrd}
\usepackage[mathscr]{eucal}
\usepackage{dsfont}
\usepackage{lscape} 
\usepackage{color}
\usepackage[usenames,dvipsnames]{xcolor}
\usepackage{colortbl}
\usepackage[all,cmtip]{xy}
\usepackage{graphicx}
\usepackage{ytableau}
\usepackage{comment}
\definecolor{light-gray}{gray}{0.9}
\definecolor{med-gray}{gray}{0.5}
\definecolor{gray1}{gray}{0.87}
\definecolor{gray2}{gray}{0.74}
\definecolor{gray3}{gray}{0.64}
\definecolor{gray4}{gray}{0.55}
\definecolor{verylight-yellow}{rgb}{1,1,0.7}
\definecolor{beaublue}{rgb}{0.74, 0.83, 0.9}
\definecolor{yellow}{rgb}{1,1,0.2}
\definecolor{vivid-blue}{rgb}{0.2,0,1}
\definecolor{light-pink}{rgb}{1,0.8,1}
\definecolor{med-pink}{rgb}{1,0.6,1}
\definecolor{aqua}{rgb}{0.0, 1.0, 1.0}
\definecolor{light-gray}{rgb}{0.5, 0.9, 0.5}
\definecolor{cadmiumgreen}{rgb}{0.0, 0.42, 0.24}

\usepackage{comment}
\usepackage{hyperref}
\usepackage{tikz}
\usetikzlibrary{calc,shadings,patterns,bending,arrows.meta}
\usepackage{fontawesome}
\theoremstyle{plain}
\newtheorem{theorem}{Theorem}[section]
\newtheorem{proposition}[theorem]{Proposition}
\newtheorem{corollary}[theorem]{Corollary}

\newtheorem{lemma}[theorem]{Lemma}
\theoremstyle{definition}
\newtheorem{definition}[theorem]{Definition}

\newtheorem{remark}[theorem]{Remark}
\newtheorem{example}[theorem]{Example}

\newtheorem{observation}[theorem]{Observation}

\newtheorem{claim}[theorem]{Claim}

\usepackage[T1]{fontenc}
\usepackage[utf8]{inputenc}
\usepackage[english]{babel}
\usepackage{multicol}
\usepackage{fancyhdr}
\usepackage{fancybox}
\usepackage{eurosym}
\usepackage{xcolor}
\usepackage{tcolorbox}
\usepackage{colortbl}
\usepackage[draft]{pdfpages}
\usepackage{textcomp}
\usepackage{graphicx}
\usepackage{tikz}
\usepackage{ytableau}
\usepackage{enumerate}

\newtheorem{thm}{Theorem}

\newtheorem{cor}[theorem]{Corollary}

\newtheorem{rem}[thm]{Remark}

\newtheorem*{remark*}{Remark}

\newtheorem*{ack}{Acknowledgment}

\numberwithin{equation}{section}
\numberwithin{table}{section}
\DeclareMathOperator{\Grass}{\rm{Grass}}

\newcommand{\compmm}[1]{{\color{blue}#1}}
\definecolor{purple}{rgb}{0.4,0.2,0.4}

\def\Mon{\mathrm{Mon}}
\def\hMon{\mathrm{hMon}}

\def\Grass{\mathrm{Grass}}

\def\<{\left<}
\def\>{\right>}

\def\F{{\sf k}}
\def\G{\mathrm{G}}

\def\ns{\footnotesize \it}

\def\max{\mathrm{max}}

\definecolor{med-gray}{gray}{0.5}
\definecolor{gray1}{gray}{0.87}
\definecolor{gray2}{gray}{0.74}
\definecolor{gray3}{gray}{0.64}
\definecolor{gray4}{gray}{0.48}
\definecolor{verylight-yellow}{rgb}{1,1,0.7}
\definecolor{yellow}{rgb}{1,1,0.2}
\definecolor{vivid-blue}{rgb}{0.2,0,1}
\definecolor{light-pink}{rgb}{1,0.8,1}
\definecolor{med-pink}{rgb}{1,0.6,1}
\definecolor{aqua}{rgb}{0.0, 1.0, 1.0}
\definecolor{light-gray}{rgb}{0.5, 0.9, 0.5}

\begin{document}
\date{}

\author{Nasrin Altafi\\[.05in]
{\ns Department of Mathematics, KTH Royal Institute of Technology, S-100 44 Stockholm, Sweden.
}\\[.2in] Anthony Iarrobino\\[.05in]
{\ns Department of Mathematics, Northeastern University, Boston, MA 02115,
 USA.
}\\[.2in] 
Leila Khatami\\[0.05in]
{\ns Union College, Schenectady, New York, 
 12308, USA.}
 \\[.2in]
Joachim Yam\'{e}ogo\\[0.05in]
{\ns Universit\'e C\^ote d'Azur, CNRS, LJAD, FRANCE.}
}
\title{Number of generators of ideals in Jordan cells of the family of graded Artinian algebras of height two
\footnote{\textbf{Keywords}: Artinian algebra, Hilbert function, hook code, Jordan type, partition, cellular decomposition, graded ideal. \textbf{2010 Mathematics Subject Classification}: Primary: 13E10;  Secondary: 05A17, 05E40, 13D40, 14C05.
\quad Email addresses: {\sf nasrinar@kth.se, a.iarrobino@northeastern.edu, khatamil@union.edu, joachim.yameogo@unice.fr}}}
\maketitle
\abstract
We let $A=R/I$ be a standard graded Artinian algebra quotient of $R={\sf k}[x,y]$, the polynomial ring in two variables over a field ${\sf k}$ by an ideal $I$, and let $n$ be its vector space dimension. 
The Jordan type $P_\ell$ of a linear form $\ell\in A_1$ is the partition of $n$ determining the Jordan block decomposition of the multiplication on $A$ by $\ell$ -- which is nilpotent. The first three authors previously determined which partitions of $n=\dim_{\sf k}A$ may occur as the Jordan type for some linear form $\ell$ on a graded complete intersection Artinian quotient $A=R/(f,g)$ of $R$, and they counted the number of such partitions for each complete intersection Hilbert function $T$ \cite{AIK}.\par
 We here consider the family $\G_T$ of graded Artinian quotients  $A=R/I$ of $R={\sf k}[x,y]$, having arbitrary Hilbert function $H(A)=T$. The Jordan cell $\mathbb V(E_P)$ corresponding to a partition $P$ having diagonal lengths $T$ is comprised of all ideals $I$ in $R$ whose initial ideal is the monomial ideal $E_P$ determined by $P$.  These cells give a decomposition of the variety $\G_T$ into affine spaces.  We determine the generic number $\kappa(P)$ of generators for the ideals in each cell $\mathbb V(E_P)$, generalizing a result of \cite{AIK}. In particular, we determine those partitions for which $\kappa(P)=\kappa(T)$, the generic number of generators for an ideal defining an algebra $A$ in $\G_T$. We also count the number of partitions $P$ of diagonal lengths $T$ having a given $\kappa(P)$. A main tool is a combinatorial and geometric result  allowing us to split $T$ and any partition $P$ of diagonal lengths $T$ into simpler $T_i$ and partitions $P_i$, such that $\mathbb V(E_P)$ is the product of the cells $\mathbb V(E_{P_i})$, and $T_i$ is single-block:  $G_{T_i}$ is a Grassmannian.
 
\tableofcontents\vskip 0.3cm\noindent
\section{Introduction.}
Let $A$ be a standard-graded Artinian algebra $A=\mathop{\oplus}\limits_{i=0}^{\sf j}A_i$ over a field $\F$,  and let $\ell\in A_1$ be a linear form. The multiplication map $m_\ell: A \to A:  a \to \ell\cdot a$ is nilpotent. The \emph{Jordan type} $P_\ell=P_{\ell,A}$ is a partition of $n=\dim_\F A$, giving the Jordan block decomposition of the multiplication map $m_\ell$. 
We consider standard graded Artinian algebra quotients $A=R/I$ where $I$ is an ideal of $ R={\sf k}[x,y]$, the polynomial ring $R$ in two variables over $\sf k$. 
We will assume that $I_1=0$, so the height (or codimension) $\dim_\F A_1$ of $A$ is two. The \emph{order} of the graded ideal $I$ is the lowest degree of a (non-zero) element.
\par
The Hilbert function $T=H(A)$ of such a graded Artinian algebra  $A$ in codimension two is a sequence of the following form 
\begin{equation}\label{Teq}
T=\left(1,2,\dots , d, t_d,t_{d+1},\dots ,t_{\sf j},0\right) \text { where } d\ge t_d\ge t_{d+1}\ge \cdots \ge t_{\sf j}>0,
\end{equation}where $t_i=\dim_{\sf k} A_i$, $\sf j$ is the {\emph{socle degree}} of $T$, $d$ is the \emph{order} of $T$ -- the order of any ideal $I$ such that $H(R/I)=T$ -- and $n=|T|=\sum t_i=\dim_{\sf k}A$.\par
\par
The first three authors in \cite{AIK} determined all possible Jordan types $P_\ell$ of linear forms $\ell$ for complete intersection (CI) graded Artinian algebras of height two: they assumed that the field $\sf k$ either has characteristic zero, or is infinite of characteristic $p>{\sf j}$, the socle degree of the algebra. In this paper we make the same assumption on characteristic because we use a standard-basis result due to J. Brian\c{c}on and A. Galligo in showing Lemma \ref{corner-kick-off-lemma}, that requires the restriction.
Our main results in this paper generalize those of \cite{AIK} to all height two Hilbert functions. \par

  Let $P$ be a partition and let $E_P$ be the monomial ideal determined by the Ferrers diagram of $P$. The \emph{diagonal lengths} of the Ferrers diagram of $P$ is just the Hilbert function $T_P=H(R/E_P)$ (Definition \ref{EPdef}).  The set $\mathcal P(T)$ of partitions having diagonal lengths $T$ -- or, equivalently -- the set of monomial ideals $E_P$ such that $H(R/E_P)=T$, has been studied by the second and last author in \cite{IY}, as well as by others, including \cite{Ev,Con}. We denote by $\G_T$ the family of graded Artinian quotients $A=R/I$ where $I$ is an ideal of $R={\sf k}[x,y]$, for which the Hilbert function $H(A)=T$. This is a smooth projective variety $\G_T$, that is locally an affine space of known dimension (Proposition \ref{GTdimthm}). Let $\mathbb V(E_P)$ be the affine cell of $\G_T$ that parametrizes the ideals of $R$ having initial ideal $E_P$ (Definition \ref{celldef}). We will term $\mathbb V(E_P)$ a \emph{Jordan cell}.\footnote{The local analogue $Z_T$ of $\G_T$ has been termed a \emph{vertical} cell by J. Brian\c{c}on \cite{brian}; $\G_T$ has been termed an
  Ellingsrud-Str\"{o}mme-G\"{o}ttsche cell \cite{IY}, defined using a $\mathbb C^\ast$ action. The fourth author showed that these concepts are the same both for $\G_T$ and the associated parameter space $Z_T$ for local algebras \cite{jy-1}.}
 L.~G\"{o}ttsche showed that the cells $\mathbb V(E_P)$ for $P\in \mathcal P(T)$ form a cellular decomposition of $\G_T$ \cite{Got}; this followed the analogous results of G. Ellingsrud and S.A. Str{\o}mme for the punctual Hilbert scheme of projective space $\mathbb P^2$, using the A. Bialynicki-Birula theorem \cite{ES1,ES2,Bia}. This cellular decomposition was as well studied by G.~Gotzmann \cite{Gm2}, and the fourth author \cite{jy-1,jy-2}. Generators and relations for height two graded ideals have been studied by many, as \cite{Bu,brian, brian-gal,MR,Con}; J.O. Kleppe studies a scheme analogue of $\G_T$ \cite{Kl}, L.~Evain discusses an equivariant Hilbert scheme, involving weights of the variables \cite{Ev};  some combinatorics of cells in \cite{IY} are seen in a larger context in~\cite{LW}. The cells $\mathbb V(E_P)$ and their connection with generators and relations of ideals have also been studied by A.~Conca and G.~Valla in  \cite{CoVa}.\vskip 0.2cm\par\noindent
 By a \emph{generic} element of an irreducible algebraic variety $X$ we will mean an element belonging to a certain non-empty Zariski-dense open subset $U\subset X$.\par
The main results of this paper, {\bf Theorems \ref{kappathm} and \ref{componentTheorem}}, give $\kappa(P)$, the minimum number of generators for the ideal $I$ defining a generic element $A=R/I$ in the cell $\mathbb{V}(E_P)$, for each partition $P\in \mathcal P(T)$, where $T$ is a Hilbert function of a height two graded Artinian algebra, namely a sequence of the form given in Equation~\eqref{Teq}. In {\bf Theorem
\ref{multicountthm}}, we give the number of partitions $P\in \mathcal P(T)$ with $\kappa(P)=k$, for any positive integer $k$.
 
\vskip 0.2cm\par\noindent
{\bf Summary.}
{\it Decomposition of cells.}
 We first prove in Section \ref{prodsec} a  new combinatorial and geometric result allowing us to associate to an arbitrary Hilbert function $T$ and to any partition $P$ of diagonal lengths $T$ their \emph{components}, a set of simpler, single-block sequences $T_i$ and partitions $P_i$ of diagonal lengths $T_i$. We show that the Jordan cell $\mathbb V(E_P)$ of $\G_T$ is in a natural way the product of its single-block components, the cells $\mathbb V(E_{P_i})$ of $\G_{T_i}$  ({\bf Theorem~\ref{projectionthm}}).  This decomposition is closely related to the \emph{hook codes} that had been studied in \cite{IY} and that we define in Section  \ref{hookcodesec}.\vskip 0.2cm\par\noindent
 {\it Single-block case.}
In the \emph{single-block} case the Hilbert function sequence $T$ of Equation \eqref{Teq} satisfies $d={\sf j}$, so
 \begin{equation}\label{single-block1eq}
 T=(1,2,\ldots,d,t,0),
 \end{equation}
 that is $t_i=i+1$ for $0\le i\le d-1$ and $ t_d=t$. 
 An ideal $I\subset R$ defining an algebra $A=R/I$ of Hilbert function $T$ of  Equation~\eqref{single-block1eq} satisfies
  \begin{equation}\label{single-block2eq}
  I=V\oplus \mathfrak m^{d+1},
  \end{equation}
  where we let $V=I_d\subset R_d$ and $\mathfrak m$ is the maximal ideal of~$R$.
 Thus, the projective variety $\G_T$ in the single-block case is isomorphic to the Grassmannian $\Grass(s,R_d), s=d+1-t$, parametrizing $s$-dimensional subspaces $V\subset R_d$.  In {\bf Theorem \ref{kappathm}} we determine the integer $\kappa(P)$, the number of generators of a generic ideal in the cell $\mathbb V(E_P)$, for single-block partitions, and in {\bf Theorem~\ref{countingthm}}, we give the number of \emph{special} single-block partitions $P$, namely partitions $P$ with $\kappa(P)>\kappa(T)$, the generic - and minimum - number of generators for an ideal defining an algebra $A$ in $\G_T$. More generally, in {\bf Corollary \ref{singlecountcor}} we determine the number of partitions $P\in \mathcal P(T)$ with $\kappa(P)=k$, for any positive integer $k$.
\vskip 0.2cm\par\noindent
 {\it General case.}
 In {\bf Theorem \ref{componentTheorem}} we determine $\kappa(P)$ for an arbitrary partition $P\in \mathcal P(T)$ in terms of the $\kappa(P_i)$ of their single-block components.  We show that $P$ is special ($\kappa(P)\not= \kappa(T)$) if and only if some component $P_i$ is special ({\bf Theorem \ref{componentthm}}). In {\bf Theorem~\ref{multicountthm}} we determine $\kappa(P)$ for arbitrary Jordan types $P$, and we give the number of partitions $P\in \mathcal P(T)$ with $\kappa(P)=k$, for any positive integer $k$; this result also determines the number of special partitions in $\mathcal P(T)$ (Corollary~\ref{countspecialcor}), and we recover the number of CI partitions shown in \cite{AIK} (Corollary \ref{CIJordantypecor}).

\vskip 0.2cm\par\noindent
{\it Hook code}. We explain in Section \ref{hookcodesec} the \emph{hook code} for partitions $P$ of diagonal lengths $T$.  The hook code of $P$ is a sequence $\mathfrak Q(P)=\left(\mathfrak h_d(P),\ldots  ,\mathfrak h_{\sf j}(P)\right)$ of partitions-in-a-box $\mathfrak {\mathfrak B}_i(T), d\le i\le \sf j$, where the box $\mathfrak {\mathfrak B}_i(T)=(\delta_{i+1})\times (1+\delta_i)$: that is, the Ferrers diagram of each $\mathfrak h_i(P)$ has at most $\delta_{i+1}$ rows and $(1+\delta_i)$ columns (Definition \ref{hookcodedef}). The partitions of diagonal lengths  $T$ are completely determined by their hook code $\mathfrak Q(P)$. Also, the component partition $P_i$ of diagonal lengths $T_i$ has as its hook code the degree-$i$ component  $\mathfrak h_i(P)$ of $\mathfrak Q(P)$ (Proposition \ref{hookP(i)lem}).
\vskip 0.2cm\noindent
  \par
The proofs involve a careful study of standard generators and relations for the ideals $I$ defining algebras in the cell $\mathbb V(E_P)$, using in particular the hook code of a partition $P$. We then compare these invariants to those for the partition $P:x$ corresponding to the ideals $I:x$.  This allows us to compare $\kappa(P)$ with $\kappa (P:x)$, and we thus determine how to compute $\kappa(P)$ from the hook code $\mathfrak Q(P)$. 
\section{Cells of the variety $\G_T$ and their hook codes.}\label{GTchapter}
\subsection{The variety $\G_T$ and the cells $\mathbb V(E_P)$.}\label{GTsec}
We need some basic notions from \cite{IY,IY2} (see also \cite[\S 4.1]{AIK}).\par

Recall that we consider graded Artinian quotients $A=R/I$, where $I$ is an ideal of $R={\sf k}[x,y]$ the polynomial ring over an arbitrary field $\sf k$. The Hilbert function of $A$ is the sequence $H(A)=(1,t_1,\ldots, t_{\sf j} )$ where $t_i=\dim_{\sf k} A_i$ and $\sf j$ is the socle degree of $A$ that is $ A_{\sf j}\not=0, A_{{\sf j} +1}=0$. The family of all such quotients having Hilbert function $H(A)=T$ is denoted by $\G_T$, which has a natural structure of subvariety $\G_T\subset \Pi_{d\le i\le{\sf j}} \Grass(t_i,R_i)$, where $\Grass(t_i,R_i)$ parametrizes quotients $A_i=R_i/I_i$ of vector space dimension $t_i$. Thus we have
$$\iota: \G_T \to \Pi_{d\le i\le {\sf j}} \Grass(t_i,R_i):  A=R/I \to (R_d/I_d,R_{d+1}/I_{d+1},\ldots, R_{\sf j}/ I_{\sf j}).$$
We now explain the affine cell decomposition $\G_T=\bigcup_{P\in \mathcal P(T)}\mathbb V(E_P)$ where $P$ runs through the set $\mathcal P(T)$ of partitions having diagonal lengths $T$.  (Theorem \ref{GTdecompthm}).\par
\begin{definition}[The monomial ideal $E_P$ and diagonal lengths of $P$]\label{EPdef}
Given a partition $P=(p_1,p_2,\ldots, p_t)$ of $n=\sum p_i$ where $p_1\ge p_2\ge\cdots \ge p_t$, we let $C_P$ be the set of $n$  monomials that fill the Ferrers diagram $F_P$ of $P$ as follows: for $i\in [1,t]$ the, $i$-th row counting from the top of $F_P$ is filled by the monomials $y^{i-1},y^{i-1}x,\ldots ,y^{i-1}x^{p_i-1}$. We let $E_P$ be the complementary set of monomials to $C_P$ and denote by $(E_P)$ the monomial ideal generated by $E_P$.
 The \emph{diagonal lengths} $T_P$ of $P$ is the Hilbert function $T_P=H(R/E_P)$.\par
 \end{definition}
   In a Ferrers diagram of monomials associated to a partition $P$ of $n$, the $x$-degrees of monomials increase as we go from left to right and the $y$-degrees increase as we go from top to bottom. We count the columns from left to right and the rows from top to bottom.  See Figure \ref{Ferrersfig} for the Ferrers diagram of the partition $P=(5,3,1)$ of diagonal lengths $T=(1,2,3,2,1)$; and Example \ref{flatHFexample} and Figure \ref{5.8figure} for that of $P=\left(10^2,4,3,2^5\right)$.  
  \begin{figure}[!h]
 \begin{center}
 \begin{ytableau}
 1&x&x^2&x^3&x^4\cr
 y&yx&yx^2\cr
y^2\cr
 \end{ytableau}
 \caption{Ferrers diagram for $P=(5,3,1)$.}\label{Ferrersfig}
 \end{center}
%

 \end{figure}\vskip 0.2cm\begin{definition}\label{hookdef}
A \emph{hook} of a partition $P$ is a subset of the Ferrers diagram $F_P$ consisting of a hook-corner $c$, an \emph{arm} 
 $(c,xc,\ldots ,\nu=x^{u-1}c)$ and a \emph{leg} $(c,yc,\ldots , \mu= y^{v-1}c)$, such that $x\nu\in E_P$ and $y\mu\in E_P$ 
 (Figure \ref{hook fig}). The \emph{arm length} is $u$ and the \emph{leg length} is $v$; the hook has arm-leg \emph{difference} 
 $u-v$. We term the monomial $\nu$ the \emph{hand}, and the monomial $\mu$ the \emph{foot} of the hook.\par
 \end{definition}

\begin{figure}[!h]
\begin{center}
	\begin{tikzpicture}[scale=.6]
	\draw[thin] 	(0,0)--(4,0)
				(0,-1)--(4,-1)
				(0,-2)--(1,-2)
				(0,-3)--(1,-3)
				(0,0)--(0,-3)
				(1,0)--(1,-3)
				(2,0)--(2,-1)
				(3,0)--(3,-1)
				(4,0)--(4,-1);			
	\draw[fill=gray!30] 	(3,0) rectangle (4,-1)
					(0,-2) rectangle (1,-3);
	\draw[black] 	node at (0.5,-0.5) {$c$}
				node at (3.5,-0.5) {$h$}
				node at (0.5,-2.5) {$f$};
         \end{tikzpicture}
         \caption{Difference-one hook with hand $h$, foot $f$, corner $c$, $P=(4,1,1)$.}\label{hook fig}
\end{center}
\end{figure}

\begin{example}\label{431ex} Let $P=(4,3,1)$. $P$ has diagonal lengths $T_P=(1,2,3,2)$. The hook with corner $x$ in the Ferrers diagram $C_P$  has arm length 3, foot length 2, hand $x^3$, foot $yx$, so has (arm $-$ leg)  difference one (Figure \ref{hook1fig}). 
\end{example}

\begin{figure}[!ht]
\begin{center}
$\begin{array}{|c|c|c|c|}
\hline
1&\cellcolor{red}x&\cellcolor{gray2}x^2&\cellcolor{gray2}x^3\\
\cline{1-3}
y&\cellcolor{gray2}yx&yx^2\\
\cline{1-3}
y^2\\
\cline{1-1}
\end{array}$
	\caption{Difference-one hook with corner $x$ for $P=(4,3,1)$.}\label{hook1fig}
\end{center}
\end{figure}\noindent
\begin{definition}[Initial ideal of $I$, and the Jordan cell $\mathbb V(E_P)$]\label{celldef}
We order the monomials of degree $i$ by $x^i<x^{i-1}y<\cdots< y^i$ (lex order).
 The \emph{initial monomial} $\mu(f)=$ in$(f)$ of a form $f=\sum_k a_ky^kx^{i-k}, a_k\in \sf k$  is the monomial $\mu(f)=y^sx^{i-s}$ of highest $y$-degree $s$ among those with non-zero coefficients $a_k$.  Given an ideal $I\subset R={\sf k}[x,y]$, defining the Artinian quotient $A=R/I$ we denote by in$(I)$ the ideal
 \begin{equation*}
{\mathrm{in}}(I)= (\{ \mathrm{in}(f), f\in I\})
 \end{equation*}
 generated by the initial monomials of all elements of $I$. We may identify $\mathrm{in}(I)$ with an ideal $E_P$ for a  partition $P=P(I)$ of diagonal lengths $T=H(A)=H(R/E_P)$. \par
 We denote by $\mathbb V(E_P)$ the affine variety parametrizing all ideals $I\subset R$ having initial ideal $E_P$ (for the affine variety structure see Theorem \ref{GTdecompthm} and \cite[Prop.~2.6]{jy-1}, or Theorem \ref{projectionthm} below).
 \end{definition}

 When counting the minimal number of generators of an ideal $I\in \mathbb V(E_P)$ we will refer to the leading terms of these genarators as corner-monomials of $E_P$. 
\begin{definition}[corner-monomial of $E_P$]\label{def_corner-mon} Let $P$ be a partition of an integer $n$. Denote by $E_P\subset R={\sf k}[x,y]$ the monomial ideal associated to $P$ (Definition \ref{EPdef}). An element of a minimal set of generators of $E_P$ is called a corner-monomial of $E_P$.
\end{definition}
\begin{example}\label{corn-mon_4_4_2_1_1} Let $P=(4,4,2,1,1)$. Then the corner-monomials of $E_P$ are $x^4, x^2y^2,xy^3$  and $y^5$ 
(see Figure \ref{fig_corn-mon}).
 \begin{figure}[!h]
\begin{center}
	\begin{tikzpicture}[scale=.8]
	\draw[dashed,very thin] 	(0,0)--(4,0)
						(0,-1)--(4,-1)
						(0,-2)--(2,-2)
						(0,-3)--(1,-3)
						(0,-4)--(1,-4)
						(0,0)--(0,-5)
						(1,0)--(1,-3)
						(2,0)--(2,-2)
						(3,0)--(3,-2);
	\draw[thin, blue,fill=black!5,rounded corners=3mm] (4,0) rectangle (5,-1)
							(2,-2) rectangle (3,-3)
							(1,-3) rectangle (2,-4)
							(0,-5) rectangle (1,-6);
	\draw[blue] node at (4.5,-0.5) {\footnotesize{$x^4$}}
			    node at (2.5,-2.5) {\footnotesize{$x^2y^2$}}
			    node at (1.5,-3.5) {\footnotesize{$xy^3$}}
			    node at (0.5,-5.5) {\footnotesize{$y^5$}};
	\draw[thick,orange] (6,0)--(4,0)--(4,-2)--(2,-2)--(2,-3)--(1,-3)--(1,-5)--(0,-5)--(0,-7);
         \end{tikzpicture}
         \caption{Ferrers diagram of $P=(4,4,2,1,1)$ : \textsl{corner-monomials of $E_P$} (in blue).}\label{fig_corn-mon}
\end{center}
\end{figure}
\end{example}

For the first of the next two results see \cite[Theorems 2.9, 2.12]{Ia}, or \cite[\S 3-B,Theorem 3.12, \S 3-F]{IY}; the cellular decomposition in the second was shown by
 L. G\"{o}ttsche \cite{Got}, and the hook count for the dimension is \cite[Theorem~3.12]{IY}. Further results involving the intersections of closures of cells was shown by the last author in \cite{jy-1,jy-2}, relying in part on methods of J. Brian\c{c}on \cite{brian}. Recall that we denote by $\mathcal P(T)$ the set of all partitions of $n=|T|$ having diagonal lengths $T$.  We denote by $\delta_i(T)$ the difference $\delta_i(T)=t_{i-1}-t_{i}$, for $i\ge d$.
\begin{proposition}[The smooth projective variety $\G_T$]\label{GTdimthm}\cite[Thm. 3.13]{Ia}.
The variety $\G_T$ parametrizing all ideals $I$ of $R={\sf k}[x,y]$ satisfying $H(R/I)=T$ is a smooth irreducible projective variety, that is locally an affine space of dimension $\sum_{i\ge d} (\delta_i+1)(\delta_{i+1})$: it has a connected cover by opens in the same affine space.
\end{proposition}
\begin{theorem}[Cellular decomposition of $\G_T$]\label{GTdecompthm}\cite{jy-1,jy-2}.
The Jordan cell $\mathbb V(E_P)$ is an affine space of dimension equal to the total number of difference-one hooks in $C_P$, viewed as the Ferrers diagram of the partition $P$ (Definition \ref{EPdef}).\par
The variety $\G_T$ has a finite decomposition into affine cells,
\begin{equation}
\G_T=\bigcup_{P\in \mathcal P(T)}\mathbb V(E_P).
\end{equation}
\end{theorem}
 \subsection{Hook code of $P$.}\label{hookcodesec}
We review the hook code, using results from \cite{IY,IY2}. First, given $T$ satisfying
Equation \eqref{Teq} we define a sequence $\mathfrak B(T)$ of rectangular partitions or boxes. We let $\delta_i(T)=t_{i-1}-t_i$ for $i\ge d(T)$.  We set
\begin{align}\label{BTeq}
\mathfrak B(T)&=\left(\mathfrak B_d(T),\mathfrak B_{d+1}(T),\ldots ,\mathfrak B_{\sf j}(T)\right), \text { where }\notag\\  {\mathfrak B}_i(T)&=(\delta_{i+1})\times (1+\delta_i), \text { a rectangular box},
\end{align}
with height $\delta_{i+1}$ and base $1+\delta_i$: so ${\mathfrak B}_i(T)$ has $\delta_{i+1}$ parts, each $1+\delta_i$.
We order the monomials of degree $i$ by $x^i<x^{i-1}y<\cdots< y^i$ (lex order); certain of these monomials are hands of degree-$i$ hooks of $P$, that is end elements of rows of $C_P$ from Definition \ref{EPdef}, and we order these correspondingly.
\begin{definition}\label{hookcodedef} Suppose that the partition $P$ has diagonal lengths $T$. The (difference-one) \emph{hook code} of $P$ is the sequence
\begin{equation}\label{hookcodeeq}
\mathfrak Q(P)=\left(\mathfrak h_d(P),\ldots  ,\mathfrak h_{\sf j}(P)\right)
\end{equation}
 where $ \mathfrak h_i(P)$ is a partition that enumerates the difference-one hooks of hand-degree $i$, according to their $\delta_{i+1}$ degree-$i$ hands. That is, the $k$-th part of $\mathfrak h_i(P)$ is the number of difference-one hooks having the $k$-th possible degree-$i$ hand. 
 \end{definition}
 \begin{remark}\label{h_i_fits_in_box_i} It is not hard to see that the number of difference-one hooks per hand is in the interval $[0,\delta_{i}+1]$, and that is non-increasing: so $\mathfrak h_i(P)$ is a partition, and  $\mathfrak h_i(P)\subset \mathfrak B_i(T)$: the degree-$i$ hook partition fits into the box  $\mathfrak B_i(T)$.\par
Thus, the code is determined by arranging the difference-one hooks of $P$ first, according to their hand-degree $i$, then according to their ``hand monomial,'' determining for each degree $i\in [d,{\sf j}]$ a partition $\mathfrak h_i(P)$.
\end{remark}
 We denote by $\mathcal Q(T)$ the set of all $({\sf j}+1-d)$-tuples of partitions $(\mathfrak h_d,\ldots, \mathfrak h_{\sf j})$ satisfying, $\mathfrak h_i\subset \mathfrak {\mathfrak B}_i(T)$. 
Here $\mathcal Q(T)$ is a lattice under the product structure given by inclusion for each component $\mathfrak h_i$: that is, $\mathcal Q\le \mathcal Q^\prime$ if each $\mathfrak h_i\subset  \mathfrak h_i^\prime$, in the sense that the Ferrers diagram for $\mathfrak h_i$ fits inside that of $\mathfrak h_i^\prime$.\par
When $T$ is a sequence satisfying Equation \eqref{Teq} - of socle degree $\sf j$ - we write $T^\vee$ for the partition obtained from $T$ as the conjugate (switch rows and columns in the Ferrers graph) of the partition having ${\sf j}+1$ parts $\{t_0,t_1,\ldots, t_{\sf j}\}$. Thus, $T^\vee$ gives the lengths of the rows of the bar-graph of $T$. 
The partition $T^\vee$ always has the maximum hook code $\mathfrak B(T)$ possible for a partition of diagonal lengths $T$.\footnote{In the study of Lefschetz properties and Jordan type of Artinian algebras, multiplication by $x$ having Jordan type $T^\vee$ on $A$ corresponds to the the strong Lefschetz property of the pair $(A,x)$ \cite[\S~2.3]{IMM}.}
\begin{example}[Hook code for $P=(6,3,3,3)$]\label{6333ex}
Let $T=(1,2,3,4,3,2,0)$  where $d=4$; we have  $\delta_4=4-3=1,\delta_5=3-2=1,\delta_6=2-0=2.$ Then $\mathfrak B(T)=\left(\mathfrak B_4,\mathfrak B_5\right)=\left((1\times 2)_4,(2\times 2)_5\right)$. The partition $T^\vee=(6,5,3,1)$ has the maximum hook code $\mathfrak Q(T^\vee)=\mathfrak B(T)$.
But $P=(6,3^3)$ has the hook code $\mathfrak Q(P)=\left(1_4,(2,1)_5\right)$: the degree four hand monomial is $y^2x^2$ with a single difference-one hook, with corner $y^2$; the degree-5 hand monomials are $x^5$ with two hooks with corners $x,x^4,$ and $y^2x^3$ with one hook, corner $y^3x$. See Figure \ref{6333fig} where we visualize the hooks by showing their corners, blue for degree 4 and red for degree~5. 
\end{example}
\begin{figure}[!h]
 \begin{center}
$ \begin{array}{|c|c|c|c|c|c|}
\hline
&\cellcolor{red}\bullet&&\phantom{ab}&\cellcolor{red}\bullet&h_{5}\\
\hline
&&\\
\cline{1-3}
\cellcolor{blue}\bullet&&h_4\\
\cline{1-3}
&\cellcolor{red}\bullet&h_5\\
\cline{1-3}
\end{array}$
 \caption{Hook code for $P=(6,3,3,3): \mathfrak Q(P)=\left(1_4,(2,1)_5\right).$ }\label{6333fig}
 \end{center}
 \end{figure}
The following is stated as part of \cite[Theorem 3.27]{IY}, and shown in \cite[Theorem~1.17]{IY2}. Recall that $\mathcal P(T)$ is the set of partitions having diagonal lengths $T$. We denote by $\mathfrak q: \mathcal P(T)\to \mathcal Q(T)$ the hook code map taking $P$ to $\mathfrak Q(P)$. For a partition $\mathfrak h_i\subset \mathfrak B_i$ we denote by $\mathfrak h_i^c$ the complement of $\mathfrak h_i$ in ${\mathfrak B}_i$. For an element $\mathfrak h=(\mathfrak h_d,\ldots ,\mathfrak h_{\sf j})\in \mathcal Q(T)$, we denote by $\mathfrak h^c=(\mathfrak h_d^c,\ldots ,\mathfrak h_{\sf j}^c)$ the complement in $\mathcal B(T)$. Recall that, given $P\in \mathcal P(T)$ we denote by $P^\vee$ the conjugate partition (switch rows and columns in the Ferrers graph of $P$); evidently $P^\vee\in \mathcal P(T)$.
\begin{theorem}\label{PTtoQTthm} Let $T$ satisfy Equation \eqref{Teq}, and let $P\in \mathcal P(T)$. Then 
the map $\mathfrak q: \mathcal P(T)\to \mathcal Q(T)$ is an isomorphism of sets satisfying $\mathfrak q(P^\vee)=(\mathfrak q(P))^c$. 
\end{theorem}
We endow $\mathcal P(T)$ with the structure of a lattice via the isomorphism $\mathfrak q$ to $\mathcal Q(T)$
(see Definition \ref{hookcodedef}). \footnote{There is an alternative poset structure  $\mathcal P_{alt}(T)$ on $\mathcal P(T)$, related to the sequences of degree-$i$ monomials in $C_P$. The inverse $\mathcal Q(T)\cong P(T)\to \mathcal P_{alt}(T)$ is an inclusion of posets, not an isomorphism of lattices as stated incorrectly in \cite[Theorem 3.27]{IY}. See the discussion in \cite{IY2}.}\par
The second and last author showed that the dimension of the cell $\mathbb V(E)$ is the total number of difference-one hooks in the partition $P_E$ determined by $E$ (Theorem \ref{GTdecompthm}): this is just the height of $\mathfrak Q(P)$ in the lattice $\mathcal Q(T)$.  It follows from the A. Bialynicki-Birula result \cite{Bia}  that over the complexes, the Betti numbers of $\G_T$ may be deduced from the cellular decomposition \cite[Theorem~3.28, Theorem~3.29]{IY}. \par
\subsection{The cell $\mathbb V(E_P)$, and its component decomposition.}\label{prodsec}
Throughout this section $T=(1, 2, \dots , d, t_d, \dots , t_{\sf j}, 0)$ will be a Hilbert function satisfying Equation \eqref{Teq} and $P$ will be a partition of diagonal lengths $T$. 
We denote by $E_P$ the monomial ideal associated to $P$ and let ${\mathbb V}(E_P)$ be the Jordan cell of $\G_T$ associated to $P$. Our purpose here is to give a description of ${\mathbb V}(E_P)$ as a product of cells of  ``small'' Grassmannians
(Theorem \ref{projectionthm}).
\subsubsection{The component Hilbert functions of $T$.}
Let $P$ be a partition with Hilbert function $T=\left(1,\dots ,d,t_d,\dots ,t_{\sf j},0\right)$ and difference-one hook code
 $\mathfrak{Q}(P)=\big( \mathfrak{h}_d,\mathfrak{h}_{d+1},\dots ,\mathfrak{h}_{\sf j}\big)$ (Definition~\ref{hookcodedef}). For $i=d, \dots, {\sf j}$ we let $T_i$ be the single-block Hilbert function 
  \begin{align}\label{HFdecomposition}
T_i&=(1, \dots, d_i-1,d_i,t_{d_i},0),
\end{align}
 where $d_i=\delta_i(T)+\delta_{i+1}(T)$,  and $t_{d_i}=\delta_{i+1}$.  
 We  set $t_{d-1}:=d$ and $t_{{\sf j} +1}:=0$.
(There is a shift in degrees, $T_i$ parametrizes ideals of order (initial degree) $d_i=t_{i-1}-t_{i+1}=\delta_i(T)+\delta_{i+1}(T)$).\par\vskip 0.2cm
\subsubsection{Definition of the component partitions of $P$.}
In the next pages we give a construction of the component partitions $P_i$ corresponding to the single-block Hilbert function $T_i$, from the partition $P$ having diagonal lengths $T$ (Definition \ref{Pidef1}).\par

We will first define the $i$-th block of $P$, denoted by $P_i$ (Definition \ref{Pidef1}). We will show that it is the partition with diagonal lengths $T_i$ and hook code the $i$-th component
$\mathfrak{Q}(P_i)=\big(\mathfrak{h}_i\big)=\big(\mathfrak{h}_i(P)\big)$ of the hook code
of $P$ (Proposition~\ref{hookP(i)lem}). This depends on defining border (foot) and hand monomials, respectively, of $P$ in degree $i$, giving vector spaces $V_{i1}, V_{i2}$, respectively. We will show that the partition $P_i$ may be derived simply from the monomials $V_{i1}\cup V_{i2}$ (Proposition~\ref{rm_rows_cols_P_get_P_i}).  A reader on a first look may just use this  Proposition as a definition of $P_i$.

Recall that for a Hilbert function $T=(1, 2, \dots , d, t_d, \dots , t_{\sf j}, 0)$, we set $\delta_i=t_{i-1}-t_i$ ($d\leq i \leq {\sf j}+1$). \\
Given $P=(p_1, \dots , p_s)$ a partition of diagonal lengths $T$, for each $i$ ($d\leq i\leq {\sf j}$) we construct a vector space $V_i$ of 
dimension $n_i=\delta_i+\delta_{i+1}+1$ such that to each element $I$ of the cell ${\mathbb V}(E_P)$ we can associate a subvector space $I_{V_i}$ of $V_i$ with $\displaystyle{\dim_k\left(I_{V_i}\right)=\delta_i+1}$. So the vector space $I_{V_i}$ is an element of 
$\Grass(\delta_i+1, V_i)$ and belongs to a cell described by the partition $\mathfrak{h}_i(P)$, the degree-$i$ block of the difference-one hook code of $P$. 

\medskip
\noindent
We first define (as in \cite{IY2} but in a more strict way) the
\textsl{horizontal-border monomials} and \textsl{vertical-border monomials} of $E_P$.

\medskip
\noindent
\begin{definition}\label{defi_horiz_vertic_border_monom}
Let $P$ be a partition of diagonal lengths $T$. Denote by $E_P$ the monomial ideal of ${\sf k}[x,y]$ associated to $P$. 
\begin{enumerate}[(i).]\item We say a monomial $x^ay^b\in E_P$ is a \textsl{horizontal-border monomial} of $E_P$ if $(b>0$ and $x^ay^{b-1}\notin E_P)$. 
\item We say a monomial $x^ay^b\in E_P$ is a \textsl{vertical-border monomial} of $E_P$ if  $a>0$ and $x^{a-1}y^{b}\notin E_P$. 
\end{enumerate}
Denote by $A(P)_i$ the set of degree-$i$ \textsl{horizontal-border monomials} of $E_P$,  
$B(P)_i$ the set of degree-$i$ \textsl{vertical-border monomials} of $E_P$ and $(C_P)_i$ the set of degree-$i$ monomials that are not in $E_P$. 
\end{definition}
\begin{example}\label{example-border-monomials}
Let $P=(9,7^2,4^2,2,1^2)\in\mathcal{P}(T)$ where $T=(1,2,3,4,5,6,7,5,2,0)$, see Figure \ref{972ex_fig_a}. The sets of degree-$7$ and $8$ horizontal-border monomials of $E_P$ are given by $A(P)_7=\{x^4y^3,x^2y^5,xy^6\}$ and  $A(P)_8=\{x^7y,x^5y^3,x^3y^5,y^8\}$. Also we observe that the vertical-border monomials of degree $7$, $B(P)_7$, is equal to $A(P)_7$. In degree $8$ the only vertical-border monomial in $E_P$ is  $x^7y$.
\end{example}

\smallskip
\noindent
\begin{claim}\label{borderclaim}
\begin{enumerate}[(i).]
\item $
\displaystyle{
\left|A(P)_i\right|=
\left\{
\begin{array}{l}
t_{i-1}-t_i+1=\delta_i+1, \ \mbox{if $x^i\in (C_P)_i$} 
\\
t_{i-1}-t_i=\delta_i,  \ \mbox{if $x^i\notin (C_P)_i$} 
\end{array}
\right.
}
$
\item $
\displaystyle{
\left|B(P)_{i+1}\right|=
\left\{
\begin{array}{l}
t_{i}-t_{i+1}+1=\delta_{i+1}+1, \ \mbox{if $y^{i+1}\in (C_P)_{i+1}$} 
\\
t_{i}-t_{i+1}=\delta_{i+1},  \ \mbox{if $y^{i+1}\notin (C_P)_{i+1}$} 
\end{array}
\right.
}
$
\end{enumerate}
\end{claim}
\smallskip
\noindent
\begin{proof}[Proof of claim] 
(Note: these formulas have been established in \cite{IY2} to define \textsl{difference-$a$ hook partitions}.) \\
One can consider the following maps 
\begin{center}
\begin{tabular}{lr}
$\displaystyle{
\begin{array}{llcc}
\varphi_i : & (C_P)_{i-1} & \longrightarrow & (C_P)_i\cup A(P)_i 
\\
 & M & \mapsto & yM
\end{array}
} $; 
&
$\displaystyle{
\begin{array}{llcc}
\psi_i : & (C_P)_{i} & \longrightarrow & (C_P)_{i+1}\cup B(P)_{i+1} 
\\
 & M & \mapsto & xM
\end{array}
} .$ 
\end{tabular}
\end{center}
The maps $\varphi_i$ and $\psi_i$ are injective. Also note that if $x^i\in (C_P)_i$, than $x^i$ is the only element of 
$(C_P)_i\cup A(P)_i $ that is not in the image of $\varphi_i$, so we have 
$\left| (C_P)_{i-1}\right|=\left|(C_P)_i \right|+\left|A(P)_i\right|-1$. 
If $x^i\notin (C_P)_i$ than $\varphi_i$ is a bijection, thus $\left| (C_P)_{i-1}\right|=\left|(C_P)_i \right|+\left|A(P)_i\right|$. 
The formula for $\left|A(P)_i\right|$ follows  from the fact that for any integer $l$, $\left|(C_P)_l \right|=t_l$.
Using the same arguments one can verify the formula for $\left|B(P)_{i+1}\right|$.
\end{proof}
\begin{remark}
\begin{enumerate}[(i).]
\item Each monomial $x^ay^b$ in $A(P)_i$ is just below a degree-$(i-1)$ foot monomial ($x^ay^{b-1}$) of $P$, thus $\left|A(P)_i\right|$ counts the number of degree-$(i-1)$ foot monomials in the Ferrers diagram of $P$.
\item The elements of $B(P)_{i+1}$ are each just right to a degree-$i$ hand monomial, so $\left|B(P)_{i+1}\right|$ counts the number of degree-$i$ hand monomials in the Ferrers diagram of $P$. In Definition \ref{Pidef1} we will consider the first  (numbering from top to bottom-- lex order) $\delta_{i+1}$ degree-$i$ hand monomials of $P$. 
\item Recall that $P=(p_1, \dots , p_s)$, with $p_1\geq p_2 \geq \cdots \geq p_s>0$. \\ 
If $x^i\notin (C_P)_i$, then $p_1\leq i$.
\end{enumerate}
\end{remark}\vskip 0.2cm
With the following key definition we are able to construct the component $P_i$ from $P$, in a fashion that is convenient for our later algebraic proofs, that involve an induction from $I:x$ to $I$.
For illustration see Example \ref{comb_PideP_rmk} and Figure \ref{tobedetermined}, and as well Example \ref{1stgensEx} and Figure~\ref{determined}.\par
\begin{definition}\label{Pidef1}[Definition of $P_i$]
For any positive integer $n\in{\mathbb N}$, denote by $\Mon(R_n)$ the set of degree $n$ monomials of $R={\sf k}[x,y]$
and recall the lex order on $\Mon(R_n)$: $x^n<x^{n-1}y<\cdots <xy^{n-1}<y^n$. 
 Let $P\in \mathcal P(T)$. For every $i\in \left[d,\sf{j} \right]$ we define the set $V_{i1}$ as the following 
$$V_{i1}=\left\{\begin{array}{ll} A(P)_i,&\mbox{ if }x^i\in (C_P)_i \\ A(P)_i\cup\left\{x^ay^b\right\},&\mbox{ if }x^i\notin (C_P)_i \end{array}\right.,$$
where $x^ay^b$ is the last (lex order) degree-$i$ \textsl{vertical-border monomial} above  $M_{i1}$ that is the first (lex order) monomial in $A(P)_i$. 
We note that $\dim V_{i1}=\delta_i+1$. \par
We now define the set $V_{i2}$ to be  the first (lex order) $\delta_{i+1}$ hand monomials in  $(C_P)_i$. We denote by $V_i$ the vector space spanned by $V_{i1}\cup V_{i2}$. 

\medskip
\noindent
By definition, $V_i$ has dimension $d_i+1$ where $d_i=\delta_i+\delta_{i+1}$. 
The set $V_{i1}\cup V_{i2}$ is lex ordered and we can consider the one-to-one correspondence 
\begin{equation}\label{monomeq}
s_i : V_{i1}\cup V_{i2} \longrightarrow\Mon(R_{d_i})
\end{equation}
 that respects the lex ordering (the $k$-th element of $V_{i1}\cup V_{i2}$ is associated to the $k$-th element of $\Mon(R_{d_i})$). \par
The vector space $s_i(V_{i1})$ has dimension $\delta_i+1$, so $\langle s_i(V_{i_1}) \rangle+\mathfrak m^{d_i+1}$ is the monomial ideal $(E_{P_i})$ for a \emph{unique partition $P_i$ } of diagonal lengths the single-block Hilbert function 
$T_i=(1, \dots, d_i-1,d_i,t_{d_i},0)$ where $d_i=\delta_i+\delta_{i+1}$ and $t_{d_i}=\delta_{i+1}$.   
\end{definition}
\begin{example}\label{972ex}
Let $P=(9,7^2,4^2,2,1^2)$ be the partition considered in Example \ref{example-border-monomials} with diagonal lengths $T=(1,2,3,4,5,6,7,5,2,0)$. (See Figure \ref{972ex_fig_a}). The difference-one hook code of $P$ is ${\mathfrak Q}(P)=\big((3,2,0)_7,(4,3)_8\big)$. \\We have $( \delta_7,\delta_8,\delta_9)=(2,3,2)$, $T_7=(1,2,3,4,5,3)$ and $T_8=(1,2,3,4,5,2)$.\\
\begin{figure}[!h]
\begin{center}
	\begin{tikzpicture}[scale=.8]
	\draw[dashed,very thin] 	(0,0)--(9,0)
						(0,-1)--(7,-1)
						(0,-2)--(7,-2)
						(0,-3)--(4,-3)
						(0,-4)--(4,-4)
						(0,-5)--(2,-5)
						(0,-6)--(1,-6)
						(0,-7)--(1,-7)
						(0,0)--(0,-8)
						(1,0)--(1,-6)
						(2,0)--(2,-5)
						(3,0)--(3,-5)
						(4,0)--(4,-3)
						(5,0)--(5,-3)
						(6,0)--(6,-3)
						(7,0)--(7,-1)
						(8,0)--(8,-1);
	\draw[thin, dashed,red,fill=yellow!10] (8,0) rectangle (9,-1)
							(6,-1) rectangle (7,-2)
							(6,-2) rectangle (7,-3)
							(3,-4) rectangle (4,-5)
							(0,-7) rectangle (1,-8);
	\draw[thin, blue,fill=black!5,rounded corners=3mm] (7,-1) rectangle (8,-2)
						     (5,-3) rectangle (6,-4)
						     (4,-3) rectangle (5,-4)
						     (3,-5) rectangle (4,-6)
						     (2,-5) rectangle (3,-6)
						     (1,-6) rectangle (2,-7)
						     (0,-8) rectangle (1,-9);
	\draw[red] node at (8.5,-0.5) {\footnotesize{$x^8$}}
			  node at (6.5,-1.5) {\footnotesize{$x^6y$}}
			  node at (6.5,-2.5) {\footnotesize{$x^6y^2$}}
			  node at (3.5,-4.5) {\footnotesize{$x^3y^4$}}
			  node at (0.5,-7.5) {\footnotesize{$y^7$}};
	\draw[blue] node at (7.5,-1.5) {\footnotesize{$x^7y$}}
			    node at (5.5,-3.5) {\footnotesize{$x^5y^3$}}
			    node at (4.5,-3.5) {\footnotesize{$x^4y^3$}}
			    node at (3.5,-5.5) {\footnotesize{$x^3y^5$}}
			    node at (2.5,-5.5) {\footnotesize{$x^2y^5$}}
			    node at (1.5,-6.5) {\footnotesize{$xy^6$}}
			    node at (0.5,-8.5) {\footnotesize{$y^8$}};
	\draw[very thick] (9,0)--(9,-1)--(7,-1)--(7,-3)--(4,-3)--(4,-5)--(2,-5)--(2,-6)--(1,-6)--(1,-8)--(0,-8);
         \end{tikzpicture}
         \caption{Ferrers diagram of $P=(9,7^2,4^2,2,1^2)$: \textsl{border monomials} are marked in blue and 
         \textsl{hand monomials} are marked in red (Example \ref{972ex}).}\label{972ex_fig_a}
\end{center}
\end{figure}
\begin{enumerate}[(a)]
\item In degree $7$ we have \textcolor{blue}{$V_{7,1}=A(P)_7=\left\{x^4y^3,x^2y^5,xy^6\right\}$}, 
\textcolor{red}{$V_{7,2}=\left\{x^6y,x^3y^4,y^7\right\}$}, so 
$V_7$ has basis $(x^6y,x^4y^3,x^3y^4,x^2y^5,xy^6,y^7)$ in the lex order. 
We consider the bijection $s_7: V_{7,1}\cup V_{7,2} \longrightarrow \Mon(R_{5})$ given by: 
$$
\begin{array}{ccc}
\textcolor{red}{s_7(x^6y)=x^5}, &\textcolor{blue}{s_7(x^4y^3)=x^4y} ,  &  \textcolor{red}{s_7(x^3y^4)=x^3y^2} \\
\textcolor{blue}{s_7(x^2y^5)=x^2y^3}, & \textcolor{blue}{s_7(xy^6)=xy^4}, & \textcolor{red}{s_7(y^7)=y^5}.
\end{array}
$$
We then get a one block partition $P_7$ (see Figure \ref{972ex_fig_b})
\begin{figure}[!h]
\begin{center}
	\begin{tikzpicture}[scale=.8]
	\draw[dashed,very thin] 	(0,0)--(6,0)
						(0,-1)--(4,-1)
						(0,-2)--(4,-2)
						(0,-3)--(2,-3)
						(0,-4)--(1,-4)
						(0,-5)--(1,-5)
						(0,0)--(0,-6)
						(1,0)--(1,-4)
						(2,0)--(2,-3)
						(3,0)--(3,-3)
						(4,0)--(4,-1)
						(5,0)--(5,-1);
	\draw[thin, dashed,red,fill=yellow!10] (5,0) rectangle (6,-1)
								     (3,-2) rectangle (4,-3)
								     (0,-5) rectangle (1,-6);
	\draw[thin, blue,fill=black!5,rounded corners=3mm] (4,-1) rectangle (5,-2)
											   (2,-3) rectangle (3,-4)
											   (1,-4) rectangle (2,-5);
	\draw[red] node at (5.5,-0.5) {\footnotesize{$x^5$}}
			  node at (3.5,-2.5) {\footnotesize{$x^3y^2$}}
			  node at (0.5,-5.5) {\footnotesize{$xy^4$}};
	\draw[blue] node at (4.5,-1.5) {\footnotesize{$x^4y$}}
			    node at (2.5,-3.5) {\footnotesize{$x^2y^3$}}
			    node at (1.5,-4.5) {\footnotesize{$xy^4$}};
	\draw[very thick] (6,0)--(6,-1)--(4,-1)--(4,-3)--(2,-3)--(2,-4)--(1,-4)--(1,-6)--(0,-6);
         \end{tikzpicture}
         \caption{Ferrers diagram of $P_7=(6,4^2,2,1^2)$ : \textsl{horizontal-border monomials} are marked in blue and 
         \textsl{hand monomials} are marked in red (Example \ref{972ex}a).}\label{972ex_fig_b}
\end{center}
\end{figure}

\item In degree $8$, \textcolor{blue}{$V_{8,1}=A(P)_8=\left\{x^7y,x^5y^3,x^3y^5,y^8\right\}$}, 
\textcolor{red}{$V_{8,2}=\left\{x^8,x^6y^2\right\}$}. 
$V_8$ has basis $(x^8,x^7y,x^6y^2,x^5y^3,x^3y^5,y^8)$ in the lex order. \\
The bijection $s_8: V_{8,1}\cup V_{8,2} \longrightarrow\Mon(R_{5})$ given by: 
$$
\begin{array}{ccc}
\textcolor{red}{s_8(x^8)=x^5}, & \textcolor{blue}{s_8(x^7y)=x^4y},  &  \textcolor{red}{s_8(x^6y^2)=x^3y^2}\\
\textcolor{blue}{s_8(x^5y^3)=x^2y^3}, & \textcolor{blue}{s_8(x^3y^5)=xy^4}, & \textcolor{blue}{s_8(y^8)=y^5}.
\end{array}
$$
This gives us a one block partition $P_8$ (see Figure \ref{972ex_fig_c})
\begin{figure}[!h]
\begin{center}
	\begin{tikzpicture}[scale=.8]
	\draw[dashed,very thin] 	(0,0)--(6,0)
						(0,-1)--(4,-1)
						(0,-2)--(4,-2)
						(0,-3)--(2,-3)
						(0,-4)--(1,-4)
						(0,0)--(0,-5)
						(1,0)--(1,-4)
						(2,0)--(2,-3)
						(3,0)--(3,-3)
						(4,0)--(4,-1)
						(5,0)--(5,-1);
	\draw[thin, dashed,red,fill=yellow!10] (5,0) rectangle (6,-1)
								     (3,-2) rectangle (4,-3);
	\draw[thin, blue,fill=black!5,rounded corners=3mm] (4,-1) rectangle (5,-2)
											   (2,-3) rectangle (3,-4)
											   (1,-4) rectangle (2,-5)
											   (0,-5) rectangle (1,-6);
	\draw[red] node at (5.5,-0.5) {\footnotesize{$x^5$}}
			  node at (3.5,-2.5) {\footnotesize{$x^3y^2$}};
	\draw[blue] node at (4.5,-1.5) {\footnotesize{$x^4y$}}
			    node at (2.5,-3.5) {\footnotesize{$x^2y^3$}}
			    node at (1.5,-4.5) {\footnotesize{$xy^4$}}
			    node at (0.5,-5.5) {\footnotesize{$y^5$}};									   
	\draw[very thick] (6,0)--(6,-1)--(4,-1)--(4,-3)--(2,-3)--(2,-4)--(1,-4)--(1,-5)--(0,-5);
         \end{tikzpicture}
         \caption{Ferrers diagram $P_8=(6,4^2,2,1)$ : \textsl{horizontal-border monomials} are marked in blue and 
         \textsl{hand monomials} are marked in red (Example \ref{972ex}b).}\label{972ex_fig_c}
\end{center}
\end{figure}
\end{enumerate}
\end{example}
\begin{proposition}\label{rm_rows_cols_P_get_P_i}
Using the notation of Definition \ref{Pidef1}, let 
\begin{equation}\label{Veq}
\displaystyle{V_{i1}\cup V_{i2}=\left\{x^{\alpha_0}y^{\beta_0}, \ldots,  x^{\alpha_{d_i}}y^{\beta_{d_i}}\right\}},
\end{equation}
where $\alpha_0<\alpha_1<\cdots <\alpha_{d_i}$  (so $\beta_0>\beta_1>\cdots >\beta_{d_i}$). 
Let $P^{\prime}$ be the partition obtained from $P$ by removing any column of $P$ whose index does not belong to the set $\left\{\alpha_0,\alpha_1,\ldots ,\alpha_{d_i} \right\}$ and any row of $P$ whose index does not  belong to the set $\left\{\beta_0,\beta_1,\ldots ,\beta_{d_i} \right\}$.  Then $P^{\prime}=P_i$.
\end{proposition}
\begin{proof}
In constructing the $i$-th component of ${\mathfrak Q}(P)$ (in the difference-one hook code of $P$), we only need the elements of $V_{i2}$ (degree $i$ hand monomials) and the elements of $V_{i1}$ (related to degree $i$ \textsl{horizontal-border monomials}).The purpose of the bijection $s_i : V_{i1}\cup V_{i2} \longrightarrow \Mon (R_{d_i})$ is to let us focus on these monomials. So by definition of the bijection $s_i$, the partition $P_i$ is obtained from $P$ by ignoring any column of $P$ whose index does not belong to the set $\left\{\alpha_0,\alpha_1,\ldots ,\alpha_{d_i} \right\}$ and any row of $P$ whose index does not  belong to the set $\left\{\beta_0,\beta_1,\ldots ,\beta_{d_i} \right\}$. Deleting these unnecessary rows and columns will result in showing only the relevant degree $i$ hands and degree $i-1$ feet of $P$.  
\end{proof}

In the following example we visualize the set $V_i=V_{i1}\cup V_{i2}$, from Definition \ref{Pidef1} by looking at the Ferrers diagram of a partition $P$. 

\begin{example}\label{comb_PideP_rmk}
Consider the two-block partition $P=(6^2,3,2^2)$ with diagonal lengths $T=\left(1,2,3,4,5,3,1\right)$, see Figure  \ref{tobedetermined}. 
Consider the diagonal corresponding to degree-$5$ monomials of ${\sf k}[x,y]$, see the grey bubbles in Figure \ref{tobedetermined}. Then the set of degree-$5$ horizontal-border monomials of $P$, $A(P)_5$ can correspond to the bubbles outside of the Ferrers diagram that are right below the horizontal edges of $P$. So $A(P)_5=\{x^3y^2,x^2y^3,y^5\}$, see the blue monomials in the left side of Figure \ref{tobedetermined}. Since the largest part of $P$ is greater than $5$ (i.e., $x^5\in (C_P)_5$), then $V_{51}$ is the same as $A(P)_5$. To obtain $A(P)_6$, notice that the largest part of $P$ is at most $6$ (i.e., $x^6\not \in (C_P)_6$) then $V_{61}$ also includes the first degree-$6$ vertical-border monomial of $P$ that is above all monomials in $A(P)_6$. This monomial corresponds to a bubble outside of the Ferrers diagram that is immediately to the right of a horizontal edge of $P$, see the red monomial on the the right of Figure \ref{tobedetermined}. So $A(P)_6=\{x^4y^2,xy^5\}$ and $V_{61}=\{x^6,x^4y^2,xy^5\}$. Finally, monomials in $V_{i2}$ consists of the first $\delta_{i+1}$ hand monomials of $P$. These correspond to bubbles inside the Ferrers diagram that are at the end of a row of $P$. So $V_{52}=\{x^5,x^4y,xy^4\}$ and $V_{62}=\{x^5y,x^3y^3,x^2y^4,y^6\}$, see the black monomials in Figure \ref{tobedetermined}.

To visualize Proposition \ref{rm_rows_cols_P_get_P_i}, we fill out the degree-$i$ grey bubbles that correspond to $V_i$ by their monomials. We then remove all rows and columns of $P$ that do not include a filled degree-$i$ bubble, see Figure \ref{tobedetermined}.
\begin{figure}
	\begin{tikzpicture}[scale=.8]
	\draw[very thin] 	(1.5,0)--(7.5,0)
						(1.5,-1)--(7.5,-1)
						(1.5,-2)--(4.5,-2)
						(1.5,-3)--(3.5,-3)
						(1.5,-4)--(3.5,-4)
						(2.5,0)--(2.5,-5)
						(3.5,0)--(3.5,-5)
						(4.5,0)--(4.5,-3)
						(5.5,0)--(5.5,-2)
						(6.5,0)--(6.5,-2)
						(7.5,0)--(7.5,-2);	

	\draw[thin,fill=black!5,rounded corners=3mm] 
	(1.5,-5) rectangle (2.5,-6)
	(2.5,-4) rectangle (3.5,-5)
	(3.5,-3) rectangle (4.5,-4)
	(4.5,-2) rectangle (5.5,-3)
	(5.5,-1) rectangle (6.5,-2)
	(6.5,0) rectangle (7.5,-1);

\draw[thick, dashed, red]	
	                    (1,-1.5)--(8,-1.5)
	                    (6,0.5)--(6, -2.5);
	
	\draw[black] 
	            node at (7,-0.5) {\footnotesize{$x^5$}}
	            node at (3,-4.5) {\footnotesize{$xy^4$}};	
	\draw[blue] 
			    node at (5,-2.5) {\footnotesize{$x^3y^2$}}
			    node at (4,-3.5) {\footnotesize{$x^2y^3$}}
			    node at (2,-5.5) {\footnotesize{$y^5$}};	  
	\draw[very thick] (7.5,0)--(7.5,-2)--(4.5,-2)--(4.5,-3)--(3.5,-3)--(3.5,-5)--(1.5,-5)--(1.5,0)--(7.5,0);
	
	\draw[very thick, red, ->]
		(5,-7)--(5,-10);
	\draw[red] 
	    node at (4.5,-7.5) {$P$}
	    node at (4.5,-8.25) {to}
	    node at (4.5,-9.25) {$P_5$};

	\draw[very thin] 	(2,-11)--(7,-11)
						(2,-12)--(7,-12)
						(2,-13)--(5,-13)
						(2,-14)--(4,-14)
						(2,-15)--(4,-15)
						(3,-11)--(3,-15)
						(4,-11)--(4,-15)
						(5,-11)--(5,-13)
						(6,-11)--(6,-12);
							
	\draw[very thick] 
	(2,-11)--(7,-11)--(7,-12)--(5,-12)--(5,-13)--(4,-13)--(4, -15)--(2,-15)--(2,-11);


	\draw[very thin] 	(11.5,0)--(17.5,0)
						(11.5,-1)--(17.5,-1)
						(11.5,-2)--(14.5,-2)
						(11.5,-3)--(13.5,-3)
						(11.5,-4)--(13.5,-4)
						(12.5,0)--(12.5,-5)
						(13.5,0)--(13.5,-5)
						(14.5,0)--(14.5,-3)
						(15.5,0)--(15.5,-2)
						(16.5,0)--(16.5,-2)
						(17.5,0)--(17.5,-2);	

	\draw[thin,fill=black!5,rounded corners=3mm] 
	(11.5,-6) rectangle (12.5,-7)
	(12.5,-5) rectangle (13.5,-6)
	(13.5,-4) rectangle (14.5,-5)
	(14.5,-3) rectangle (15.5,-4)
	(15.5,-2) rectangle (16.5,-3)
	(16.5,-1) rectangle (17.5,-2)
	(17.5,0) rectangle (18.5,-1);

\draw[thick, dashed, red]	
	                    (12,0.5)--(12,-7.5)
	                    (14,0.5)--(14,-5.5)
	                    (15,0.5)--(15,-4.5)
	                    (11,-3.5)--(16, -3.5)
	                    (11,-4.5)--(15, -4.5);
	
	\draw[red] 
	            node at (18,-0.5) {\footnotesize{$x^6$}};
	\draw[black]           
	            node at (17,-1.5) {\footnotesize{$x^5y$}};	
	\draw[blue] 
			    node at (16,-2.5) {\footnotesize{$x^4y^2$}}
			    node at (13,-5.5) {\footnotesize{$xy^5$}};	  
	\draw[very thick] (17.5,0)--(17.5,-2)--(14.5,-2)--(14.5,-3)--(13.5,-3)--(13.5,-5)--(11.5,-5)--(11.5,0)--(17.5,0);
	
	\draw[very thick, red, ->]
		(15,-7)--(15,-10);
	\draw[red] 
	    node at (15.5,-7.5) {$P$}
	    node at (15.5,-8.25) {to}
	    node at (15.5,-9.25) {$P_6$};
	    
	
	\draw[very thin] 	(13.5,-11)--(16.5,-11)
						(13.5,-12)--(16.5,-12)
						(13.5,-13)--(14.5,-13)
						(14.5,-11)--(14.5,-14)
						(15.5,-11)--(15.5,-13)
						(16.5,-11)--(16.5,-13);
							
	\draw[very thick] 
	(13.5,-11)--(16.5,-11)--(16.5,-13)--(14.5,-13)--(14.5,-14)--(13.5,-14)--(13.5,-11);
	
         \end{tikzpicture}
         \caption{Illustration of Example \ref{comb_PideP_rmk}. On the left, the monomials in blue represent $V_{5,1}=A(P)_5$. On the right, the monomials in blue represent $A(P)_6$ and the monomial in red represents the additional vertical-border monomial in $V_{6,1}$. }\label{tobedetermined}
\end{figure}
\end{example}

\begin{remark}\label{empty_hook_code_block_partition}
For $i\in \left[d,\sf{j} \right]$, it may happen that $t_i=t_{i+1}$, so $\delta_{i+1}=0$. In that case we have: 
\begin{enumerate}[(i).]
\item The rectangular box ${\mathfrak B}_i(T)=(\delta_{i+1})\times (1+\delta_i)$ of Equation \ref{BTeq} is empty and so the degree-$i$ partition 
$ \mathfrak h_i(P)$ in Equation \ref{hookcodeeq} is empty. 
\item $V_{i2}$ is empty and so $s_i(V_{i1})=\Mon(R_{d_i})$. 
\item The partition $P_i$  associated to the monomial ideal $\langle s_i(V_{i_1}) \rangle+\mathfrak m^{d_i+1}$ is just the basic triangle 
$\Delta_{d_i}=\Delta_{\delta_i}=(\delta_i, \delta_i-1, \ldots , 1)$. Of course, if $\delta_i=0$, then $\Delta_{d_i}=\emptyset$ and 
$\langle s_i(V_{i_1}) \rangle+\mathfrak m^{d_i+1}=R$.
\end{enumerate}
\end{remark}
\medskip

The following Proposition follows directly from Definition \ref{hookcodedef} and 
Definition~\ref{Pidef1}.
\begin{proposition}\label{hookP(i)lem}
The difference-one hook code of $P_i$ is exactly that of the $i$-th component of ${\mathfrak Q}(P)$ (in the difference-one hook code of $P$).
\end{proposition}\noindent
Note, however, the shift in degree:  the difference-one hook code of $P_i$ occurs in the degree $d_i=t_{i-1}-t_{i+1}=\delta_i(T)+\delta_{i+1}(T)$.

\begin{example}\label{1stgensEx}
Consider the three-block partition $P=(15, 12^4, 11, 7, 6^2, 5, 3^4)$ with diagonal lengths $T=\left(1,2,\dots ,13,10_{13},6_{14},3_{15},0\right)$ and hook code $$\mathfrak{Q}(P)=\left((3,1^2, 0)_{13},(5,4,1)_{14}, (2^2,1)_{15}\right).$$

In Figure \ref{determined}, we illustrate the process of decomposing $P$ into its single-block components, $P_{13}=(7^2, 5, 4^2, 3, 1^2)$, $P_{14}=(8, 6^2,4, 3, 2^2)$, and $P_{15}=(6, 5^2, 4, 2^2)$ of diagonal lengths, respectively, $T_{13}=(1, \dots, 7, 4,0)$, 
$T_{14}=(1, \dots, 7, 3,0)$, and $T_{15}=(1, \dots, 6, 3,0)$.

In each part of Figure \ref{determined}, the bubbles correspond to degree-$i$ monomials. The hand monomials of $V_{i2}$ are black and the horizontal-border monomials, which all belong to $V_{i1}$ are blue. For $i=15$, since $x^{15}\not \in (C_P)_{15}$, in addition of the horizontal-border monomials, the set $V_{15,2}$ also includes a vertical-border monomial that is illustrated in red. The rest of the bubbles which are in light grey determine the rows and columns of $P$ that need to be removed, according to Proposition \ref{rm_rows_cols_P_get_P_i}, in order to obtain the corresponding single-block component.

\begin{figure}

\begin{tikzpicture}[scale=.45]
\draw[black] 
	  node at (6,2.5) {Obtaining $P_{13}=(7^2,5,4^2,3,1^2)$}
	  node at (6,1.25){from $P=(15,12^4,11,7,6^2,5,3^4)$};
	    
\draw[very thin]  
	                (0,0)--(15,0)
	                (0,-1)--(15,-1)
	                (0,-2)--(12,-2)
	                (0,-3)--(12,-3)
	                (0,-4)--(12,-4)
	                (0,-5)--(12,-5)
	                (0,-6)--(11,-6)
	                (0,-7)--(7,-7)
	                (0,-8)--(6,-8)
	                (0,-9)--(6,-9)
	                (0,-10)--(5,-10)
	                (0,-11)--(3,-11)
	                (0,-12)--(3,-12)
	                (0,-13)--(3,-13)

	                (1,0)--(1,-14)
	                (2,0)--(2,-14)
	                (3,0)--(3,-10)
	                (4,0)--(4,-10)
	                (5,0)--(5,-9)
	                (6,0)--(6,-7)
	                (7,0)--(7,-6)
	                (8,0)--(8,-6)
	                (9,0)--(9,-6)
	                (10,0)--(10,-6)
	                (11,0)--(11,-6)
	                (12,0)--(12,-1)
	                (13,0)--(13,-1)
	                (14,0)--(14,-1)
	                     ;	

\draw[very thick] 
	(0,0)--(15,0)--(15,-1)--(12,-1)--(12,-2)--(12,-5)--(11,-5)--(11,-6)--(7,-6)--(7,-7)--(6,-7)--(6,-9)--(5,-9)--(5,-10)--(3,-10)--(3,-14)--(0,-14)--(0,0);
	
\draw[very thin,fill=black!5,rounded corners=2mm] 
	(13,0) rectangle (14,-1)
	(10,-3) rectangle (11,-4)
	(9,-4) rectangle (10,-5)
	(8,-5) rectangle (9,-6)
	(1,-12) rectangle (2,-13)
	(0,-13) rectangle (1,-14);

\draw[very thin,fill=blue!25,rounded corners=2mm]
    (12,-1) rectangle (13,-2)
    (7,-6) rectangle (8,-7)
	(6,-7) rectangle (7,-8)
	(3,-10) rectangle (4,-11);

\draw[very thin,fill=black!50,rounded corners=2mm]
    (11,-2) rectangle (12,-3)
    (5,-8) rectangle (6,-9)
	(4,-9) rectangle (5,-10)
	(2,-11) rectangle (3,-12);
\draw[white]
    node at (0,-14.25) {space};

\draw[thick, dashed, red]	
	                    (-0.5,-0.5)--(15.5,-0.5)
	                    (-0.5,-3.5)--(12.5,-3.5)
	                    (-0.5,-4.5)--(12.5,-4.5)
	                    (-0.5,-5.5)--(11.5,-5.5)
	                    (-0.5,-12.5)--(3.5,-12.5)
	                    (-0.5,-13.5)--(3.5,-13.5)
	                    
	                    (0.5,0.5)--(0.5,-14.5)
	                    (1.5,0.5)--(1.5,-14.5)
	                    (8.5,0.5)--(8.5,-6.5)
	                    (9.5,0.5)--(9.5,-6.5)
	                    (10.5,0.5)--(10.5,-6.5);
	                    
\end{tikzpicture}	
\quad \quad \quad \quad
\begin{tikzpicture}[scale=.45]	                    
\draw[black] 
	    node at (6,2.5) {Obtaining $P_{14}=(8,6^2,4,3,2^2)$}
	    node at (6,1.25){from $P=(15,12^4,11,7,6^2,5,3^4)$};
	   	                    
	\draw[very thin]  
	                (0,0)--(15,0)
	                (0,-1)--(15,-1)
	                (0,-2)--(12,-2)
	                (0,-3)--(12,-3)
	                (0,-4)--(12,-4)
	                (0,-5)--(12,-5)
	                (0,-6)--(11,-6)
	                (0,-7)--(7,-7)
	                (0,-8)--(6,-8)
	                (0,-9)--(6,-9)
	                (0,-10)--(5,-10)
	                (0,-11)--(3,-11)
	                (0,-12)--(3,-12)
	                (0,-13)--(3,-13)

	                (1,0)--(1,-14)
	                (2,0)--(2,-14)
	                (3,0)--(3,-10)
	                (4,0)--(4,-10)
	                (5,0)--(5,-9)
	                (6,0)--(6,-7)
	                (7,0)--(7,-6)
	                (8,0)--(8,-6)
	                (9,0)--(9,-6)
	                (10,0)--(10,-6)
	                (11,0)--(11,-6)
	                (12,0)--(12,-1)
	                (13,0)--(13,-1)
	                (14,0)--(14,-1)
	                     ;	

    \draw[very thick] 
	(0,0)--(15,0)--(15,-1)--(12,-1)--(12,-2)--(12,-5)--(11,-5)--(11,-6)--(7,-6)--(7,-7)--(6,-7)--(6,-9)--(5,-9)--(5,-10)--(3,-10)--(3,-14)--(0,-14)--(0,0);
	
	\draw[very thin,fill=black!5,rounded corners=2mm] 
	 (12,-2) rectangle (13,-3)
	(10,-4) rectangle (11,-5)
	(9,-5) rectangle (10,-6)
	(7,-7) rectangle (8,-8)
	(6,-8) rectangle (7,-9)
	(5,-9) rectangle (6,-10)
    (3,-11) rectangle (4,-12)
	(1,-13) rectangle (2,-14);

    \draw[very thin,fill=blue!25,rounded corners=2mm]
    (13,-1) rectangle (14,-2)
    (8,-6) rectangle (9,-7)
    (5,-9) rectangle (6,-10)
	(4,-10) rectangle (5,-11)
	(0,-14) rectangle (1,-15);

    \draw[very thin,fill=black!50,rounded corners=2mm]
    (14,0) rectangle (15,-1)
    (11,-3) rectangle (12,-4)
    (2,-12) rectangle (3,-13);

\draw[thick, dashed, red]	
	                    
	                    (-0.5,-2.5)--(13.5,-2.5)
	                    (-0.5,-4.5)--(12.5,-4.5)
	                    (-0.5,-5.5)--(11.5,-5.5)
	                    (-0.5,-7.5)--(8.5,-7.5)
	                    (-0.5,-8.5)--(7.5,-8.5)
	                    (-0.5,-11.5)--(4.5,-11.5)
	                    (-0.5,-13.5)--(3.5,-13.5)

	                    (1.5,0.5)--(1.5,-14.5)
	                    (3.5,0.5)--(3.5,-12.5)
	                    (6.5,0.5)--(6.5,-9.5)
	                    (7.5,0.5)--(7.5,-8.5)
	                    (9.5,0.5)--(9.5,-6.5)
	                    (10.5,0.5)--(10.5,-6.5)
	                    (12.5,0.5)--(12.5,-3.5);
	                    
 \end{tikzpicture}
 
 \bigskip
 \bigskip
 \bigskip
 
 \begin{center}
 \begin{tikzpicture}[scale=.45]	
\draw[black] 
        node at (6,2.5) {Obtaining $P_{15}=(6,5^2,4,2^2)$}
	    node at (6,1.25){from $P=(15,12^4,11,7,6^2,5,3^4)$};
	\draw[very thin]  
	                (0,0)--(15,0)
	                (0,-1)--(15,-1)
	                (0,-2)--(12,-2)
	                (0,-3)--(12,-3)
	                (0,-4)--(12,-4)
	                (0,-5)--(12,-5)
	                (0,-6)--(11,-6)
	                (0,-7)--(7,-7)
	                (0,-8)--(6,-8)
	                (0,-9)--(6,-9)
	                (0,-10)--(5,-10)
	                (0,-11)--(3,-11)
	                (0,-12)--(3,-12)
	                (0,-13)--(3,-13)

	                (1,0)--(1,-14)
	                (2,0)--(2,-14)
	                (3,0)--(3,-10)
	                (4,0)--(4,-10)
	                (5,0)--(5,-9)
	                (6,0)--(6,-7)
	                (7,0)--(7,-6)
	                (8,0)--(8,-6)
	                (9,0)--(9,-6)
	                (10,0)--(10,-6)
	                (11,0)--(11,-6)
	                (12,0)--(12,-1)
	                (13,0)--(13,-1)
	                (14,0)--(14,-1)
	                     ;	

    \draw[very thick] 
	(0,0)--(15,0)--(15,-1)--(12,-1)--(12,-2)--(12,-5)--(11,-5)--(11,-6)--(7,-6)--(7,-7)--(6,-7)--(6,-9)--(5,-9)--(5,-10)--(3,-10)--(3,-14)--(0,-14)--(0,0);
	
	\draw[very thin,fill=black!5,rounded corners=2mm] 
	 (13,-2) rectangle (14,-3)
	(6,-9) rectangle (7,-10)
    (4,-11) rectangle (5,-12)
	(12,-3) rectangle (13,-4)
	(8,-7) rectangle (9,-8)
	(7,-8) rectangle (8,-9)
	(6,-9) rectangle (7,-10)
	(5,-10) rectangle (6,-11)
	(3,-12) rectangle (4,-13)
	(0,-15) rectangle (1,-16);

    \draw[very thin,fill=blue!25,rounded corners=2mm]
    (14,-1) rectangle (15,-2)
    (9,-6) rectangle (10,-7)
	(1,-14) rectangle (2,-15);

    \draw[very thin,fill=black!50,rounded corners=2mm]
    (2,-13) rectangle (3,-14)
    (11,-4) rectangle (12,-5)
	(10,-5) rectangle (11,-6);
	
	\draw[very thin,fill=red!25,rounded corners=2mm]
	(15,0) rectangle (16,-1);

\draw[thick, dashed, red]	
	                    
	                    (-0.5,-2.5)--(14.5,-2.5)
	                    (-0.5,-3.5)--(13.5,-3.5)
	                    (-0.5,-7.5)--(9.5,-7.5)
	                    (-0.5,-8.5)--(8.5,-8.5)
	                    (-0.5,-9.5)--(7.5,-9.5)
	                    (-0.5,-10.5)--(6.5,-10.5)
	                    (-0.5,-11.5)--(5.5,-11.5)
	                    (-0.5,-12.5)--(4.5,-12.5)

	                    (0.5,0.5)--(0.5,-16.5)
	                    (3.5,0.5)--(3.5,-13.5)
	                    (4.5,0.5)--(4.5,-12.5)
	                    (5.5,0.5)--(5.5,-11.5)
	                    (6.5,0.5)--(6.5,-10.5)
	                    (7.5,0.5)--(7.5,-9.5)
	                    (8.5,0.5)--(8.5,-8.5)
	                    (12.5,0.5)--(12.5,-4.5)
	                    (13.5,0.5)--(13.5,-3.5);

 \end{tikzpicture}
 \end{center}
         \caption{Finding single-block components for partition $P=(15,12^4,11,7,6^2,5,3^4)$ of Example \ref{1stgensEx} using Proposition \ref{rm_rows_cols_P_get_P_i}.}\label{determined}

\end{figure}

\end{example}
\medskip

\subsubsection{The projection map $\pi$ from ${\mathbb V}(E_P)$ to a product of cells of  ``small'' Grassmannians.}
\begin{definition}[The map $\pi$ of $\mathbb V(E_P)$ to the product of small Grassmannians]\label{def_pi}\footnote{We use the term ``small'' Grassmannian to distinguish these from the (large) Grassmannians determined by the projection $\G_T\to \prod_{d\le i\le {\sf j}} \Grass(i+1-t_i, i+1)$ given by $I\to (I_d,\ldots, I_{\sf j})$.}

Suppose $P\in \mathcal P(T)$ and let $I\in \mathbb V(E_P)$. Denote by $W_i$ the vector space generated by $V_{i1}\cup (C_P)_i$ for $i\in [d,{\sf j}]$. Note that $\dim_kW_i=\delta_i+1+t_i$.  
It is straightforward to see that the vector space $I_{W_i}=I\cap W_i$ has dimension $(\delta_i+1)$. The leading monomial of any non-zero element of $I_{W_i}$ belongs to $V_{i1}$ and conversely, given an element $M$ of $V_{i1}$, there is an element of $I_{W_i}$ whose leading monomial is $M$.\\
Let $K_i$ be the vector space generated by $V_{i3}=(C_P)_i\setminus V_{i2}$. We write $W_i$ as a direct sum $W_i=V_i\oplus K_i$ and consider the projection on the first factor 
$\displaystyle{pr_1: W_i\longrightarrow V_i }$ and let $I_{V_i}=pr_1(I_{W_i})$. Then $I_{V_i}\in \Grass(\delta_i+1, V_i)$.
\\
Thus, we have constructed a morphism 
${\mathbb V}(E_P) \stackrel{\pi_i}{\longrightarrow} \Grass(\delta_i+1, V_i)$,
and, taking the product $\pi=(\pi_d,\ldots, \pi_{\sf j})$
we have a morphism 
\begin{equation}\label{isomeq}
{\mathbb V}(E_P) \stackrel{\pi}{\longrightarrow} \prod_{i=d}^{i={\sf j}}\Grass(\delta_i+1, V_i).
\end{equation}
\end{definition}
\begin{remark}\label{im_of_pi_i}
Note that by Definition \ref{def_pi} and the definition of the difference-one hook code, the image of $\pi_i$ is a Schubert cell $\mathbb V(E_{P_i})$ in
$\Grass(\delta_i+1, V_i)$,
whose dimension is $\left|{\mathcal Q}(P)_i\right|$, the length of the $i$-th block of the hook code of $P$. \\
When $\delta_{i+1}=0$ (that is, $t_i=t_{i+1}$) we have $\dim(V_{i})=\delta_i+1$ and $\Grass(\delta_i+1, V_i)$ is just one point: the 
$i$-th block of the hook code of $P$ in this case is empty, so has length zero (see Remark \ref{empty_hook_code_block_partition}).
\end{remark}
\begin{lemma}[Morphism $\pi$ to the component small Grassmannians]\label{morphismlem} Let $P\in \mathcal P(T)$. The morphism $\pi$ of \eqref{isomeq} determines a morphism $$\pi:{\mathbb V}(E_P) \stackrel{\pi}{\longrightarrow} \prod_{i=d}^{i={\sf j}}\mathbb V(E_{P_i}).$$
\end{lemma}
\begin{proof} The bijection $s_i : V_i\rightarrow\Mon(R_{d_i})$ of Equation \eqref{monomeq} induces a linear isomorphism
 $\tilde{s_i}:V_i\to R_{d_i}$. We thus have an isomorphism $\tilde{s_i}:\Grass(\delta_i+1, V_i)\to \Grass(\delta_i+1, R_{d_i})$, taking 
 $I_{V_i}=pr_1(I_{W_i})$ to the subspace $\tilde{s_i} (I_{V_i})\subset R_{d_i}$. Then by Definition \ref{Pidef1}, Equation \ref{isomeq} and Remark \ref{im_of_pi_i} determine the morphism $\pi$ of the Lemma.
\end{proof}

\begin{theorem}\label{projectionthm}The morphism $\pi$ of 
Lemma \ref{morphismlem} is an isomorphism from the Jordan cell $\mathbb V(E_P)$ onto its image $\prod_{i=d}^{i={\sf j}}\mathbb V(E_{P_i})$.
\end{theorem}
\begin{proof} Each $\mathbb V(E_{P_i})$ is an affine space; from Theorem \ref{GTdecompthm} and Proposition \ref{hookP(i)lem} the dimension of $\mathbb V(E_P)$ is the sum $\sum \dim \mathbb V(E_{P_i})$. 
We know that the difference-one hook code of $P_i$ is the $i$-$th$ component $\big(\mathfrak{h}_i(P)\big)$. 
\\
Using the notation of Definition \ref{Pidef1}, consider the set $V_{i1}$, which contains all the degree-$i$ horizontal-border monomials of $E_P$. 
\\Let $V_{i1}=\left\{M_{i,1}, \dots , M_{i,\delta_i}, M_{i,{\delta_i+1}} \right\}$ (numbered from top to bottom). Denote by $b_{i,l}$ the number of degree-$i$ hand monomials above $M_{i,l}$. Then we have $b_{i,{\delta_i+1}}\geq b_{i,\delta_i}\geq \cdots \geq b_{i,1}$, and the affine space $\mathbb V(E_{P_i})$ has dimension $\displaystyle{\sum_{l=1}^{l=\delta_i+1}b_{i,l}=\left|\big(\mathfrak{h}_i(P)\big)\right|}$. We may think of $\mathbb V(E_{P_i})$ as 
$\displaystyle{\prod_{l=1}^{l=\delta_i+1}{\sf k}^{b_{i,l}}}$ (\textsl{$b_{i,l}$ free parameters} for each $M_{i,l}$). Thus, if $P$ is a single-block partition, then we are done. Suppose $P$ is not a single-block partition and let $\hMon (E_P)$ be the set of horizontal-border monomials of $E_P$. 
Suppose $\hMon (E_P)=\left\{y^{\beta_0}, xy^{\beta_1}, \cdots , x^my^{\beta_m} \right \}$, with $\beta_0\geq \beta_1 \geq \cdots \geq \beta_m$. \\
For $0\leq l \leq m$, let $b_l$ be the number of degree-$(l+\beta_l)$ hand monomials above $x^ly^{\beta_l}$. \\
Let $P^{\prime}$ be the partition obtained by deleting the first column of $P$. 
The morphism ${\mathbb V}(E_P) \rightarrow {\mathbb V}(E_{P^{\prime}}): I\mapsto (I:x)$ is a trivial fibration whose fiber has dimension $b_0$ and we have 
${\mathbb V}(E_P) \cong {\sf k}^{b_0}\times {\mathbb V}(E_{P^{\prime}})$ (see for example \cite{Y0}, Prop. 1.7).  It is easy to see that $E_{P^{\prime}}$ has one less horizontal-border monomial than $E_P$. The horizontal-border monomials of $E_{P^{\prime}}$ are deduced from that of $E_P$ by dividing each of the monomials of the set $\left\{xy^{\beta_1}, \cdots , x^my^{\beta_m} \right \}$ by $x$. Also, for $l>0$, the number of degree-$(l+\beta_l)$ hand monomials above $x^ly^{\beta_l}$ is the same as that  of degree-$(l-1+\beta_l)$ hand monomials above $x^{l-1}y^{\beta_l}$ (for the new partition $P^{\prime}$).\\ Suppose $P=\left(p_1^{r_1}, \cdots , p_s^{r_s}\right)$, with $p_1>\cdots >p_s>0$. Iterating the trivial fibration ${\mathbb V}(E_P) \rightarrow {\mathbb V}(E_{P^{\prime}}): I\mapsto (I:x)$, $(p_1-1)$ times, we see that the partition associated to $\left(I:x^{p_1-1}\right)$ is the partition 
$P_{p_1-1}=\left(1^{r_1}\right)$. This partition $\left(1^{r_1}\right)$ is that of a zero dimensional cell of Hilbert function $T_{p_1-1}=(1, \cdots , 1)$.
We then obtain the proof of the Theorem by induction.
\end{proof}

\subsubsection{Examples of the projection map $\pi$.}
\begin{example} Let $T=(1,2,2,1)$.  We have $\delta_2=t_1-t_2=0$, $\delta_3=t_2-t_3=1$, 
$\delta_4=t_3-t_4=1$, and $\mathfrak B(T)=\left(\mathfrak B_2,\mathfrak B_3\right)=\left((1\times 1)_2,(1\times 2)_3\right)$. 
Here the product of  ``small'' Grassmannians is 
$$G=\Grass(\delta_2+1, \delta_2+1+\delta_3)\times \Grass(\delta_3+1, \delta_3+1+\delta_4)=\Grass( 1,2)\times \Grass(2,3).$$
Consider the partition $P=(3,3)$. 
\begin{center}
$P: \ydiagram{3,3}.$
\end{center}
By Definition \ref{def_pi}, we have $V_2=\langle y^2,x^2\rangle$,  $W_2=\langle y^2,xy, x^2\rangle$, $ V_3=\langle xy^2,x^2y,x^3\rangle$ and $ W_3=V_3$. \\
Any element $I$ in the cell ${\mathbb V}(E_P)$ is of the form $I=(y^2+a_1xy+a_2x^2,x^3)$, with $(a_1,a_2)\in {\sf k}^2$. So ${\mathbb V}(E_P)$ 
is a two dimensional affine space. \\
For $I=(y^2+a_1xy+a_2x^2,x^3)\in {\mathbb V}(E_P)$ we get $I\cap W_2=\langle y^2+a_1xy+a_2x^2\rangle$ and the projection of $I\cap W_2$ on $V_2$ is $I_{V_2}=\langle y^2+a_2x^2\rangle \in \Grass(1,V_2)$.  
 Also $I\cap W_3=\langle xy^2+a_1x^2y,x^3\rangle=I_{V_3}\in \Grass(2, V_3)$. 
We view ${\mathbb V}(E_P)$ as a product of two cells, one in $\Grass( 1,2)$ corresponding to  single-block $T_2=(1,1)$ and another one in $\Grass(2,3)$ corresponding to $T_3=(1,2,1)$. \\
The difference-one hook code of $P=(3,3)$ is $${\mathfrak Q}(P)=((1)_2,(1)_3)\subset \mathcal B(T)=((1\times 1)_2, (1\times 2)_3).$$ The code $(1)_2$ corresponds to the vector space $V_2$ and the small cell $\mathbb V(E_{P_2})=\left\{ \langle y^2+a_2x^2\rangle,{a_2\in {\sf k}}\right\}\subset \Grass( 1,2)\cong\Grass(1,V_2)$:
\begin{center}
\begin{tabular}{c}
$P_2$ \\ \ydiagram{2}
\end{tabular}
.
\end{center}
 The code $(1)_3$ corresponds to the vector space $V_3$ and the small cell $\mathbb V(E_{P_3})=\left\{ \langle xy^2+a_1x^2y,x^3\rangle,{a_1\in {\sf k}}\right\}\subset\Grass(2,3)\cong\Grass(2,V_3)$ (note, these are labelled by degree: $V_i\subset R_i$):
 \begin{center}
\begin{tabular}{c}
$P_3$ \\ \ydiagram{2,2}
\end{tabular}
.
\end{center}

So we get ${\mathbb V}(E_P)$ as a product of two affine lines: ${\mathbb V}(E_P)=\mathbb V(E_{P_2})\times \mathbb V(E_{P_3})$. \end{example}
\noindent
\begin{example}\label{53313ex} Let $T=(1,2,3,4,5,4,2)$ and consider $P=(5^3,3,1^3)$.  (See Figure \ref{fig5}.)
We have $(\delta_5,\delta_6,\delta_7)=(1,2,2)$. The two single-block Hilbert functions associated to $T$ are $T_5=(1,2,3,2)$ and $T_6=(1,2,3,4,2)$ 
(see Definition \ref{Pidef1}). \\
We will view the cell ${\mathbb V}(E_P)$ as a product of cells in 
$\Grass( 2,4)\times \Grass(3,5)$. \\
The difference-one hook code of $P$ is ${\mathfrak Q}(P)=\big((1,1)_5,(2,0)_6\big)$. \\
The partition $P_5=(3^2,2)$ of diagonal lengths $T_5$ has hook code ${\mathfrak Q}(P_5)=\big( (1,1)\big)$ (the degree-$5$ block of ${\mathfrak Q}(P)$). 
The partition $P_6=(4^2,2,1^2)$ of diagonal lengths $T_6$ has hook code ${\mathfrak Q}(P_6)=\big( (2,0)\big)$ (the degree-$6$ block of ${\mathfrak Q}(P)$). 
\begin{figure}[!h]
\begin{center}
	\begin{tikzpicture}[scale=.8]
	\draw[dashed,very thin] 	(0,0)--(5,0)
						(0,-1)--(5,-1)
						(0,-2)--(5,-2)
						(0,-3)--(3,-3)
						(0,-4)--(1,-4)
						(0,-5)--(1,-5)
						(0,-6)--(1,-6)
						(0,0)--(0,-7)
						(1,0)--(1,-4)
						(2,0)--(2,-4)
						(3,0)--(3,-3)
						(4,0)--(4,-3);			
	\draw[thin, red,dashed,fill=yellow!10] (4,-1) rectangle (5,-2)
							 (2,-3) rectangle (3,-4)
							 (4,-2) rectangle (5,-3)
							 (0,-6) rectangle (1,-7);
	\draw[thin, blue,fill=black!5,rounded corners=3mm] (1,-4) rectangle (2,-5)
											    (2,-4) rectangle (3,-5)
											    (3,-3) rectangle (4,-4)
											    (5,0) rectangle (6,-1);
	\draw[red] node at (4.5,-1.5) {\footnotesize{$x^4y$}}
			  node at (2.5,-3.5) {\footnotesize{$x^2y^3$}}
			  node at (4.5,-2.5) {\footnotesize{$x^4y^2$}}
			  node at (0.5,-6.5) {\footnotesize{$y^6$}};	
	\draw[blue] node at (1.5,-4.5) {\footnotesize{$xy^4$}}
			    node at (2.5,-4.5) {\footnotesize{$x^2y^4$}}
			    node at (3.5,-3.5) {\footnotesize{$x^3y^3$}}
			    node at (5.5,-0.5) {\footnotesize{$x^5$}};	  
	\draw[very thick] (5,0)--(5,-3)--(3,-3)--(3,-4)--(1,-4)--(1,-7)--(0,-7);
         \end{tikzpicture}
         \caption{Ferrers diagram of $P=(5^3,3,1^3)$: \textsl{hand monomials} are marked in red and 
         \textsl{border monomials} are {marked} in blue (Example \ref{53313ex}).}\label{fig5}
\end{center}
\end{figure}
\\
We now display the isomorphism $\pi: {\mathbb V}(E_P)\to{\mathbb V}(E_{P_5})\times {\mathbb V}(E_{P_6})$. \\

By Definition \ref{def_pi}, we have 
$$V_5=\langle xy^4,x^2y^3,x^4y,x^5 \rangle, W_5=R_5, V_6=\langle y^6,x^2y^4,x^3y^3,x^4y^2,x^5y\rangle=W_6.$$
The projection of $I_5$ onto $V_5$ is a $2$-dimensonial vector space $$I_{V_5}=\langle xy^4+\alpha_1x^2y^3+\alpha_2x^4y,x^5 \rangle ,
(\alpha_1,\alpha_2)\in {\sf k}^2.$$
The projection of $I_6$ onto $V_6$ is a $3$-dimensonial vector space $$I_{V_6}=\langle x^2y^4+a_1x^4y^2,x^3y^3+a_2x^4y^2,x^5y \rangle ,
(a_1,a_2)\in {\sf k}^2.$$
So we have $I_{V_5}\in \Grass(2,V_5)\cong \Grass(2,4)$, $I_{V_5}\in \Grass(3,V_6)\cong \Grass(3,5)$ and ${\mathbb V}(E_P)$ can be viewed as ${\mathbb V}(E_{P_5})\times {\mathbb V}(E_{P_6})$. \vskip 0.2cm\noindent
\textbf{Note.} Suppose we are given $(L_5,L_6)\in {\mathbb V}(E_{P_5})\times {\mathbb V}(E_{P_6})$ with 
$L_5=\langle xy^4+\alpha_1x^2y^3+\alpha_2x^4y,x^5 \rangle , (\alpha_1,\alpha_2)\in {\sf k}^2$ and 
$L_6=\langle x^2y^4+a_1x^4y^2,x^3y^3+a_2x^4y^2,x^5y \rangle , (a_1,a_2)\in {\sf k}^2$. Then, using standard basis techniques (Theorem I.1.9 of \cite{brian}, Propositions 2 and 3 of \cite{brian-gal}) one can see that there 
is a unique ideal $I\in {\mathbb V}(E_P)$ such that $I_{V_5}=L_5$ and $I_{V_6}=L_6$: 
$$
I=\left(x^5, x^3y^3+a_2x^4y^2, x((y^4+a_1x^2y^2)+\alpha_1(xy^3+a_2x^2y^2)+\alpha_2x^3y), y^7\right).
$$
\end{example}

\noindent
In connection with Lemma \ref{degree_i+1_rel_gen_formulas} where we will be counting the degree $i+1$ relations and corner-monomials (generators -Definition \ref{def_corner-mon}) of $E_P$ it is interesting to note that,
\begin{lemma}\label{monom_EPi_corners_and_gens_bijection}
The bijection $s_i : V_{i1}\cup V_{i2} \longrightarrow \Mon (R_{d_i})$ (Equation \ref{monomeq})  induces 
\begin{enumerate}[(i).]
\item a one to one correspondence between the degree $i+1$ relations of $E_P$ and the degree $d_i+1=\delta_i+\delta_{i+1}+1$ 
relations of $E_{P_i}$
\item a one to one correspondence between the first  (numbering from top to bottom-- lex order) $\delta_{i+1}$ degree $i+1$ \textsl{vertical-border monomials} of the ideal $E_P$ and the degree $d_i+1$ \textsl{vertical-border monomials} of $E_{P_i}$
\item a one to one correspondence between the degree $i+1$ corner-monomials of $E_P$ and the degree $d_i+1$ corner-monomials of $E_{P_i}$.
\end{enumerate}
\end{lemma}
\begin{proof}
\begin{enumerate}[(i).]
\item Suppose that the monomial $x^{\alpha}y^{i+1-\alpha}$ ($0<\alpha <i+1$) corresponds to a degree $i+1$ relation. Then $x^{\alpha-1}y^{i+1-\alpha}$ 
is a \textsl{horizontal-border monomial} of $E_P$ and $x^{\alpha}y^{i-\alpha}$ is a \textsl{vertical-border monomial} of $E_P$. 
Now, consider the set $(E_P)_{i,\alpha}$ of degree $i$ \textsl{horizontal-border monomials} of $E_P$ that are above $x^{\alpha}y^{i-\alpha}$. If this set is empty, then $x^i\notin (C_P)_i$, so $x^{\alpha}y^{i-\alpha}\in  V_{i1}$ and the degree $i+1$ relation of $E_P$ corresponding to $x^{\alpha}y^{i+1-\alpha}$ is sent to a degree $d_i+1$ relation of $E_{P_i}$. If the set $(E_P)_{i,\alpha}$ is not empty, let $x^{\alpha^{\prime}}y^{i-\alpha^{\prime}}$ be the first element of $(E_P)_{i,\alpha}$ just above $x^{\alpha}y^{i-\alpha}$. By definition, $s_i$ sends $x^{\alpha-1}y^{i+1-\alpha}$ and $x^{\alpha^{\prime}}y^{i-\alpha^{\prime}}$ to two consecutive \textsl{horizontal-border monomials} of $E_{P_i}$, resulting to a degree $d_i+1$ relation of $E_{P_i}$.
\item By Definition \ref{defi_horiz_vertic_border_monom} we know that any degree $i+1$ \textsl{vertical-border monomial} is just to the right of a unique degree $i$ hand monomial. The second statement of the Lemma is then just a remark based on the fact that any element of $V_{i2}$ is a degree $i$ hand  monomial of $P$ and $s_i$ sends the elements of $V_{i2}$ to the degree $d_i$ hand monomials of $E_{P_i}$. 
\item Suppose $x^{\alpha}y^{i+1-\alpha}$ is a degree $i+1$ hook corner of $E_P$.
\begin{enumerate}[a)]
\item If $\alpha=0$, then one can easily see that $y^{d_i}\notin E_{P_i}$, so $y^{d_i+1}$ is a corner-monomial of $E_{P_i}$.
\item If $\alpha>0$ then the corner-monomial $x^{\alpha}y^{i+1-\alpha}$ is also a \textsl{vertical-border monomial} of $E_P$  that will correspond via $s_i$ to a degree $d_i+1$ corner-monomial of $E_{P_i}$.
\end{enumerate}
\end{enumerate}
\end{proof}

\section{Number of generators for a single-block partition.}\label{singlekappasection}

We first state the known bounds for the number of generators of a graded ideal $I$ of Hilbert function $H(R/I)=T$ for arbitrary $T$ (Lemma \ref{kappaTlem}).
In Theorem \ref{kappathm} we determine the number of generators $\kappa(P)$ for generic ideals in the cell $\mathbb V(E_P)$ where $P$ has diagonal lengths $T$ satisfying the single-block Equation \eqref{Tsingleeqn}.
\subsection{Lower bound $\kappa(T)$ on the number of generators of an ideal $I$ in $\G_T$.}
We recall Equation \ref{Teq} for an arbitrary codimension two Hilbert function $T$:
\begin{equation*}
T=(1,2,\ldots, d-1,d,t_{d},\ldots, t_{\sf j},0) \text { where } d\ge t_d\ge t_{d+1}\ge \cdots \ge t_{\sf j}>0.
\end{equation*}
Here ${\sf j}$ is the (highest) socle degree of $A=R/I$. Recall from Section \ref{GTsec} that $\G_T$ is the irreducible projective variety parametrizing the graded ideals $I$ in $R={\sf k}[x,y]$ such that $A=R/I$ has Hilbert function $T$.
\begin{definition}[Order of a Hilbert function $T$]\label{order_of_Tdef}
Let $T$ be a sequence satisfying Equation \eqref{Teq}. Set $\nu(T)=d$, usually called the order of $T$: that is $\nu(T)$ is the order of graded ideals $I\in \G_T$ -- that define an Artinian algebra $A=R/I$ of Hilbert function $T$.
\end{definition}

\begin{definition}\label{kappadef} We let $\kappa(T)$ be the minimum number of generators for the ideal $I$ corresponding to a generic element of $\G_T$. 
If $P$ is a partition of diagonal lengths $T$, we denote by $E_P$ the monomial ideal associated to $P$ and set $\kappa(P)$ to be the minimum number of generators for a generic element $I$ in the cell $\mathbb{V}(E_P)$. 
\end{definition}

Given a sequence $T$ satisfying Equation \eqref{Teq}, recall that we denote by $\delta_i$ the first difference function of $T$:
\begin{equation}\label{diffeq}
\delta_i=t_{i-1}-t_{i}  \text { for }  i \in [\nu(T),{\sf j}+1].
\end{equation}
\par
The following result (i)-(ii) is shown in \cite[Theorem 4.3, Lemma 4.5]{Ia}, but a separate proof will also result from
our work here (see Remark \ref{lower_bound_for_degree_i+1_gens}).  A different proof of (ii) is given by M. Mandal and M.E. Rossi in \cite[Theorem~2.1]{MR}. The statement (iii) is obvious. 
We denote by $[k]^+=\max\{k,0\}$.
\begin{lemma}\label{kappaTlem} Let $T$ satisfy Equation \eqref{Teq}, and let $I$ be a homogeneous ideal such that $A=R/I$ has Hilbert function $T$. Then
\begin{enumerate}[(i).]
\item  $I$ has at least  $[\delta_i-\delta_{i-1}]^+$ generators of each degree $i\ge \nu(T)$.
\item  A generic graded ideal $I\in\G_T$ has
\begin{equation}\label{kappaTeq}
\kappa(T)=1+\delta_{\nu(T)}+\sum_{i> \nu(T)}[\delta_{i+1}-\delta_i]^+
\end{equation} 
generators, exactly $[\delta_i-\delta_{i-1}]^+$ in each degree $i\ge \nu(T)$.

\item  The ideal $E_{P_0}, P_0=T$ (listed as a partition) has $\nu(T)+1=1+\sum_{i\ge \nu(T)} \delta_i$ generators, and $\kappa_{E_{P_0},i}=\delta_i$ for $i> \nu(T)$ and $1+\delta_{\nu(T)}$ for $i=\nu(T)$. This is the termwise maximum $\kappa_I(z)$ that occurs for any ideal $I\in \G_T$: that is $\kappa_{I,i}\le \delta_i$ for $i>\nu(T)$ and $\kappa_{I,\nu(T)}\le 1+\delta_{\nu(T)}$.
\end{enumerate}
\end{lemma}
\begin{definition}\label{special-partition-definition}
If $P$ is a partition of diagonal lengths $T$ such that $\kappa(P)\neq \kappa(T)$, then we say $P$ is \emph{special}. 
If $\kappa(P) = \kappa(T)$ we say $P$ is \emph{non-special}. 
\end{definition}\par\noindent
\begin{example}\label{12321ex} Let $T=(1,2,3,2,1)$. We have $\kappa (T)=2$, as the generic ideal in $\G_T$ is a complete intersection of generator degrees $(3,3)$. For $P=(5,3,1)$ we also have $\kappa (P)=2$: here $\mathbb V(E_P)$ is open dense in $\G_T$, so $P=(5,3,1)$ is non-special. But for $P=(3,3,1^3)$ we have $\kappa(P)=3$ since an $R$-relation between the generators $y^5,y^2x+\cdots$ cannot yield the generator $x^3$: so $P=(3,3,1^3)$ is special.
\end{example}
\subsection{Single-block partitions $P$, and $\kappa(P)$.}\label{single block-partition-and-kappa-section}
Henceforth in this section we let $T$ be a Hilbert function that satisfies 
\begin{equation}\label{Tsingleeqn}
T=(1,2,\ldots, d-1, d,t_d,0).
\end{equation}
where $d\ge t_d$ and we let $s=d+1-t_d$. We term this a \emph{single-block} Hilbert function.
 In this case, $\G_T$ is isomorphic to the Grassmannian variety $\Grass(s,R_d)$ where $R_d$ is the vector space of the degree $d$ homogeneous forms of $R={\sf k}[x,y]$:
$$
\begin{array}{rcl}
\Phi: \G_T & \rightarrow & \Grass(s, R_d)\\
I & \mapsto & I_d 
\end{array}
.
$$
Also, by Equation \ref{kappaTeq} we have for a single-block Hilbert function
\begin{equation}\label{kappaTsingleeq}
\kappa (T)=s+\delta, \text { where }\delta=\max\{t_d+1-s,0\}. 
\end{equation}
Let $P$ be a partition of diagonal lengths $T$. The corners of the Ferrers diagram of $P$ correspond to monomials $x^{\alpha}y^{\beta}$ that belong to a minimal set of generators for the monomial ideal $E_P$. We may call such monomials,  corner-monomials of $P$.
Let $I\in \G_T$ be an element of the Jordan cell $\mathbb{V}(E_P)$. Then the corner-monomials of $P$ are leading terms of a system of generators 
$\mathcal{B}(I)$ of $I$. The system $\mathcal{B}(I)$ may not be minimal. 
By definition of $T$, a minimal set of generators of $E_P$ must contain $s$ degree $d=\nu(T)$ corner-monomials. These degree $d$ corner-monomials are 
leading monomials for the degree $d$ elements of the system of generators $\mathcal{B}(I)$. Since we are looking for a minimal set of generators for $I$, we want a criterion to decide that a degree $d+1$ element of $\mathcal{B}(I)$ can be obtained using a relation involving degree $d$ elements of 
$\mathcal{B}(I)$. That is where corner ``kick-off'' comes into play.

\medskip
Let $a$ be integer such that $0\leq a <d=\nu(T)$ and set $d^{\prime}=d-a$. Suppose $m$ is an integer such that $1<m\leq  d^{\prime}$. 
For any integer $i$ such that $0\leq i\leq m$, set $\displaystyle{K_i=x^{a}\cdot\left(x^{m-i}y^{d^{\prime}-m+i}\right)}$. 
The $K_i$'s form a set of $m+1$ consecutive degree $d$ monomials in two variables:
$$
x^{a+m}y^{d^{\prime}-m},x^{a+m-1}y^{d^{\prime}-m+1}, \ldots , x^{a+1}y^{d^{\prime}-1},x^{a}y^{d^{\prime}}
.
$$
From these $m+1$ consecutive monomials we have $m$ relations: $yK_i-xK_{i+1}=0$, $(0\leq i <m)$. Suppose the $K_i$'s are leading monomials of 
some elements of $\mathcal{B}(I)$.
Note that by definition, if $f_0, \ldots, f_m$ are degree-$i$ forms such that $f_i$ has leading monomial $K_i$, then $x^a$ divides any element of the ideal generated by
$(f_0, \ldots, f_m)$. Thus, assuming that $\dim_{{\sf k}}\left(R_1\cdot \langle f_0, \ldots, f_m \rangle \right)=2(m+1)$ requires $2(m+1)\leq d^{\prime}+2=\dim_{{\sf k}}(R_{d^{\prime}+1})$, that is, $2m\leq d^{\prime}$. 

\medskip \noindent 
For simplicity we now assume $a=0$, so $\displaystyle{K_i=x^{m-i}y^{d-m+i}}$,  $2m\leq d$, and we let $N_1, \ldots , N_m$ be the $m$ degree $d+1$ monomials given by
$$
\left\{
\begin{array}{l}
N_i=x^{\alpha_i}y^{\beta_i}, \ \alpha_i+\beta_i=d+1, 
\\
0\leq \beta_1<\beta_2< \cdots <\beta_m<d-m
\\
m+1<\alpha_m<\alpha_{m-1}< \cdots <\alpha_1\leq d+1
\end{array}
\right.
.
$$\par
Concerning the next Lemma, although J.~Brian\c{c}on and A.~Galligo state their standard basis result that we use in characteristic zero, it is valid also for characteristic greater than the socle degree $d$.
 This is the key step in the paper where we need to restrict the characteristic of ${\sf k}$.

\begin{lemma}[How to kick off corners]\label{corner-kick-off-lemma} Assume that the characteristic of $\sf k$ is zero, or that $\sf k$ is infinite of characteristic $p$ greater than the socle degree $d$.
With the above notation, there exist $m+1$ degree $d$ forms $f_0, \ldots, f_m$ such that $f_i$ has leading monomial $K_i$ 
and $N_i$ is a leading monomial of a degree $d+1$ element of the ideal generated by $(f_0, \ldots, f_m)$.
\end{lemma}
\begin{proof}
Using a technique of standard basis calculations developed by J.~Brian\c{c}on and A.~Galligo in \cite{brian-gal} (requiring the restriction on the characteristic of $\sf k$)\footnote{See \cite[Theorem I.1.9]{brian}, \cite[Props. 2,3]{brian-gal}, also \cite[\S 1]{PfR}.},  we can inductively  construct $f_0, \ldots, f_m$ such that $N_i\in (f_0, \ldots, f_m)$.  Let $$f_0=x^my^{d-m}, \ f_1=x^{m-1}y^{d-m+1}+\lambda_1x^{\alpha_1-1}y^{\beta_1},$$ 
where $\lambda_1\in {\sf k}$.
One can see that $xf_1-yf_0=\lambda_1x^{\alpha_1}y^{\beta_1} $, so if 
$\lambda_1\neq 0$, we have 
$$N_1=x^{\alpha_1}y^{\beta_1}\in (f_0, \ldots, f_m).$$
In general, for $0\leq i<m$, suppose that we have $f_i=x^{m-i}u_i+\lambda_ix^{\alpha_i-1}y^{\beta_i}$ where $u_i$ is a degree $d-m+i$ form 
such that $u_i(0,y)=y^{d-m+i}$. 
\\
Then we set 
$$f_{i+1}=x^{m-i-1}y\left(u_i+\lambda_ix^{\alpha_i-1-m+i}y^{\beta_i}\right)+\lambda_{i+1}x^{\alpha_{i+1}-1}y^{\beta_{i+1}}.$$
So, $xf_{i+1}-yf_i=\lambda_{i+1}x^{\alpha_{i+1}}y^{\beta_{i+1}}$ and for $\lambda_{i+1}\neq 0$, we have 
$N_{i+1}\in (f_0, \ldots, f_m)$.
Note that for $i=0$, $u_0=y^{d-m}, \lambda_0=0$; for $i=1$, $u_1=y^{d-m+1}$; thus, inductively, we have constructed 
$f_0, \ldots, f_m$ such that $N_i\in (f_0, \ldots, f_m)$.
\end{proof}

\begin{remark}[Choosing which corner should be kicked off]\label{selected-corner-kick-off}
Given  $r$ indices $i_1, \ldots, i_r$ such that $1\leq i_1< \ldots <i_r\leq m$, in the inductive construction of 
$(f_0, \ldots, f_m)$ of Lemma \ref{corner-kick-off-lemma}, if we let $\lambda_{i_l}=0$ ($1\leq l\leq r$), then none of the monomials 
$N_{i_l}$ will be kicked off. So, if $\lambda_{i_l}=0$ for $1\leq l\leq r$, then $N_{i_l}\notin (f_0, \ldots, f_m)$.
\end{remark}
We remind the reader of the Definition \ref{hookdef} and Figure \ref{hook1fig} of a difference-one hook, and Definition \ref{hookcodedef} of the hook code. In the next Lemma and Theorem a hook code of $P$ has a single non-zero partition $\mathfrak h_{d}(P)=\mathfrak{Q}(P)$, which for short we term its hook code. Note that $n=\delta_{d+1}$ is the number of parts of $\mathfrak h_{d}(P)$ and some parts may be zero.

\begin{lemma}[Counting the corner-monomials of $P$]\label{countcornerlem}
Let $T=(1,2, \ldots , d,t_d=t,0)$, $t>0$ and  set $s=d+1-t$. Suppose that $P$ is a partition of diagonal lengths $T$ and difference-one hook code $\mathfrak{Q}(P)=(h_1^{l_1}, \dots, h_n^{l_n})$ (where $s\geq h_1> h_2 >\cdots >h_n \geq 0$). Then the minimum number of generators $b_1(E_P)$ of the monomial ideal $E_P$ is given by the following formula. 
$$
\begin{array}{ll}
b_1(E_P)=s+t-n,&\mbox{ if } h_1<s \mbox{ and } h_n>0\\
b_1(E_P)=s+t-n+1, &\mbox{ if } h_1=s \mbox{ and } h_n>0, \mbox{ or } h_1<s \mbox{ and } h_n=0 \\
b_1(E_P)=s+t-n+2,&\mbox{ if } h_1=s \mbox{ and } h_n=0.
\end{array}
$$
\end{lemma}

\begin{proof}
This is an easy count that we obtain by looking at the Ferrers diagram of $P$.
\end{proof}\par
Note that $b_1(E_P)-\kappa(P)$ counts the number of degree $d+1$ corner-monomials we have been able to kick-off.

\begin{example}\label{972bex}
Suppose $T=(1,2,3,4,5,6,7,8,4)$. Then $d=8$, $t_d=4$ and $s=5$. Let $P$ be the partition of diagonal lengths $T$ 
defined by $P=(9,7^2,6,4^2,2,1)$ (See Figure \ref{fig11}). We have $\mathfrak{Q}(P)=(5,4^2,3)$. The monomial ideal $E_P$ associated to $P$ is 
generated by $\left(y^8,xy^7,x^2y^6,x^4y^4,x^7y,x^6y^3,x^9\right)$. 
Using Lemma \ref{corner-kick-off-lemma}, we see that the degree $9$ corners of $P$ associated to $x^6y^3$
and {$x^9$} can be kicked-off using the degree $8$ corners associated to the consecutive monomials 
{$y^8, xy^7, x^2y^6$}.
\begin{figure}[!h]
\begin{center}
	\begin{tikzpicture}[scale=.8]
	\draw[dashed,very thin] 	(0,0)--(9,0)
						(0,-1)--(7,-1)
						(0,-2)--(7,-2)
						(0,-3)--(6,-3)
						(0,-4)--(4,-4)
						(0,-5)--(4,-5)
						(0,-6)--(2,-6)
						(0,-7)--(1,-7)
						(0,0)--(0,-8)
						(1,0)--(1,-7)
						(2,0)--(2,-6)
						(3,0)--(3,-6)
						(4,0)--(4,-4)
						(5,0)--(5,-4)
						(6,0)--(6,-3)
						(7,0)--(7,-1)
						(8,0)--(8,-1);
	\draw[tips, -{Latex[open,length=8pt,bend]},dashed,red,thick] 	(1.5,-8.5) to [bend left] (6,-3.5);
	\draw[tips, -{Latex[open,length=8pt,bend]},red,thick] (2.5,-7.5) to [bend left] (9.2,-0.5);	
	\draw node at (1.5,-8.5) {\faSoccerBallO};
	\draw node at (2.5,-7.5) {\faSoccerBallO};
	\draw[thin, blue,fill=black!5,rounded corners=3mm] (9,0) rectangle (10,-1)
											   (6,-3) rectangle (7,-4)
											   (2,-6) rectangle (3,-7)
											   (1,-7) rectangle (2,-8)
											   (0,-8) rectangle (1,-9);
	\draw[thin, blue,fill=green!10,rounded corners=3mm] (2,-7) rectangle (3,-8)
											      (1,-8) rectangle (2,-9);	
	\draw[tips, -{Latex[open,length=8pt,bend]},dashed,red,thick] 	(1.5,-8.5) to [bend left] (6,-3.5);
	\draw[tips, -{Latex[open,length=8pt,bend]},red,thick] (2.5,-7.5) to [bend left] (9.2,-0.5);	
	\draw node at (1.5,-8.5) {\faSoccerBallO};
	\draw node at (2.5,-7.5) {\faSoccerBallO};
	\draw[blue] node at (9.5,-0.5) {\footnotesize{$x^9$}}
			    node at (6.5,-3.5) {\footnotesize{$x^6y^3$}}
			    node at (2.5,-6.5) {\footnotesize{$x^2y^6$}}
			    node at (1.5,-7.5) {\footnotesize{$xy^7$}}
			    node at (0.5,-8.5) {\footnotesize{$y^8$}};			
	\draw[very thick] (9,0)--(9,-1)--(7,-1)--(7,-3)--(6,-3)--(6,-4)--(4,-4)--(4,-6)--(2,-6)--(2,-7)--(1,-7)--(1,-8)--(0,-8);
         \end{tikzpicture}
         \caption{Kicking off corners of the partition $P=(9,7^2,6,4^2,2,1)$ (Example \ref{972bex}).}\label{fig11}
         \end{center}
	\end{figure}		
\end{example}

\medskip
Let $P$ be a partition of diagonal lengths $T$. Suppose $P=(p_1, \ldots , p_m)$, with $p_1\geq p_2 \geq \cdots \geq p_m$. Let 
$P^{\prime}=(p_1^{\prime}, \ldots , p_m^{\prime})$ with $p_i^{\prime}=p_i-1$. Let $T^{\prime}$ be the Hilbert function associated to $P^{\prime}$. 
If $I$ is an element of the cell $\mathbb{V}(E_P)$ of $\G_T$, then $(I:x)$ is an element of the cell 
$\mathbb{V}(E_{P^{\prime}})$ of $\G_{T^{\prime}}$. 
In fact we have a morphism $\varphi : \mathbb{V}(E_P) \rightarrow \mathbb{V}(E_{P^{\prime}})$ defined by $I\mapsto (I:x)$ whose fiber is an affine 
space of dimension the number of difference-one hooks having their feet at $y^{m-1}$ (\cite[Proposition 2.6]{jy-1}).
\begin{lemma}\label{inductionlem}
Assume that $T=(1, \dots, d, t_d,0)$, and that $P$ is a partition having diagonal lengths $T$ and difference-one hook code $\mathfrak{Q}(P)=(h_1^{l_1}, \dots, h_n^{l_n})$. Set $s=d+1-t_d$. Suppose that $I$ is a generic ideal in the cell $\mathbb V(E_P)$ and let $\bar{I}=(I:x)$.

\begin{itemize}
\item[(a)] If $h_n=0$ then $\bar{I}\in \mathbb V(E_{\bar{P}})$ where $\bar{P}$ is the partition of diagonal lengths $\bar{T}=(1, \dots, d-1, t-1)$ and hook code $\mathfrak{Q}(\bar{P})=(h_1^{l_1}, \dots, h_n^{l_n-1}).$ Furthermore, in this case $\kappa(P)=\kappa(\bar{P})+1.$
\medskip

\item[(b)] If $h_n>0$ then $\bar{I}\in \mathbb V(E_{\bar{P}})$ where $\bar{P}$ is the partition of diagonal lengths $\bar{T}=(1, \dots, d-1, t)$ and  hook code $\mathfrak{Q}(\bar{P})=\big((h_1-1)^{l_1}, \dots, (h_n-1)^{l_n}\big).$ Furthermore, in this case 
$$\kappa(P)=\left\{\begin{array}{ll} \kappa(\bar{P}),&\mbox{ if }\kappa(\bar{P})\geq s \\ s,&\mbox{ if } \kappa(\bar{P})=s-1.\end{array}\right.$$

\end{itemize} 
\end{lemma}

\par

\begin{proof}

We note that the Ferrers diagram of $\bar{P}$ is always obtained from the Ferrers diagram of $P$ by removing the first column. Let $\mathcal{\bar{B}}=\{f_1, \dots, f_\kappa\}$ be a minimal set of generators for $\bar{I}$ (the $f_i$'s are ordered according to their leading monomials, from top to bottom in the Ferrers diagram).
\medskip

\noindent Part (a) If $h_n=0$, then $\mathcal{B}=\{xf_1, \dots, xf_\kappa, y^{d+1}\}$ is a minimal set of generators for $I$. Thus the equality in part (a) holds.
\medskip

\noindent Part (b). If $h_n>0$ then by definition $\bar{s}=s-1$. This in particular implies that in this case $\kappa(\bar{P})\geq s-1$. 

Assume that $h_n=1$. Then the leading term of 
$f_\kappa$ is $y^d$ (we can even set $f_\kappa=y^d$ here). Let $g$ be a generic polynomial with leading term $y^d$. Then 
$\mathcal{B}=\{xf_1, \dots, xf_{\kappa-1}, g\}$ is a minimal set of generators for $I$. Thus in this case $\kappa(P)=\kappa(\bar{P})$. We also note that since $\kappa(P)\geq s$ (Definition \ref{order_of_Tdef} and Lemma~\ref{kappaTlem})
the equality $\kappa(P)=\kappa(\bar{P})$ in particular implies that when $h_n=1$, we have $\kappa(\bar{P})\geq s$.

\medskip

Next,  assume that $h_n>1$. Suppose $\kappa(\bar{P})=s-1$ (this is the minimum value posssible for $\kappa(\bar{P})$). In this case, all the degree $d$ corner-monomials of $P^{\prime}$ have been kicked-off. After multiplication by $x$, these degree $d$ corner-monomials of $P^{\prime}$ become degree $d+1$ 
corner-monomials of $P$, so are kicked-off by $(xf_1, \dots, xf_\kappa)$ and therefore $\mathcal{B}=\{xf_1, \dots, xf_\kappa, g\}$, where $g$ is a generic polynomial with leading term $y^d$, is a minimal set of generators of the generic element of $\mathbb{V}(E_P)$, so 
$\kappa(P)=\kappa(\bar{P})+1$. \\
Now, suppose $\kappa(\bar{P})>s-1$. This means that there is at least one  degree-$d$ form in any minimal set of generators of a generic element of 
$\mathbb{V}(E_{P^{\prime}})$. So we have $\mathcal{\bar{B}}=\{f_1, \dots, f_{\kappa-s+1}, f_{\kappa-s+2}, \ldots, f_\kappa\}$, 
$\deg( f_{\kappa-s+1})=d$, $\deg( f_{\kappa-s+2})=\cdots = \deg(f_\kappa)=d-1$. Note that $f_\kappa$ has leading monomial $y^{d-1}$. Now, 
let $g=yf_\kappa+\lambda f_{\kappa-s+1}$ ($\lambda\neq 0$). Then $xg-y(xf_\kappa)=\lambda xf_{\kappa-s+1}$. Since $\lambda\neq 0$, this 
means that $xf_{\kappa-s+1}$ can be kicked off. If $\mathcal{\bar{B}}=\{f_1, \dots, f_{\kappa-s+1}, f_{\kappa-s+2}, \ldots, f_\kappa\}$ 
is a minimal set of generators of $(I:x)=\bar{I}$, then $\mathcal{B}=\{xf_1, \dots, xf_{\kappa-s}, xf_{\kappa-s+2}, \ldots, xf_\kappa,g\}$ 
is a minimal set of generators of $I$. Thus $\kappa(P)=\kappa(\bar{P})$.
\end{proof}

Recall that for a partition $P$ of diagonal lengths $T$, we denote the minimum number of generators for a generic element $I$ in the cell $\mathbb{V}(E_P)$ by $\kappa(P)$. Also note that, as discussed at the beginning of this section, if $T=(1, \dots, d, t_d,0)$ is a single-block Hilbert function, then a minimal system of generators for $I$ consists of $s=d+1-t_d$ generators of degree $d$ and $(\kappa(P)-s)$ generators of degree $d+1$. The following theorem provides an explicit formula for $\kappa(P)$ in the single-block case.

\begin{theorem}[The invariant $\kappa(P)$ for a single-block $T$]\label{kappathm}
Assume that $T=(1, \dots, d, t_d=t,0)$, set $s=d+1-t$ and let $P$ be a partition of diagonal lengths $T$ and difference-one hook code $\mathfrak{Q}(P)=(h_1^{l_1}, \dots, h_n^{l_n})$. For $k=1, \dots, n$, let $\tau_k=\sum\limits_{i=k}^nl_i-h_k.$ Then
\begin{equation}\label{kappaformula}\kappa(P)=s+\max\{t+1-s, 0, \tau_k\}_{k=1, \dots, n}.\end{equation}

\end{theorem}

\begin{proof}
We prove the theorem by induction on $d$. 
\bigskip

First assume that $d=2$.
\medskip

If $t=2$, then $s=1$. In this case there are three partitions of diagonal lengths $T$, namely
\begin{itemize}
\item[(i)] Partition $P=(3,2)$ with hook code $\mathfrak{Q}(P)=(1^2)$, $\kappa(P)=3$ and $$s+\max\{t+1-s,0, \tau_1\}=1+\max\{2,0, 2-1\}=3;$$

\item[(ii)] Partition $P=(3,1^2)$ with hook code $\mathfrak{Q}(P)=(1,0)$, $\kappa(P)=3$ and $$s+\max\{t+1-s,0, \tau_1, \tau_2\}=1+\max\{2,0, 2-1, 1-0\}=3;$$

\item[(iii)] Partition $P=(2^2,1)$ with hook code $\mathfrak{Q}(P)=(0^2)$, $\kappa(P)=3$ and $$s+\max\{t+1-s,0, \tau_1\}=1+\max\{2,0,2-0\}=3.$$
\end{itemize}
\medskip

On the other hand, if $t=1$, then $s=2$. In this case, there are three partitions of diagonal lengths $T$, namely

\begin{itemize}
\item[(i)] Partition $P=(3,1)$ with hook code $\mathfrak{Q}(P)=(2)$, $\kappa(P)=2$ and $$s+\max\{t+1-s,s, s+\tau_1\}=2+\max\{0,0,1-2\}=2;$$

\item[(ii)] Partition $P=(2^2)$ with hook code $\mathfrak{Q}(P)=(1)$, $\kappa(P)=3$ and $$s+\max\{t+1-s,0, \tau_1\}=2+\max\{0,0,1-1\}=2;$$

\item[(iii)] Partition $P=(2,1^2)$ with hook code $\mathfrak{Q}(P)=(0)$, $\kappa(P)=3$ and $$s+\max\{t+1-s,0, \tau_1\}=2+\max\{0,0,1-0\}=3.$$
\end{itemize}

This shows that the desired equality holds when $d=2$.
\bigskip

Now assume that $d>2$ and that Equation (\ref{kappaformula}) holds for any partition of diagonal lengths $(1, \dots, d',t,0)$ with $d'<d$. 

Suppose that $P$ is a partition of diagonal lengths $T=(1, \dots, d, t)$ and hook code  $\mathfrak{Q}(P)=(h_1^{l_1}, \dots, h_n^{l_n})$. Let $\bar{P}$ be the partition associated to $P$ defined in Lemma \ref{inductionlem}. Then by the inductive hypothesis Equation \ref{kappaformula} holds for $\kappa(\bar{P})$.

\medskip

{\bf Case 1.} Assume that $h_n=0$. Then $\bar{t}=t-1$, $\bar{s}=s$, and for $k=1, \dots, n$, we have $\bar{\tau}_k=\tau_k-1$. Thus
$$\begin{array}{ll}
\kappa(\bar{P})&=\bar{s}+\max\{\bar{t}+1-\bar{s}, 0, \bar{\tau}_k\}_{k=1, \dots, n}\\
&=s+\max\{t-s,0, \tau_k-1\}_{k=1, \dots, n}.
\end{array}$$

Since $\tau_n=l_n-h_n=l_n\geq1$, we have 

$$\begin{array}{ll}
\max\{t-s,0, \tau_k-1\}_{k=1, \dots, n}&=\max\{t-s,\tau_k-1\}_{k=1, \dots, n}\\ \\
&=\max\{t+1-s, \tau_k\}_{k=1, \dots, n}-1.
\end{array}$$

Thus using part (a) of Lemma \ref{inductionlem} we have

$$\begin{array}{ll}
\kappa(P)&=\kappa(\bar{P})+1\\
&s+\max\{t+1-s, \tau_k\}_{k=1, \dots, n}\\
&s+\max\{t+1-s,0, \tau_k\}_{k=1, \dots, n}.
\end{array}$$

\medskip

{\bf Case 2.} Assume that $h_n>0$. Then $\bar{t}=t$, $\bar{s}=s-1$, and for $k=1, \dots, n$, we have $\bar{\tau}_k=\tau_k+1$. By the inductive hypothesis

$$\begin{array}{ll}
\kappa(\bar{P})&=\bar{s}+\max\{\bar{t}+1-\bar{s}, 0, \bar{\tau}_k\}_{k=1, \dots, n}\\
&=s-1+\max\{t+1-s+1,0, \tau_k+1\}_{k=1, \dots, n}.
\end{array}$$

If $h_n=1$, then $\tau_n=l_n-1\geq 0$. Furthermore, if $\kappa(\bar{P})\geq s$ then $t+1\geq s$ or $\tau_k\geq 0$ for some $k$. In either of these cases, we have 
 $$
\max\{t+1-s+1,0, \tau_k+1\}_{k=1, \dots, n}=\max\{t+1-s+1, \tau_k+1\}_{k=1, \dots, n}.$$ Thus, using Lemma \ref{inductionlem}, we have 

$$\begin{array}{ll}
\kappa(P)&=\kappa(\bar{P})\\
&=s-1+\max\{t+1-s+1, \tau_k+1\}_{k=1, \dots, n}\\
&=s+\max\{t+1-s, \tau_k\}_{k=1, \dots, n}\\
&=s+\max\{t+1-s,0, \tau_k\}_{k=1, \dots, n}.
\end{array}$$

Finally, if $\kappa(\bar{P})=s-1$, then $t+1\leq s-1$ and $\tau_k+1\leq 0$, for all $k=1, \dots, n$. This in particular implies that in this case $s+\max\{t+1-s,0, \tau_k\}_{k=1, \dots, n}=s$.

By Lemma \ref{inductionlem}, we also have 

$$\begin{array}{ll}
\kappa(P)&=\kappa(\bar{P})+1\\
&=s-1+1\\
&=s
\end{array}$$
\end{proof}\par
Recall that a partition in $\mathcal P(T)$ is \emph{special} if $\kappa(P)>\kappa (T)$ from Equation~\eqref{kappaTeq}.
\begin{corollary}[Special partitions]\label{special1cor}
Assume that $P$ is a single-block partition.  Then $P$ is special if and only if
some $\tau_k$ from Theorem \ref{kappathm} satisfies $\tau_k>\delta$ where $\delta=\max\{t_d+1-s,0\}$.
\end{corollary}
\begin{proof} This follows from Equation \ref{kappaTeq} and Theorem \ref{kappathm}.
\end{proof}
\begin{remark}
We note that if at least one entry in the hook code of $P$ is zero, then $\tau_n=l_n>0$. Thus, in this case $\kappa(P)\geq s+1$. This in particular implies that in part (b) of Lemma \ref{inductionlem}, if $h_n=1$ then the hook code of $\bar{P}$ has a zero entry and therefore $\kappa(\bar{P})\geq \bar{s}+1=s$. 

\end{remark}

\subsection{Lattice path correspondence.}\label{latticecorrsec}
In this subsection we introduce a one-to-one correspondence between single-block partitions with a given Hilbert function $T=(1, 2, \ldots, d-t_d=t,0)$ and the north-east lattice paths from $(0,0)$ to $(s,t)$. This correspondence will in particular provide a straightforward geometric illustration of the statement of 
Theorem~\ref{kappathm}.

\begin{definition}\label{codetopathcorrdefn}
Let $P$ be a single-block partition of diagonal lengths $T=(1, \dots, d, t_d=t,0)$ and difference-one hook code $\mathfrak{Q}(P)=(h_1^{l_1}, \dots, h_n^{l_n})$. Recall that $s=d+1-t$. Let $\mathfrak{L}(P)$ be the NE lattice path form $(0,0)$ to $(s,t)$ represented by the word $$\mathfrak{L}(P)=E^{h_n}N^{l_n}E^{h_{n-1}-h_n}N^{l_{n-1}}\dots E^{h_1-h_2}N^{l_1}E^{s-h_1}.$$ 

In other words, to obtain $\mathfrak{L}(P)$ from $\mathfrak{Q}(P)$, we start at the origin in $\mathbb{Z}^2$, move to the right by $h_n$ steps, then move up by $l_n$ steps, then move to the right by $h_{n-1}-h_n$ steps and up by $l_{n-1}$ steps, etc. Note that we may have $h_n=0$ or $h_1=s$, and therefore the path may start or end with northward steps.
\end{definition}
\begin{lemma}
Let $T=(1, \dots, d, t_d=t,0)$ and $s=d+1-t$. The map sending a partition $P$ of diagonal lengths $T$ to the NE lattice path $\mathfrak{L}(P)$ defined in Definition~\ref{codetopathcorrdefn} is a 1-1 correspondence between the set of partitions of diagonal lengths $T$  and the set of NE lattice paths from $(0,0)$ to $(s,t)$.
\end{lemma}
\begin{proof}
By Theorem~\ref{PTtoQTthm}, the map sending each partition $P$ of diagonal lengths $T$ to its difference-one hook code $\mathfrak{Q}(P)$ is an isomorphism. Moreover, the following map from the set of NE lattice paths from $(0,0)$ to $(s,t)$ to the set of difference-one hook codes for partitions of diagonal lengths $T$ is the inverse of the map defined in Definition~\ref{codetopathcorrdefn}. 
Consider a NE lattice path from $(0,0)$ to $(s,t)$ given by a word $L=E^{e_r}N^{n_r}\ldots E^{e_1}N^{n_1}E^{e_0}$ where $e_r$ and $e_0$ are non-negative while the rest of $e_i$ and $n_i$'s are positive. Then the corresponding partition $P$ is the partition with diagonal lengths $T=(1, 2, \dots, s+t-1, t,0)$ and difference-one hook code 
$$\mathfrak{Q}(P)=\left(\left(\sum_{i=1}^re_i\right)^{n_1}, \dots, \left(\sum_{i=k}^re_i\right)^{n_k}, \dots, e_r^{n_r}\right).$$
\end{proof}

\begin{example}\label{NEpathcorrexample}
Consider the single-block partition $P=(5,4,2^2)$ of diagonal lengths $T=(1,2,3,4,3,0)$. Then $\mathfrak{Q}(P)=(2^2,1)$. Therefore $\mathfrak{L}(P)$ is the path corresponding to the word $ENENN$. Note that the first horizontal step corresponds to the entry 1 in the hook-code. It is then followed by one step up, corresponding to the multiplicity of the entry 1 in $\mathfrak{Q}(P)$. Then there is another step to the right which corresponds to the difference $2-1$ of the consecutive entries in $\mathfrak Q(P)$, followed by two steps up because of the multiplicity 2 of the entry 2 in $\mathfrak{Q}(P)$. See Figure~\ref{NEpathcorrfig} for a visualization of this, as well as similar path correspondences for two other partitions of the same diagonal lengths. 
\end{example}
\begin{figure}
\begin{center}
\begin{tikzpicture}[scale=.7]

    \draw[black]
    node at (2,0) {\underline{\underline{Partition $P$}}}
    node at (15,0) {\underline{\underline{NE lattice path}}};

    \draw[very thick] (0,-2)--(5,-2)--(5,-3)--(4,-3)--(4,-4)--(2,-4)--(2,-6)--(0,-6)--(0,-2);
    
	\draw[very thin] 	(0,-2)--(5,-2)
						(0,-3)--(4,-3)
						(0,-4)--(2,-4)
						(0,-5)--(2,-5)
						(1,-2)--(1,-6)
						(2,-2)--(2,-6)
						(3,-2)--(3,-4)
						(4,-2)--(4,-3);

	\draw[black]
	node at (0.5,-2.5) {$\bullet$}
	node at (2.5,-2.5) {$\bullet$}
	node at (0.5,-3.5) {$\bullet$}
	node at (2.5,-3.5) {$\bullet$}
	node at (0.5,-5.5) {$\bullet$};
	
    \draw[very thick, black, <->]
		(6,-4)--(11,-4);

    \draw[ultra thin]
    (14,-5)--(16,-5)--(16,-2)--(14,-2)--(14,-5);
    
    \draw[ultra thick] 
    (14,-5)--(15,-5)--(15,-4)--(16,-4)--(16,-2);
    
    \draw[black]
    node at (14,-5) {$\bullet$}
    node at (13.25,-5) {(0,0)}
    node at (15,-5) {$\bullet$}
    node at (15,-4) {$\bullet$}
    node at (16,-4) {$\bullet$}
    node at (16,-2) {$\bullet$}
    node at (16.75,-2) {(2,3)};

    \draw[thick, red]
    (13,-4)--(16,-1);
    \draw[red]
     node at (17.5, -1) {$y=x+\delta$};

    \draw[black]
    node at (2,-7) {$\mathfrak{Q}(P)=(2^2,1)$}
    node at (8.5,-7.5) {$\boxed{\kappa(P)=s+\delta=4}$}
    node at (15,-7) {$\mathfrak{L}(P)=ENENN$};


	\draw[very thick] (0,-11)--(4,-11)--(4,-13)--(2,-13)--(2,-15)--(1,-15)--(1,-16)--(0,-16)--(0,-11);
	
	\draw[very thin] 	(0,-12)--(4,-12)
						(0,-13)--(4,-13)
						(0,-14)--(2,-14)
						(0,-15)--(2,-15)
						(1,-11)--(1,-15)
						(2,-11)--(2,-13)
						(3,-11)--(3,-13);

	\draw[black]
	node at (2.5,-12.5) {$\bullet$}
	node at (0.5,-2.5) {$\bullet$}
	node at (2.5,-2.5) {$\bullet$}
	node at (0.5,-4.5) {$\bullet$};
	
    \draw[very thick, black, <->]
		(6,-14)--(11,-14);

    \draw[ultra thin]
    (14,-15)--(16,-15)--(16,-12)--(14,-12)--(14,-15);
    
    \draw[ultra thick] 
    (14,-15)--(14,-13)--(15,-13)--(15,-12)--(16,-12);
    
    \draw[black]
    node at (14,-15) {$\bullet$}
    node at (13.25,-15) {(0,0)}
    node at (14,-13) {$\bullet$}
    node at (15,-13) {$\bullet$}
    node at (15,-12) {$\bullet$}
    node at (16,-12) {$\bullet$}
    node at (16.75,-12) {(2,3)};
    
    \draw[thick, red]
    (13,-14)--(16,-11);
    \draw[red]
     node at (17.5, -11) {$y=x+\delta$};

    \draw[black]
    node at (2,-17) {$\mathfrak{Q}(P)=(1,0^2)$}
       node at (8.5,-17.5) {$\boxed{\kappa(P)=s+\delta=4}$}
    node at (15,-17) {$\mathfrak{L}(P)=NNENE$};  
    

    \draw[very thick] (0,-21)--(4,-21)--(4,-22)--(3,-22)--(3,-24)--(2,-24)--(2,-25)--(1,-25)--(1,-26)--(0,-26)--(0,-21);
    
    \draw[very thin] 	(0,-22)--(4,-22)
						(0,-23)--(3,-23)
						(0,-24)--(2,-24)
						(0,-25)--(2,-25)
						(1,-21)--(1,-25)
						(2,-21)--(2,-24)
						(3,-21)--(3,-23);
						
	\draw[very thick, black, <->]
		(6,-23)--(11,-23);						

    \draw[ultra thin]
    (14,-25)--(16,-25)--(16,-22)--(14,-22)--(14,-25);
    
    \draw[ultra thick] 
    (14,-25)--(14,-22)--(16,-22);
    
    \draw[black]
    node at (14,-25) {$\bullet$}
    node at (13.25,-25) {(0,0)}
    node at (14,-22) {$\bullet$}
    node at (16,-22) {$\bullet$}
    node at (16.75,-22) {(2,3)};
    
    \draw[thick, , dashed, red]
    (13,-24)--(16,-21);
    \draw[red]
     node at (17.5, -21) {$y=x+\delta$};
    
    \draw[thick, red]
    (13,-23)--(15,-21);
    \draw[red]
     node at (13.5, -21) {$y=x+3$};
   
    \draw[black]
    node at (2,-27) {$\mathfrak{Q}(P)=(0^3)$}
       node at (8.5,-27.5) {$\boxed{\kappa(P)=s+3=5}$}
    node at (15,-27) {$\mathfrak{L}(P)=NNNEE$};

\end{tikzpicture}
\end{center}
\caption{NE lattice path correspondence for three different partitions of diagonal lengths $T=(1,2,3,4,3,0)$. See Example~\ref{NEpathcorrexample}.}\label{NEpathcorrfig}
\end{figure}

The key observation about the lattice path correspondence defined above is that $\max\{\tau_k\,|\, k=1, \ldots, n\}$ is in fact the maximum value $b$ for which the line $y=x+b$ intersects $\mathfrak{L}(P)$, excluding the end points.

Therefore, in order to find $\kappa(P)$ for a partition $P$, we consider the corresponding lattice path $\mathfrak{L}(P)$. If $\mathfrak{L}(P)$ does not cross the line $y=x+\delta$, where $\delta=\max\{0,t+1-s\}$, then $\kappa(P)=s+\delta$. Otherwise,
$\kappa(P)=s+b$, where $b$ is the largest integer such that $\mathfrak{L}(P)$ intersects the line $y=x+b$. See Figure~\ref{NEpathcorrfig}.

\begin{remark}\label{kappafrompathalg}
In general, the following steps lead to a visual and relatively straightforward way of finding $\kappa(P)$ for a single-block partition $P$ of diagonal lengths $T=(1, \dots, d, t_d=t,0)$ and difference-one hook code $\mathfrak{Q}(P)=(h_1^{l_1}, \dots, h_n^{l_n})$. See Figure~\ref{kappafrompathgeneral}.
\begin{itemize}
    \item Consider the NE lattice path $\mathfrak{L}(P)$ associated with the hook code $\mathfrak{Q}(P)$. Note that this is a path form $(0,0)$  to $(s,t)$ where the eastward movements are determined by entries of the hook code ($h_k$'s) while the northward steps are determined by the multiplicities of the entries of the hook code ($l_k$'s). The corners of the path are at $\left(h_k, \sum\limits_{i=k}^n l_i\right)$ for $k=1, \ldots, n$.
    \item Consider the set $L$ of all lines of slope 1 that intersect $\mathfrak{L}(P)$. If $y=x+\delta$ is not in the set, add it to $L$. Intersect all lines in $L$ with the vertical line $x=s$ and consider the most northerly intersection point. The $y$-coordinate of this point is $\kappa(P)$. 
\end{itemize}
\end{remark} 

 \begin{figure}
 \hspace{-0.7 in}
\begin{tikzpicture}[scale=.7, transform shape]

    \draw[black]
    node at (1,1) {\Large{\underline{{\bf Case 1}}. Assume that $s\leq t$. Then $\delta=t+1-s>0$.}};
    
    \draw[ultra thin]
    (0,-5)--(2,-5)--(2,-2)--(0,-2)--(0,-5);
    
    \draw[ultra thin, dashed]
    (2,-2)--(2,0);
    
    \draw[ultra thick] 
    (0,-5)--(1,-5)--(1,-4)--(2,-4)--(2,-2);
    
    \draw[black]
    node at (0,-5) {$\bullet$}
    node at (-0.85,-5) {(0,0)}
    node at (2,-2) {$\bullet$}
    node at (2.75,-2) {$(s,t)$};

    \draw[thick, red, dashed]
    (-0.5,-5.5)--(2.5,-2.5)
    (-.5,-4.5)--(2.5,-1.5);
     \draw[thick, red]
    (-1,-4)--(2.5,-0.5);
    \draw[red]
     node at (-2, -3.5) {$y=x+\delta$};

    \draw[black]
    node at (2,-1) {$\bullet$};
    
    \draw[black]
    node at (3.5,-1) {$(s,\kappa(P))$};
    
    \draw[black]
    node at (1,-8) {\Large{$\kappa(P)=s+\delta$}};
    \draw[black]
    node at (1,-9) {\Large{(Here $s=2$, $t=3$ and $\delta=2.$)}};

    \draw[ultra thin]
    (14,-5)--(16,-5)--(16,-2)--(14,-2)--(14,-5);
    
    \draw[ultra thin, dashed]
    (16,-2)--(16,0);
    
    \draw[ultra thick] 
    (14,-5)--(14,-2.5)--(14.5,-2.5)--(14.5,-2)--(16,-2);

    \draw[black]
    node at (14,-5) {$\bullet$}
    node at (13.125,-5) {(0,0)}
    node at (16,-2) {$\bullet$}
    node at (16.75,-2.25) {$(s,t)$};

    \draw[thick, red, dashed]
    (13.5, -4)--(16.5,-1);
    \draw[red]
     node at (12.5, -4.125) {$y=x+\delta$};
     
    \draw[red]
     node at (12.5, -2.75) {$y=x+b$};     

    \draw[thick, red]
    (13.25,-3.25)--(16.5,0);

    \draw[black]
    node at (16,-0.5) {$\bullet$};
    
    \draw[black]
    node at (17.5,-0.25) {$(s,\kappa(P))$};
    
    \draw[black]
    node at (15,-8) {\Large{$\kappa(P)=s+b$}};

    \draw[black]
    node at (15,-9) {\Large{(Here $s=4$, $t=6$, $\delta=3$ and $b=5$)}};

    \draw[black]
    node at (0,-14) {\Large{\underline{{\bf Case 2}}. Assume that $s>t$. Then $\delta=0$.}};
    
    \draw[ultra thin]
    (0,-20)--(3,-20)--(3,-18)--(0,-18)--(0,-20);
    
    \draw[ultra thick] 
    (0,-20)--(2.5,-20)--(2.5,-19.5)--(3,-19.5)--(3,-18);
    
    \draw[black]
    node at (0,-20) {$\bullet$}
    node at (-0.85,-20) {(0,0)}
    node at (3,-18) {$\bullet$}
    node at (4,-18) {$(s,t)$};

    \draw[ultra thin, dashed]
    (3,-16.5)--(3,-18);
    
    \draw[black]
    node at (3, -17){$\bullet$}
    node at (4.5,-17) {$(s,\kappa(P))$};
    
    \draw[thick,red]
    (-0.5,-20.5)--(3.5,-16.5);
    \draw[red]
     node at (1, -17.5) {$y=x+\delta$};
    
    \draw[thick,red, dashed]
    (0.5,-20.5)--(3.5,-17.5)
    (1.5,-20.5)--(3.5,-18.5);
    
    \draw[black]
    node at (1,-22) {\Large{$\kappa(P)=s+\delta=s$}};
    \draw[black]
    node at (1,-23) {\Large{(Here $s=6$ and $t=4$.)}};

    \draw[ultra thin]
    (14,-20)--(17,-20)--(17,-18)--(14,-18)--(14,-20);
    
    \draw[ultra thick] 
    (14,-20)--(14.5,-20)--(14.5,-18.5)--(15.5,-18.5)--(15.5,-18.25)--(17,-18.25)--(17,-18);
    
    \draw[black]
    node at (14,-20) {$\bullet$}
    node at (14,-20.5) {(0,0)}
    node at (17,-18) {$\bullet$}
    node at (18,-18) {$(s,t)$};

    \draw[thick, red, dashed]
    (13.25,-20.625)--(17.5,-16.5)
    ;
    \draw[red]
     node at (12.25, -20.5) {$y=x+\delta$};
     \draw[red]
     node at (13, -19) {$y=x+b$};
    \draw[thick, red]
    (13.5,-19.5)--(17.5,-15.5);

    \draw[black]
    node at (15,-22) {\Large{$\kappa(P)=s+b$}};
    \draw[black]
    node at (15,-23) {\Large{(Here $s=12$, $t=8$ and $b=4$.)}};
    
    \draw[ultra thin, dashed]
    (17,-15)--(17,-18);
    
    \draw[black]
    node at (17, -16){$\bullet$}
    node at (18.5,-16) {$(s,\kappa(P))$};
    
    \draw[thick, black]
    (-6,-24)--(-6,2)--(20,2)--(20,-24)--(-6,-24);
\end{tikzpicture}

\caption{Finding $\kappa(P)$ from $\mathfrak{L}(P)$. See Remark~\ref{kappafrompathalg}.}\label{kappafrompathgeneral}
\end{figure}

\section{Partitions in $\mathcal P(T)$ having a given number of generators, for single-block $T$.}\label{countSingleSection}
We begin with a result counting the total number of partitions having diagonal lengths a single-block Hilbert function $T$. We then in Theorem \ref{countingthm} count those associated to a given generic number of generators $\kappa(P)$. Throughout the section $T$ will be a single-block Hilbert function $T=(1,2,\ldots, d,t_d=t,0)$ of Equation~\eqref{Tsingleeqn}, we let $s=d+1-t$ and we set $\delta=\max\{t+1-s,0\}=\max\{2t-d,0\}$. 

\begin{lemma}\label{count1blockcor}
The number of partitions having the single-block diagonal lengths $T=(1,2,\ldots, d, t_d=t)$ satisfies
\begin{equation}\label{count1blockeq}
\# \mathcal P(T)=\binom{s+t}{t}.
\end{equation}
\end{lemma}
\begin{proof} By Theorem \ref{PTtoQTthm}, $\# \mathcal P(T)$ counts the total number of partitions whose Ferrers diagram 
can be placed in a $t\times s$ box $\mathfrak B_d(T)$ or, equivalently, the number of lattice paths from $(0,0)$ to $(s,t)$, which satisfies \eqref{count1blockeq}. 
\end{proof}

\par
By Theorem \ref{kappathm}, the number of generators for a generic ideal in the cell $\mathbb V(E_P)$ is $\kappa(P)=s+\max\{\delta, \tau_k\}_{k=1, \dots, n}.$
In particular, for all partitions $P$ of diagonal lengths $T$, we have $$\kappa(T)=s+\delta\leq \kappa(P) \leq s+t.$$

\begin{theorem}[Number of special partitions of diagonal lengths $T$]\label{countingthm}

Let $T=(1, \dots, d, t,0)$, $s=d+1-t$, and $\delta=\max\{t+1-s,0\}$. Assume that $k$ is an integer such that $s+\delta< k \leq s+t$. Then the number of partitions $P$ of diagonal lengths $T$ and $\kappa(P)\geq k$ is 
\begin{equation}\label{countsinglegekeq}
\binom{s+t}{k}.
\end{equation}
In particular the number of special partitions of diagonal lengths $T$ is 
\begin{equation}\label{numberspecial1eq}
\binom{s+t}{s+\delta+1}=\binom{s+t}{\min\{s-1,t\}}.
\end{equation}
And the number of non-special partitions of diagonal lengths $T$ is
\begin{equation}\label{numbernonspecial1eq}
\binom{s+t}{s}-\binom{s+t}{s+\delta+1}.
\end{equation}
\end{theorem}

\begin{proof}

Consider an integer $k$ such that $s+\delta<k\leq s+t$. Let $\gamma=k-s$. Then $\delta <\gamma \leq t$. Using the correspondence established in Section~\ref{latticecorrsec}, partitions $P$ of diagonal lengths $T$ and $\kappa(P)\geq k$ correspond to paths from $(0,0)$ to $(s,t)$ intersecting the line $y=x+\gamma$. Note that the inequality $\delta <\gamma$ implies that the line $y=x+\gamma$ is above the line $y=x+\delta$. 

To count the number of paths from $(0,0)$ to $(s,t)$ intersecting the line $y=x+\gamma$, we first count the number of paths from $(0,0)$ to $(s,t)$ that stay weakly below the line $y=x+\gamma-1$. We will then subtract this number from the total number of paths from $(0,0)$ to $(s,t)$, which is $\binom{s+t}{s}$.
\par 
Since $\delta <\gamma$ and $\delta=\max\{0, t+1-s\}$, both endpoints $(0,0)$ and $(s,t)$ are weakly below $y=x+\gamma-1$. Therefore the set of paths weakly below the line $y=x+\gamma-1$ is non empty. 
\par
Using a simple vertical translation by $\gamma-1$ units, it is clear that the number of paths from $(0,0)$ to $(s,t)$ that are weakly below the line $y=x+\gamma-1$ is the same as the number of paths from $(0,-(\gamma-1))$ to $(s,t-(\gamma-1))$ that are weakly bellow the line $y=x$. By Theorem 10.3.1 of \cite{Kc}, the number of such paths is 
$\displaystyle{\binom{s+t}{s}-\binom{s+t}{s+\gamma}}$. 
This completes the proof of the theorem.
\end{proof}
\begin{cor}\label{singlecountcor}
Let $T=(1, \dots, d, t_d=t,0)$, $s=d+1-t$, and $\delta=\max\{t+1-s,0\}$. For a positive integer $k$, we define $\mu(T,k)$ to be the number of partitions $P$ with diagonal lengths $T$ and $\kappa(P)=k$. Then 
\[
\mu(T,k)=
\left\{
\begin{array}{ll}
    \binom{s+t}{s}-\binom{s+t}{s+\delta+1},   & \mbox{ if } k=s+\delta \mbox{ (non-special $P$)}, \\\\
    \binom{s+t}{k}-\binom{s+t}{k+1},           & \mbox{ if } s+\delta<k\leq s+t,\\\\
    0,                                         & \mbox{ otherwise. }
\end{array}
\right.
\]
\end{cor}
\begin{remark}

 
 We note that for $k=s+\delta$,  the number $\mu(T,k)$ is the coefficient of the degree $k=s+\delta$ term in $(1+z)^{s+t}\left(z^\delta-\frac{1}{z}\right)$ while 
for $s+\delta<k< s+t$, the number $\mu(T,k)$ is the same as the coefficient of the degree $k$ term in $(1+z)^{s+t}\left(1-\frac{1}{z}\right)$.
\end{remark}

\begin{remark}
In \cite[Section 2C]{IY} it was shown that for any single-block Hilbert function $T$, there will be a unique minimal finite set of special partitions of diagonal lengths $T$, such that any special partition is in the closure of the minimal set. The following example shows that the special cells do not form an irreducible subfamily of $\G_T$.
\end{remark}
\begin{example}[Single-block table]\label{table2ex}
Let $T=(1,2,3,4,2,0)$, then $t=2,s=3$ and  $\mathfrak B(T)=({\mathfrak B}_4(T))=((2\times 3)),$ and there are $\binom{5}{2}=10$ partitions of diagonal lengths $T$.  We give Figure \ref{420table} for these, specifying the hook code, and $\kappa(P)$ for each. Here $\delta(T)=\max\{0,s+1-t\}=0$ and $ \kappa(T)=3$.  We have placed
conjugate partitions in symmetric positions from the center line; the two middle partitions of hook codes $(3,0)$ and $(2,1)$ are self-conjugate. Note also that the conjugate partition $P^\vee$ has the complementary hook code in ${\mathfrak B}_4(T)=(3,3)$. \par
Figure~\ref{specializefig} gives the specialization diagram for  $\mathcal P(T)$, corresponding to inclusion of the Ferrers diagrams for the hook codes $\mathfrak h_4(P)$ (on the left).  \par
We see from the table that the cells in 
$\kappa(P)\ge 4$ are the union of the closures of cells having hook codes  $ (1,1)$ and $(3,0)$: $ \kappa(P) = 4$ includes also the cells with hook codes $(2,0)$ and $ (1,0),$ while the cell with hook code $(0,0)$ is the unique with $\kappa(P)=5$ (these cells are colored red/blue on the left of Figure~\ref{specializefig}.)  Thus, the subvariety of cells corresponding to special partitions is here the union of two irreducible components, of dimensions three (closure of $(3,0)$) and two (closure of $(1,1)$), respectively.  \par

\end{example}
\begin{figure}
\begin{center}$
\begin{array}{|c|c|c|c|c|c|}
\hline
P&\mathfrak h_4&\tau_1&\tau_2&\kappa(P)\\
\hline\hline
(5,4,2,1)&(3,3)&-1&&3\\
\hline
(5,3,3,1)&(3,2)&-1&-1&3\\
\hline
( 5,3,2,2)&(3,1)&0&-1&3\\
\hline
(4,4,3,1)&(2,2)&0&&3\\
\hline\hline
(5,3,2,1,1)&(3,0)&-2&1&4\\
\hline
(4,4,2,2)&(2,1)&0&0&3\\
\hline\hline
(4,3,3,2)&(1,1)&1&0&4\\
\hline
(4,4,2,1,1)&(2,0)&0&1&4\\
\hline
(4,3,3,1,1)&(1,0)&1&1&4\\
\hline
(4,3,2,2,1)&(0,0)&2&&5\\
\hline
\end{array}$
\end{center}
\caption{Table of $\mathcal P(T), T=(1,2,3,4,2,0)$. See Example \ref{table2ex}.}\label{420table}
\end{figure}

\footnotesize
\begin{figure}
		$\boxed{\xymatrix{&\underline{\phantom{a}\mathfrak h_4(P)\phantom{a}}&\\&{(3,3)}\ar[d]& \\
		&{(3,2)}\ar[ld]\ar[rd]&\\
		(3,1)\ar[d] \ar[rrd]&&(2,2)\ar[d]\\
		{\textcolor{red}{(3,0)}}\ar[d]&&(2,1)\ar[d] \ar[lld]\\
		{\textcolor{red}{(2,0)}}\ar[rd]&&{\textcolor{red}{(1,1)}}\ar[ld]\\
		&{\textcolor{red}{(1,0)}}\ar[d]&\\
		&{\textcolor{blue}{(0,0)}}&}$
			\qquad\quad $\xymatrix{&\underline{\phantom{a}P\phantom{a}}&\\&{(5,4,2,1)}\ar[d]& \\
		&{(5,3,3,1)}\ar[ld]\ar[rd]&\\
		(5,3,2,2)\ar[d] \ar[rrd]&&(4,4,3,1)\ar[d]\\
		(5,3,2,1,1)\ar[d]&&(4,4,2,2)\ar[d] \ar[lld]\\
		(4,4,2,1,1)\ar[rd]&&(4,3,3,2)\ar[ld]\\
		&(4,3,3,1,1)\ar[d]&\\
		&(4,3,2,2,1)&}}$
				\caption{Specialization diagram for $\mathcal P(T), T=(1,2,3,4,2,0)$. See Example \ref{table2ex}.}\label{specializefig}
			\end{figure}
\normalsize

\section{Number of generators for multiblock partitions.}\label{numGensMultiSection}
Throughout this section, $T=(1, \dots, d, t_d, \dots, t_{\sf j}, 0)$, and $P$ is a partition  lengths $T$ and difference-one hook code $\mathfrak{Q}(P)=(\mathfrak{h}_d, \dots , \mathfrak{h}_{\sf j})$.

Recall from Equation \eqref{HFdecomposition} that for $i=d, \dots, {\sf j}$, we set $\delta_i=t_{i-1}-t_i, t_{d-1}:=d$ and $$T_i=(1, \dots, \delta_i+\delta_{i+1}, \delta_{i+1}, 0).$$ 
As we saw in Definition \ref{Pidef1} and Proposition \ref{hookP(i)lem}, a partition $P\in \mathcal P(T)$ can be decomposed into single-block ``component'' partitions $P_i$. For $i=d, \dots, {\sf j}$, the partition $P_i$ has diagonal lengths $T_i$ and difference-one hook code $\mathfrak{h}_i$. We note that by construction the hook code for $P_i$ is $\mathfrak{h}_i$: however, the hand-degree for the hooks in $P_i$ is $\delta_i+\delta_{i+1}$, and they correspond to hooks in $P$ of hand-degree $i$. We showed in Theorem~\ref{projectionthm} that the cells $\mathbb V(E_P)$ are naturally the product of the corresponding cells  $\mathbb V(E_{P_i})$.
\par 
In this section we count the minimum number of generators for ideals in the cell associated to an arbitrary partition. The main results of this section are Theorem~\ref{extra_gens_in_given_degree} and Theorem \ref{componentTheorem}. Theorem~\ref{extra_gens_in_given_degree} specifies the number $\beta_{i,0}(P)$ of degree-$i$ generators of an ideal $I$ defining a generic element of $\mathbb V(E_P)$. In Theorem \ref{componentTheorem} we provide a formula for $\kappa(P)$ in terms of the $\kappa(P_i)$ of single-block components of $P$. The value $\kappa(P_i)$ for a single-block partition was determined in Theorem~\ref{kappathm}.


The $i$-th block $\mathfrak{h}_i(P)$ of the hook code $\mathfrak{Q}(P)$ is a key element of our statements. So we first give a formula for $\kappa(P_i)$ when $\mathfrak{h}_i(P)$ is empty.  Recall that the case $\mathfrak{h}_i=\emptyset$ occurs when $t_i=t_{i+1}$ and in this case, we can think of $P_i$ as the basic triangle $\Delta_{\delta_i}$ (Remark \ref{empty_hook_code_block_partition}).
\begin{proposition}\label{kappa-of-empty-bloc}
If the block $\mathfrak{h}_i(P)$ is empty, then $\kappa(P_i)=\delta_i+1$. 
\end{proposition}
\begin{proof}
Recall that the case $\mathfrak{h}_i=\emptyset$ occurs when $t_i=t_{i+1}$ and in this case, we can think of $P_i$ as the basic triangle 
$\Delta_{\delta_i}$ (see Remark \ref{empty_hook_code_block_partition}). It is then clear that the monomial ideal whose cobasis is the basic triangle 
$\Delta_{\delta_i}$ has exactly $\delta_i+1$ generators, the degree $\delta_i$ monomials. Note that the basic triangle 
$\Delta_{\delta_i}$ is empty when $\delta_i=0$.
\end{proof}
\begin{example}\label{kappa-of-empty-bloc-example}
Consider the partition $P=(8,7,4,2,1,1)$ of diagonal lengths $T=(1,2,3,4,5,4,2,2,0)$. The degree 6 block $\mathfrak{h}_6(P)$ is empty ($t_6=t_7=2$). 
$\kappa(P_6)=\delta_6+1=3$ and this number corresponds to the three border monomials $(x^4y^2,x^3y^3,y^6)$ of $E_P$.
\end{example}

From now on, we will assume that the $i$-th block $\mathfrak{h}_i(P)$ of the hook code $\mathfrak{Q}(P)$ is not empty. 
For $i=d, \dots, {\sf j}$, the degree $i$ block $\mathfrak{h}_i$ in the hook code of $P$ can be written as
\begin{equation}\label{hookcomponent}
\mathfrak{h}_i=\left(h_{i,1}^{l_{i,1}}, \dots , h_{i,k}^{l_{i,k}}, \dots, h_{i,n_{i}}^{l_{i,n_{i}}}\right)
\end{equation}

where $\sum\limits_{k=1}^{n_i}l_{i,k}=\delta_{i+1}$ and
$
\delta_i+1\geq h_{i,1}>h_{i,2}>\cdots >h_{i,n_i}\geq 0.
$
\par

In the following, we define two sequences of integers $R(P)_i$ and $G(P)_i$ for single-block components $P_i$ of $P$ that allow us to count degree $i+1$ relations and degree $i+1$ generators of the monomial ideal $E_P$ associated to $P$.\\

 \begin{definition}[Sequences of degree $i+1$ relations and generators]\label{degree_i+1_rel_gen_sequences}  \
 Denote by $R_{i,k}$ the set of degree $i+1$ relations below the $\left(l_{i,1}+ \dots + l_{i,k}\right)$-th degree $i$ hand of $P$. 
 Then we let $R(P)_i=\left(r_{i,1}, \dots , r_{i,n_{i}} \right)$ where $r_{i,k}$ counts the number of  elements of  $R_{i,k}$. 

\medskip
\noindent
Also, let $G_{i,1}$ be the set of degree $i+1$ corner-monomials of $E_P$ and 
 for $2\leq k\leq n_i$, denote by $G_{i,k}$ the set of degree $i+1$ corner-monomials of $E_P$ below the 
 $\left(l_{i,1}+ \dots + l_{i,k-1}+1\right)$-th degree $i$ hand of $P$. Then $G(P)_i=\left(g_{i,1}, \dots , g_{i,n_{i}} \right)$ 
 where $g_{i,k}$ counts the number of  elements of  $G_{i,k}$. 
 \end{definition}
 
 \begin{remark}[Chains of degree $i+1$ relations and generators]\label{chains_of_degree_i+1_rel_gen}
 Note that using the above notation, we have  the following sequences of inclusions. 
 $$
 \begin{array}{l}
 R_{i,n_i}\subset R_{i,n_i-1} \subset \cdots \subset R_{i,k+1}\subset R_{i,k} \subset \cdots \subset R_{i,1}\\
 G_{i,n_i}\subset G_{i,n_i-1} \subset \cdots \subset G_{i,k+1}\subset G_{i,k} \subset \cdots \subset G_{i,1}
 .
 \end{array}
 $$
 So by definition, $R(P)_i=\left(r_{i,1}, \dots , r_{i,n_{i}} \right)$ and $G(P)_i=\left(g_{i,1}, \dots , g_{i,n_{i}} \right)$ are non-increasing sequences. 

\end{remark}

 \bigskip
 \noindent
 We now use the hook code $\mathfrak{h}_i=\left(h_{i,1}^{l_{i,1}}, \dots,  h_{i,k}^{l_{i,k}}, \dots , h_{i,n_{i}}^{l_{i,n_{i}}}\right)$ to give formulas for the integers $r_{i,k}$ and $g_{i,k}$ of Definition \ref{degree_i+1_rel_gen_sequences}. 
 
 \begin{observation}\label{degree_i+1_rel_gen_observation}
 By definition of the difference-one hook code, for any integer $1\leq k\leq n_i$, there are $(h_{i,k}-h_{i,k+1})$ 
 \textsl{horizontal-border monomials} below the $\left(l_{i,1}+ \dots + l_{i,k}\right)$-th degree $i$ hand and above the 
 $\left(l_{i,1}+ \dots + l_{i,k}+1\right)$-th degree $i$ hand of $P$. 
 Let $\left(M_{i,1}, \dots ,M_{i,h_{i,k}-h_{i,k+1}}\right)$ be the list of these monomials, ordered from top to bottom. 
 Then there is exactly one linear (occuring in degree $(i+1)$) relation between any two consecutive monomials $M_{i,m}$ and $M_{i,m+1}$ 
 ($1\leq m\leq h_{i,k}-h_{i,k+1}$). So, we have exactly $(h_{i,k}-h_{i,k+1}-1)$ linear relations below the $\left(l_{i,1}+ \dots + l_{i,k}\right)$-th degree-$i$ hand and above the $\left(l_{i,1}+ \dots + l_{i,k}+1\right)$-th degree-$i$ hand of $P$. \\
 Also we observe that each degree-$i$ hand of $P$ is just left of a degree $(i+1)$ \textsl{vertical-border monomial} of $E_P$. 
 We use this property to count the degree $(i+1)$ corner-monomials of $E_P$. 
 \end{observation}

 \begin{example}\label{degree_i+1_rel_gen_observation_ex}
 Consider the partition $P=(12^2,11,8^2,7,6^2,3^2,2^2)$ (see Figure \ref{degree_i+1_rel_gen_observation_fig}) whose Hilbert function is 
 $T=(1,2,\cdots, 10,11,10,4)$. Here we have $\mathfrak{h}_{12}=\left(6^2,3,1\right)$. So $h_{12,1}=6$, $h_{12,2}=3$, $h_{12,3}=1$, $l_{12,1}=2$, 
 $l_{12,2}=1$ and $l_{12,3}=1$. There are $(h_{12,1}-h_{12,2})=3$ \textsl{horizontal-border monomials} that are below the second hand monomial \textcolor{red}{$x^{10}y^2$} and above the third ($l_{12,1}+1=3$) hand monomial \textcolor{red}{$x^{5}y^7$}. Also we have exactly two linear relations that are below \textcolor{red}{$x^{10}y^2$} and above \textcolor{red}{$x^{5}y^7$}.
 \end{example} 
 \begin{figure}[!h]
\begin{center}
	\begin{tikzpicture}[scale=.8]
	\draw[dashed,very thin] 	(0,0)--(12,0)
						(0,-1)--(12,-1)
						(0,-2)--(11,-2)
						(0,-3)--(8,-3)
						(0,-4)--(8,-4)
						(0,-5)--(7,-5)
						(0,-6)--(6,-6)
						(0,-7)--(6,-7)
						(0,-8)--(3,-8)
						(0,-9)--(3,-9)
						(0,-10)--(2,-10)
						(0,-11)--(2,-11)
						(0,0)--(0,-12)
						(1,0)--(1,-12)
						(2,0)--(2,-10)
						(3,0)--(3,-8)
						(4,0)--(4,-8)
						(5,0)--(5,-8)
						(6,0)--(6,-6)
						(7,0)--(7,-5)
						(8,0)--(8,-3)
						(9,0)--(9,-3)
						(10,0)--(10,-3)
						(11,0)--(11,-2);
	\draw[thin, dashed,red,fill=yellow!10] (11,-1) rectangle (12,-2)
							(10,-2) rectangle (11,-3)
							(5,-7) rectangle (6,-8)
							(1,-11) rectangle (2,-12);
	\draw[thin, blue,fill=black!5,rounded corners=3mm] (9,-3) rectangle (10,-4)
						     (7,-5) rectangle (8,-6)
						     (6,-6) rectangle (7,-7)
						     (4,-8) rectangle (5,-9)
						     (2,-10) rectangle (3,-11)
						     (0,-12) rectangle (1,-13);
	\draw[red] node at (11.5,-1.5) {\footnotesize{$x^{11}y$}}
			  node at (10.5,-2.5) {\footnotesize{$x^{10}y^2$}}
			  node at (5.5,-7.5) {\footnotesize{$x^5y^7$}}
			  node at (1.5,-11.5) {\footnotesize{$xy^{11}$}};
	\draw[blue] node at (9.5,-3.5) {\footnotesize{$x^9y^3$}}
			    node at (7.5,-5.5) {\footnotesize{$x^7y^5$}}
			    node at (6.5,-6.5) {\footnotesize{$x^6y^6$}}
			    node at (4.5,-8.5) {\footnotesize{$x^4y^8$}}
			    node at (2.5,-10.5) {\footnotesize{$x^2y^{10}$}}
			    node at (0.5,-12.5) {\footnotesize{$y^{12}$}};
	\draw[very thick] (12,0)--(12,-2)--(11,-2)--(11,-3)--(8,-3)--(8,-5)--(7,-5)--(7,-6)--(6,-6)--(6,-8)--(3,-8)--(3,-10)--(2,-10)--(2,-12)--(0,-12);
         \end{tikzpicture}
         \caption{Ferrers diagram of $P=(12,11,8^2,7,6^2,3^2,2^2)$: \textsl{horizontal-border monomials} are marked in blue and 
         \textsl{hand monomials} are marked in red (Example \ref{degree_i+1_rel_gen_observation_ex}).}\label{degree_i+1_rel_gen_observation_fig}
\end{center}
\end{figure}

\begin{lemma}[Counting degree $i+1$ relations and corner-monomials of $E_P$]\label{degree_i+1_rel_gen_formulas} 
Let $\mathfrak{h}_i=\left(h_{i,1}^{l_{i,1}}, \dots,  h_{i,k}^{l_{i,k}}, \dots , h_{i,n_{i}}^{l_{i,n_{i}}}\right)$ be the $i$-th block of the hook code $\mathfrak{Q}(P)$. \\
 Then the sequences $R(P)_i=\left(r_{i,1}, \dots , r_{i,n_{i}} \right)$ and $G(P)_i=\left(g_{i,1}, \dots , g_{i,n_{i}} \right)$ of Definition \ref{degree_i+1_rel_gen_sequences} are given by the following numbers. 
 \begin{enumerate}[(a)]
 \item If $h_{i,n_i}>0$, then 
 \begin{itemize}
 \item[$\bullet$]
 $
 g_{i,1}=\left\{
 \begin{array}{lr}
 \delta_{i+1}-(n_i-1) & \text{if} \ h_{i,1}=\delta_i +1\\
 \delta_{i+1}-n_i & \text{if} \ h_{i,1}<\delta_i +1
 \end{array}
 \right.
 $
 \item[$\bullet$]
 $
 \displaystyle{
 g_{i,k}=\sum_{j=k}^{j=n_i}l_{i,j}-(n_i+1-k)
 }
 $, for $2\leq k\leq n_i$
 \item[$\bullet$]
 $
 r_{i,k}=h_{i,k}-(n_i+1-k) 
 $, 
 for $1\leq k\leq n_i$.
 \end{itemize}
 \item If $h_{i,n_i}=0$, then 
 \begin{itemize}
 \item[$\bullet$]
 $
 g_{i,1}=\left\{
 \begin{array}{lr}
 \delta_{i+1}-(n_i-2) & \text{if} \ h_{i,1}=\delta_i +1\\
 \delta_{i+1}-(n_i-1) & \text{if} \ h_{i,1}<\delta_i +1
 \end{array}
 \right.
 $
 \item[$\bullet$]
 $
 \displaystyle{
 g_{i,k}=\sum_{j=k}^{j=n_i}l_{i,j}-(n_i-k)
 }
 $, for $2\leq k\leq n_i$
 \item[$\bullet$]
 $
 r_{i,k}=h_{i,k}-(n_i-k)
 $
 for $1\leq k\leq n_i$.
 \end{itemize}
\end{enumerate}
\end{lemma}

\begin{proof}
Note that by definition of the difference-one hook code, we have
$\delta_i+1\geq h_{i,1}> \dots > h_{i,n_{i}}\geq 0$  and $\displaystyle{\sum_{k=1}^{k=n_i}l_{i,k}=\delta_{i+1}}$. \\
The numbers given in the Lemma come directly from Observation \ref{degree_i+1_rel_gen_observation}. 
\begin{enumerate}[(a)]
\item If $h_{i,n_i}>0$, then we have 
\begin{itemize}
\item[$\bullet$] 
$\displaystyle{g_{i,1}=c_{i,1}+(l_{i,2}-1)+\cdots + (l_{i,n_i}-1)=c_{i,1}+\sum_{k=2}^{k=n_i}l_{i,k}-(n_i-1)}$ where \\
 $
 \displaystyle{
 c_{i,1}=\left\{
 \begin{array}{ll}
 l_{i,1} & \text{if} \ h_{i,1}=\delta_i+1 \\
 l_{i,1}-1 & \text{if} \ h_{i,1}<\delta_i+1
 \end{array}
 \right.
 }
 ,
 $ 
and for $2\leq k\leq n_i$,  we get \\
$\displaystyle{g_{i,k}=(l_{i,k}-1)+\cdots + (l_{i,n_i}-1)=\sum_{j=k}^{j=n_i}l_{i,j}-(n_i+1-k)}$.
\item[$\bullet$] 
$r_{i,k}=(h_{i,k}-h_{i,k+1}-1) + \cdots + (h_{i,n_i-1}-h_{i,n_i}-1)+(h_{i,n_i}-1)=h_{i,k}-\left(n_i -(k-1)\right)$
\end{itemize}
\item If $h_{i,n_i}=0$, then we have 
\begin{itemize}
\item[$\bullet$] 
$\displaystyle{g_{i,1}=c_{i,1}+(l_{i,2}-1)+\cdots +(l_{i,n_{i-1}}-1)+l_{i,n_i}=c_{i,1}+\sum_{k=2}^{k=n_i}l_{i,k}-(n_i-2)}$ where \\
 $
 \displaystyle{
 c_{i,1}=\left\{
 \begin{array}{ll}
 l_{i,1} & \text{if} \ h_{i,1}=\delta_i+1 \\
 l_{i,1}-1 & \text{if} \ h_{i,1}<\delta_i+1
 \end{array}
 \right.
 }
 ,
 $ 
and for $2\leq k\leq n_i$,  we get \\
$\displaystyle{g_{i,k}=(l_{i,k}-1)+\cdots + (l_{i,n_{i-1}}-1)+l_{i,n_i}=\sum_{j=k}^{j=n_i}l_{i,j}-(n_i-k)}$.
\item[$\bullet$] 
$r_{i,k}=(h_{i,k}-h_{i,k+1}-1) + \cdots +(h_{i,n_{i-1}}-1)=h_{i,k}-\left(n_i -k\right)$.
\end{itemize}
\end{enumerate}
\end{proof}

\begin{remark}\label{rem_theta}
Using the formulas for $r_{i,k}$ and $g_{i,k}$, if we let $\theta_{i,k}=g_{i,k}-r_{i,k}$ ($1\leq k\leq n_i$),
we obtain
\begin{itemize}
\item[$\bullet$]
$
\theta_{i,1}=\left\{
\begin{array}{ll}
\delta_{i+1}+1-h_{i,1} & \text{if} \ h_{i,1}=\delta_i+1\\
\delta_{i+1}-h_{i,1} & \text{if} \ h_{i,1}<\delta_i+1
\end{array}
\right.
$
\item[$\bullet$]
$
\displaystyle{
\theta_{i,k}=\sum_{j=k}^{j=n_i}l_{i,j}-h_{i,k}
}
$ for $2\leq k\leq n_i$.
\end{itemize}
\end{remark}

\begin{remark}\label{compact_formulas_for_gik_rik}
For $1\leq k\leq n_i$, the formulas in Lemma \ref{degree_i+1_rel_gen_formulas}  can be rewritten in a compact way using the invariants of the hook code Equation \eqref{hookcomponent}
\begin{itemize}
\item[$\bullet$]
$
 \displaystyle{
 g_{i,k}=\sum_{j=k}^{j=n_i}l_{i,j}-(n_i+1-k)+\max\{1-h_{i,n_i},0\}+\max\{h_{i,k}-\delta_i,0\}
 }
 $;
 \item[$\bullet$]
 $
 r_{i,k}=h_{i,k}-(n_i+1-k) +\max\{1-h_{i,n_i},0\}
 $.
\end{itemize}
So,  $
\displaystyle{
\theta_{i,k}=\sum_{j=k}^{j=n_i}l_{i,j}-h_{i,k}+\max\{h_{i,k}-\delta_i,0\}
}
$ for $1\leq k\leq n_i$.  \\
As in Theorem \ref{kappathm}, let $\displaystyle{\tau_{i,k}=\sum_{j=k}^{j=n_i}l_{i,j}-h_{i,k}}$. Then $\tau_{i,k}=\theta_{i,k}$ for $2\leq k \leq n_i$ and 
$\tau_{i,1}=\theta_{i,1}-\max\{h_{i,1}-\delta_i,0\}$. These are the ingredients for the formula Equation \eqref{beta0eq} for $\beta_{i+1,0}(P)=N_{i+1}-N_i$ in Theorem \ref{extra_gens_in_given_degree}, which gives the number of generators of degree-$i$ for an ideal $ I$ defining a generic element $A $ of $\mathbb V(E_P)$.
\end{remark}

\begin{theorem}\label{extra_gens_in_given_degree}  
Let $P\in \mathcal P(T), T=\left(1,\dots ,d,t_d,\dots ,t_{\sf j},0\right)$ have hook code
 $\mathfrak{Q}(P)=\left(\mathfrak{h}_d,\dots ,\mathfrak{h}_{i}, \dots, \mathfrak{h}_{{\sf j}}\right)$. 
 Suppose $\mathfrak{h}_i=\left(h_{i,1}^{l_{i,1}}, \dots,  h_{i,k}^{l_{i,k}}, \dots , h_{i,n_{i}}^{l_{i,n_{i}}}\right)$ is the $i$-th block 
 of the hook code $\mathfrak{Q}(P)$ of $P$.  
 Let  $I$ be a generic element of the cell ${\mathbb V}(E_P)$ and ${\mathcal G}(I)$ a minimal set of generators of $I$, and $\beta_{i,0}(P)$ the number of degree-$i$ generators. We denote by $\overset{\to}{\kappa}(P)=\beta_0(P)=\left(\beta_{d,0}(P),\ldots ,\beta_{{\sf j}+1,0}(P)\right)$.
 For $d\leq m \leq {\sf j}+1$, let $$N_m=\#\left\{ f\in {\mathcal G}(I), \text{ such that degree}(f) \leq m \right\}.$$
 Then $N_d=d+1-t_d$, and using the previously defined numbers $r_{i,k}$ and $g_{i,k}$, for $i\in\left[d,{\sf j}\right]$, we have for $\beta_{i+1,0}(P)=N_{i+1}-N_i$,
 \begin{equation}\label{beta0eq}
 \begin{array}{ll}
 N_{i+1}-N_{i}&=\max\left\{0,g_{i,k}-r_{i,k} \right\}_{1\leq k \leq n_i}\\&=\max\left\{0,\theta_{i,k} \right\}_{1\leq k\leq n_i} \\
 &=\max\left\{\delta_{i+1}-\delta_i,0,\tau_{i,k}\right\}_{1\leq k\leq n_i}.
 \end{array}
 \end{equation}

\end{theorem}
\begin{proof}
\noindent
By standard basis construction techniques (see Theorem I.1.9 of \cite{brian}, Propositions 2 and 3 of \cite{brian-gal}), one can  first see that if $r_{i,n_i}\geq g_{i,n_i}$, then there are enough degree 
$i+1$ relations to kick out all of the $g_{i,n_i}$ generators that are just above these relations. Also, if $r_{i,n_i}<g_{i,n_i}$, 
then we need $g_{i,n_i}-r_{i,n_i}$ extra generators whose leading terms are corner- monomials of $E_P$ below the 
$\left(l_{i,1}+ \dots + l_{i,n_i-1}+1\right)$-th degree $i$ hand of $P$. 

\medskip 
\noindent
It is clear that if for all $k$ ($1\leq k\leq n_i$) we have $r_{i,k}\geq g_{i,k}$, then we can kick out all degree $i+1$ generators whose leading terms are degree $i+1$ corner-monomials of $E_P$. \\
Now, suppose there exists an integer $k$ such that $r_{i,k}<g_{i,k}$. Then we can inductively consider the following sets and numbers. 
$$
\begin{array}{lr}
S_0=\left\{k\in {\mathbb N}, \quad  r_{i,k}<g_{i,k}\right\}, & s_0=\max(S_0);
\\
S_1=\left\{k\in {\mathbb N},  \quad k<s_0, \quad r_{i,k}-r_{i,s_0}<g_{i,k}-g_{i,s_0}\right\}, & s_1=\max(S_1);
\\
\dotfill &\dotfill ;
\\
S_{q}=\left\{k\in {\mathbb N},  \quad k<s_{q-1}, \quad r_{i,k}-r_{i,s_{q-1}}<g_{i,k}-g_{i,s_{q-1}}\right\}, & s_q=\max(S_q);
\\
S_{q+1}=\emptyset. & 
\end{array}
$$
The meaning of the sets $S_0, \cdots, S_q$ and the numbers $s_0, \cdots, s_q$ is the following: 
\begin{itemize}
\item[$\bullet$] First, we have $r_{i,s_{0}+1}\geq g_{i,s_{0}+1}, \cdots, r_{i,n_i}\geq g_{i,n_i}$ and $r_{i,s_0}<g_{i,s_0}$. So we need 
$g_{i,s_0}-r_{i,s_0}$ generators whose leading terms are in $G_{i,s_{0}}$, the $s_0$-th part of the chain 
$G_{i,n}\subset G_{i,n-1} \subset \cdots \subset G_{i,s_0}\subset G_{i,s_0-1} \subset \cdots \subset G_{i,1}$. 
This means we have used all the relations in the $s_0$-th part of the chain 
$R_{i,n}\subset R_{i,n-1} \subset \cdots \subset R_{i,s_0}\subset R_{i,s_0-1} \subset \cdots \subset R_{i,1}$.
\item[$\bullet$] Since we have used all the relations in $R_{i,s_0}$, if we are looking for more extra generators, 
the next step is to consider the chains of inclusions 
$$
 \begin{array}{l}
 (R_{i,s_0-1}-R_{i,s_0}) \subset \cdots \subset (R_{i,1}- R_{i,s_0})\\
 (G_{i,s_0-1}-G_{i,s_0}) \subset \cdots \subset (G_{i,1}-G_{i,s_0})
 .
 \end{array}
 $$
If the set $S_1=\left\{k\in {\mathbb N},  \quad k<s_0, \quad r_{i,k}-r_{i,s_0}<g_{i,k}-g_{i,s_0}\right\}$ is not empty, then we 
set $s_1=\max(S_1)$ and continue looking for extra generators until $S_{q+1}=\emptyset$ and $S_{q}\neq \emptyset$ for some index $q$.
\end{itemize}
By construction, the number of degree $i+1$ extra generators needed is 
$$
 g_{i,s_0}-r_{i,s_0}+\sum_{j=1}^{j=q}\left((g_{i,s_j}-g_{i,s_{j-1}})-(r_{i,s_j}-r_{i,s_{j-1}})\right)=g_{i,s_q}-r_{i,s_q}
 .
$$
It is then clear that 
$$
N_{i+1}-N_{i}=g_{i,s_q}-r_{i,s_q}=\max\left\{0,g_{i,k}-r_{i,k} \right\}_{1\leq k \leq n_i} 
.
$$
From the formulas in Remark \ref{rem_theta}  we trivially have $N_{i+1}-N_{i}=\max\left\{0,\theta_{i,k} \right\}_{1\leq k\leq n_i}$.  Now to show the last equality in Equation \ref{beta0eq} notice that  if $h_{i,1}=\delta_{i}+1$, then $\theta_{i,1}=\delta_{i+1}-\delta_i=\tau_{i,1}+1$.  
We then have  
$\left\{0,\theta_{i,k} \right\}_{1\leq k\leq n_i}=\left\{\delta_{i+1}-\delta_i,0,\tau_{i,k}\right\}_{2\leq k\leq n_i}$ because 
$\theta_{i,k}=\tau_{i,k}$ for $2\leq k \leq n_i$. 
Thus 
$$\max\left\{0,\theta_{i,k} \right\}_{1\leq k\leq n_i}=\max\left\{\delta_{i+1}-\delta_i,0,\tau_{i,k}\right\}_{1\leq k\leq n_i}.$$
Otherwise, if $h_{i,1}<\delta_{i}+1$, we have $\theta_{i,k}=\tau_{i,k}$ for $1\leq k \leq n_i$ and 
$\delta_{i+1}-\delta_i\leq \tau_{i,1}$. 
So again 
$\max\left\{0,\theta_{i,k} \right\}_{1\leq k\leq n_i}=\max\left\{\delta_{i+1}-\delta_i,0,\tau_{i,k}\right\}_{1\leq k\leq n_i}$.
\end{proof}\par
\begin{corollary}\label{relate_kappa_P_i-to-kappa_P}
Let $P\in \mathcal P(T), T=\left(1,\dots ,d,t_d,\dots ,t_{\sf j},0\right)$ have hook code
 $\mathfrak{Q}(P)=\left(\mathfrak{h}_d,\dots ,\mathfrak{h}_{i}, \dots, \mathfrak{h}_{{\sf j}}\right)$. 
 Suppose $\mathfrak{h}_i=\left(h_{i,1}^{l_{i,1}}, \dots,  h_{i,k}^{l_{i,k}}, \dots , h_{i,n_{i}}^{l_{i,n_{i}}}\right)$ is the $i$-th block 
 of the hook code $\mathfrak{Q}(P)$ of $P$ and let $P_i$ be the single-block partition related to $\mathfrak{h}_i$.
 Using the notation of Theorem \ref{extra_gens_in_given_degree}, let $\overset{\to}{\kappa}(P)=\beta_0(P)=\left(\beta_{d,0}(P),\ldots ,\beta_{{\sf j}+1,0}(P)\right)$. 
 Then  for $d\leq i \leq j$, $\beta_{i+1,0}(P)=\kappa(P_i)-(\delta_i+1)$.
\end{corollary}
\begin{proof}
Recall that the single-block partition $P_i$ related to $$\mathfrak{h}_i=\left(h_{i,1}^{l_{i,1}}, \dots,  h_{i,k}^{l_{i,k}}, \dots , h_{i,n_{i}}^{l_{i,n_{i}}}\right)$$
is $T_i=\left(1, \ldots,\delta_i+\delta_{i+1}-1, \delta_i+\delta_{i+1},\delta_{i+1} \right)$ (\ref{HFdecomposition}). Applying Theorem \ref{kappathm} to the single-block case where $d=\delta_i+\delta_{i+1}$, $t=\delta_{i+1}$ and $s=\delta_i+1$, we get $\kappa(P_i)=\delta_i+1+\max\left\{\delta_{i+1}-\delta_i,0,\tau_{i,k}\right\}_{1\leq k\leq n_i}$. That is $\kappa(P_i)-\left(\delta_i+1\right)=\max\left\{\delta_{i+1}-\delta_i,0,\tau_{i,k}\right\}_{1\leq k\leq n_i}$.
By Theorem \ref{extra_gens_in_given_degree} we have $\beta_{i+1,0}(P)=N_{i+1}-N_i=\max\left\{\delta_{i+1}-\delta_i,0,\tau_{i,k}\right\}_{1\leq k\leq n_i}$, so $\beta_{i+1,0}(P)=\kappa(P_i)-(\delta_i+1)$.
\end{proof}
\begin{remark}\label{interprete_tau_i_k}
Let $\mathfrak{h}_i=\left(h_{i,1}^{l_{i,1}}, \dots,  h_{i,k}^{l_{i,k}}, \dots , h_{i,n_{i}}^{l_{i,n_{i}}}\right)$ be the $i$-th block of the hook code $\mathfrak{Q}(P)$. 
Moving along the $i$-th diagonal of $P$ from top to bottom, we see that $\displaystyle{\sum_{j=k}^{j=n_i}l_{i,j}}$ is, by definition, the number of degree-$i$ hand monomials of $P$ below the $(l_{i,1}+\ldots+l_{i,k-1})$-th degree-$i$ hand monomial of $P$. Also, by definition, $h_{i,k}$ is the number of degree $i$ \textsl{horizontal-border monomials} of $P$ below the $(l_{i,1}+\ldots+l_{i,k-1}+1)$-th degree-$i$ hand monomial of $P$. \\
We may visualize the key integer $\tau_{i,k}$ related to $\beta_{i+1,0}(P)$ in the Ferrers diagram of $P$ by coloring the corresponding $\displaystyle{\sum_{j=k}^{j=n_i}l_{i,j}}$ hand monomials in red and the $h_{i,k}$ \textsl{horizontal-border monomials} in blue, in the next example (Figure \ref{tau_14_2_fig}). 
\end{remark}
\begin{example}\label{tau_figure}
Let $T=(1,2,\ldots,12,13,12,6)$ and consider the partition $P$ of diagonal lengths $T$ given by $P=(14^2,12,11^2,10,7^2,5^3,4,3,1)$. 
$T$ is a two-block Hilbert function. 
The hook code of $P$ is $\mathfrak{Q}(P)=\left(\mathfrak{h}_{13},\mathfrak{h}_{14}\right)$. We have 
$\delta_{13}=1$, $\delta_{14}=6$, $\delta_{15}=6$, $\mathfrak{h}_{13}(P)=(2^3,1^1,0^2)$ and $\mathfrak{h}_{14}(P)=(6^1,4^2,2^3)$. 
\\
Note that $h_{13,1}=2=\delta_{13}+1$, $h_{14,1}=6<\delta_{14}+1$ and $\delta_{14}-\delta_{13}=5$.
\begin{itemize}
\item[$\bullet$] $\tau_{13,1}$, $\tau_{13,2}$ and $\tau_{13,3}$ are computed using the hook code block $\mathfrak{h}_{13}=(2^3,1^1,0^2)$. 
\begin{itemize}
\item $\tau_{13,1}=3+1+2-2=4$. We may visualise $\tau_{13,1}$ by coloring the sequence $(2^3,1^1,0^2)$:  
$(2^3,1^1,0^2)=(\textcolor{blue}{2}^{\textcolor{red}{3}},1^{\textcolor{red}{1}},0^{\textcolor{red}{2}})$; 
$$
\begin{array}{c}
\tau_{13,1}=\text{\textsl{\textcolor{red}{sum of red integers} minus the \textcolor{blue}{blue integer}}}
\\
\displaystyle{
\tau_{i,k}=\textcolor{red}{\sum_{j=k}^{j=n_i}l_{i,j}}-\textcolor{blue}{h_{i,k}}
}
\end{array}
.
$$
\item $\tau_{13,2}=1+2-1=2$. $\tau_{13,2}$ can be visualised by coloring the subsequence $(1^1,0^2)$ of $(2^3,1^1,0^2)$: 
$(2^3,1^1,0^2)=(2^3,\textcolor{blue}{1}^{\textcolor{red}{1}},0^{\textcolor{red}{2}})$, so $\tau_{13,2}=\text{\textsl{sum of red integers minus the blue integer}}$.
\item $\tau_{13,3}=2-0=2$, computed using $(2^3,1^1,0^2)=(2^3,1^1,\textcolor{blue}{0}^{\textcolor{red}{2}})$; 
$\tau_{13,3}=\text{\textsl{sum of red integers minus the blue integer}}$.
\end{itemize}
We then find that $\beta_{14,0}(P)=\max\left\{\delta_{14}-\delta_{13},0,\tau_{13,1}, \tau_{13,2},\tau_{13,3}\right\}=5$.
\item[$\bullet$] to compute $\tau_{14,1}$, $\tau_{14,2}$ and $\tau_{14,3}$ we can use the same coloring method on the hook code block $\mathfrak{h}_{14}=(6^1,4^2,2^3)$.
\begin{itemize}
\item $\tau_{14,1}=1+2+3-6=0$; this is $\tau_{14,1}$=\textsl{sum of red integers minus the blue integer}, using the coloring 
$(\textcolor{blue}{6}^{\textcolor{red}{1}},4^{\textcolor{red}{2}},2^{\textcolor{red}{3}})$. 
\item $\tau_{14,2}=\textcolor{red}{2+3}-\textcolor{blue}{4}=1$ and $\tau_{14,3}=\textcolor{red}{3}-\textcolor{blue}{2}=1$ using the colorings
$(6^1,\textcolor{blue}{4}^{\textcolor{red}{2}},2^{\textcolor{red}{3}})$, $(6^1,4^2,\textcolor{blue}{2}^{\textcolor{red}{3}})$.
\end{itemize}
We then find that $\beta_{15,0}(P)=\max\left\{\delta_{15}-\delta_{14},0,\tau_{14,1}, \tau_{14,2},\tau_{14,3}\right\}=1$.
\end{itemize}
We illustrate $\tau_{14,2}$ in Figure \ref{tau_14_2_fig} by coloring the degree-$i$ hand monomials and the degree-$i$ \textsl{horizontal-border monomials} as suggested in Remark \ref{interprete_tau_i_k}
\begin{figure}[!h]
\begin{center}
	\begin{tikzpicture}[scale=.65]
	\draw[dashed,very thin] 	(0,0)--(14,0)
						(0,-1)--(14,-1)
						(0,-2)--(12,-2)
						(0,-3)--(11,-3)
						(0,-4)--(11,-4)
						(0,-5)--(10,-5)
						(0,-6)--(7,-6)
						(0,-7)--(7,-7)
						(0,-8)--(5,-8)
						(0,-9)--(5,-9)
						(0,-10)--(5,-10)
						(0,-11)--(4,-11)
						(0,-12)--(3,-12)
						(0,-13)--(1,-13)
						(0,0)--(0,-14)
						(1,0)--(1,-13)
						(2,0)--(2,-13)
						(3,0)--(3,-12)
						(4,0)--(4,-11)
						(5,0)--(5,-8)
						(6,0)--(6,-8)
						(7,0)--(7,-6)
						(8,0)--(8,-6)
						(9,0)--(9,-6)
						(10,0)--(10,-5)
						(11,0)--(11,-3)
						(12,0)--(12,-2)
						(13,0)--(13,-2);
	\draw[thin,red,fill=red!8] (10,-4) rectangle (11,-5)
							(9,-5) rectangle (10,-6)
							(4,-10) rectangle (5,-11)
							(3,-11) rectangle (4,-12)
							(2,-12) rectangle (3,-13);
	\draw[thin, blue,fill=black!5,rounded corners=2mm] (8,-6) rectangle (9,-7)
					     	(6,-8) rectangle (7,-9)
						(1,-13) rectangle (2,-14)
					  	(0,-14) rectangle (1,-15);
	\draw[red] node[font=\tiny] at (10.5,-4.5) {$x^{10}y^4$}
			  node[font=\tiny] at (9.5,-5.5) {$x^9y^5$}
			  node[font=\tiny] at (4.5,-10.5) {$x^4y^{10}$}
			  node[font=\tiny] at (3.5,-11.5) {$x^3y^{11}$}
			  node[font=\tiny] at (2.5,-12.5) {$x^2y^{12}$};
	\draw[blue] node[font=\tiny] at (8.5,-6.5) {$x^8y^6$}
			    node[font=\tiny] at (6.5,-8.5) {$x^6y^8$}
			    node[font=\tiny] at (1.5,-13.5) {$xy^{13}$}
			    node[font=\tiny] at (0.5,-14.5) {$y^{14}$};		    
\draw[thick] (14,0)--(14,-2)--(12,-2)--(12,-3)--(11,-3)--(11,-5)--(10,-5)--(10,-6)--(7,-6)--(7,-8)--(5,-8)--(5,-11)--(4,-11)--(4,-12)--(3,-12)--(3,-13)--(1,-13)--(1,-14)--(0,-14);
\end{tikzpicture}
\caption{$P=(14^2,12,11^2,10,7^2,5^3,4,3,1)$ : $\tau_{14,2}=\textcolor{red}{2+3}-\textcolor{blue}{4}=1$ }
\label{tau_14_2_fig}
\end{center}
\end{figure}
\end{example}
\pagebreak

\begin{remark}\label{lower_bound_for_degree_i+1_gens}
Using the notation of Lemma \ref{kappaTlem} and Theorem \ref{extra_gens_in_given_degree} one has 
\begin{enumerate}[(i).]
\item $\left[\delta_{i+1}-\delta_i\right]^{+}\leq \max\left\{\delta_{i+1}-\delta_i,0,\tau_{i,k}\right\}_{1\leq k\leq n_i}=N_{i+1}-N_{i}$, so the number of degree $i+1$ generators of $I\in {\mathbb V}(E_P)$ is at least $\left[\delta_{i+1}-\delta_i\right]^{+}$ (this is statement i. of Lemma \ref{kappaTlem}).
\item Suppose $P$ is the partition associated to the generic cell of $\G_T$. Then for the $i$-th block  
$\mathfrak{h}_i=\left(h_{i,1}^{l_{i,1}}, \dots,  h_{i,k}^{l_{i,k}}, \dots , h_{i,n_{i}}^{l_{i,n_{i}}}\right)$ of the hook code $\mathfrak{Q}(P)$ of $P$, we have $n_i=1$, $h_{i,1}=\delta_i+1$ and $l_{i,1}=\delta_{i+1}$, that is $\mathfrak{h}_i=\left(\delta_i+1\right)^{\delta_{i+1}}$. In this case, $g_{i,1}=\delta_{i+1}$, $r_{i,1}=\delta_i$ and formula \ref{beta0eq} of Theorem \ref{extra_gens_in_given_degree} 
gives $N_{i+1}-N_{i}=\left[\delta_{i+1}-\delta_i\right]^{+}$ (this is statement ii. of Lemma \ref{kappaTlem}).
\item An immediate corollary of Theorem \ref{componentTheorem} in the single-block case recovers Theorem \ref{kappathm}. More precisely, for a single-block partition $P_i$ with Hilbert function $T_i=(1, \cdots, d-1,d,t_d,0)$ where $d=\delta_i+\delta_{i+1}$ and $t_d=\delta_{i+1}$. Denote the difference-one hook code of $P$, $\mathfrak{Q}(P_i)$, by $\mathfrak{h}_i$ as above which satisfies 
$$
\delta_i+1\geq h_{i,1}> \dots > h_{i,n_{i}}\geq 0 \quad \text{and} \quad \displaystyle{\sum_{k=1}^{k=n_i}l_{i,k}=\delta_{i+1}}.
$$ Then $\kappa(P_i)=\delta_i+1+\max\left\{0,g_{i,k}-r_{i,k} \right\}_{1\leq k \leq n_i}$.

\end{enumerate}
\end{remark}
We can now state and prove the main theorem of this section.
\begin{theorem}[Decomposition of $\kappa(P)$ into components]\label{componentTheorem}
Let $P$ be a partition of lengths $T=\left(1,2,\dots ,d,t_d,\dots ,t_{\sf j},0\right)$.  Then
\begin{equation}\label{kappaDeompositionEq}
\kappa(P) = \sum\limits_{i=d}^{\sf j}\kappa(P_i)-(t_d-t_{\sf j})-({\sf j}-d).
\end{equation}
Moreover, a minimal set of generators for a generic ideal in the cell $\mathbb{V}(E_P)$ includes $d+1-t_d$ generators of degree $d$ and  $\kappa(P_i)-(\delta_i+1)$ generators of degree $i+1$, for $i\in [d, {\sf j}]$.
\end{theorem}
\begin{proof}
Let  $I$ be a generic element of the cell ${\mathbb V}(E_P)$, ${\mathcal G}(I)$ a minimal set of generators of $I$, and $\beta_{i,0}(P)$ the number of degree-$i$ generators of $I$. From Theorem \ref{extra_gens_in_given_degree} and Corallary \ref{relate_kappa_P_i-to-kappa_P}, we have $\beta_{i+1,0}(P)=N_{i+1}-N_i=\kappa(P_i)-(\delta_i +1)$ for all $i\in [d,{\sf j}]$. By definition, $\kappa(P)=\displaystyle{\#{\mathcal G}(I)=N_{{\sf j}+1}=N_d+\sum_{i=d+1}^{i={\sf j}+1}\beta_{i,0}(P)}$. Since $\beta_{i+1,0}(P)=\kappa(P_i)-(\delta_i +1)$, we have $\kappa(P)=\displaystyle{N_d+\sum_{i=d}^{i={\sf j}}\left(\kappa(P_i)-(\delta_i +1)\right)}$. Clearly, $N_d=d+1-t_d$. Recall that by definition, $\delta_i=t_{i-1}-t_i$. We then have 
\begin{align*}
\kappa(P)&=N_d+\sum_{i=d}^{i={\sf j}}\left(\kappa(P_i)-(\delta_i +1)\right)\\
&= d+1-t_d + \sum_{i=d}^{i={\sf j}}{\kappa(P_i)} - \sum_{i=d}^{i={\sf j}}{(t_{i-1}-t_i)}- \sum_{i=d}^{i={\sf j}}1\\
&=  \sum_{i=d}^{i={\sf j}}{\kappa(P_i)} -(t_d-t_{\sf j})-({\sf j}-d).
\end{align*}
\end{proof}
\begin{remark}\label{empty-component-contribution}
Note that if $t_i=t_{i+1}$, then it is clear that $N_{i+1}-N_{i}=0$. Also, we found  (Proposition \ref{kappa-of-empty-bloc}) that when $t_i=t_{i+1}$, we have 
$\kappa(P_i)=\delta_i +1$. So the empty hook blocs contribute to zero in the sum $\kappa(P)=\displaystyle{N_d+\sum_{i=d}^{i={\sf j}}\left(\kappa(P_i)-(\delta_i +1)\right)}$ that computes $\kappa(P)$ in Theorem \ref{componentTheorem}.
\end{remark}
\medskip

\begin{proposition}\label{E_P_Betti_sequence}
Let $P$ be a partition of diagonal lengths $T=(1, 2, \cdots , d-1, t_d, \cdots , t_{\sf j}, 0)$ and suppose 
$\mathfrak{h}_i=\left(h_{i,1}^{l_{i,1}}, \dots,  h_{i,k}^{l_{i,k}}, \dots , h_{i,n_{i}}^{l_{i,n_{i}}}\right)$ is the $i$-th block 
 of the hook code $\mathfrak{Q}(P)$ of $P$.  \\
 Denote by $b_{i+1}(E_P)$ the number of degree $i+1$ corner-monomials (generators) of $E_P$. Then we have 
 $$
 \begin{array}{l}
 b_d=d+1-t_d,\\
 b_{i+1}(E_P)=\delta_{i+1}-n_i+\max\{1-h_{i,n_i},0\}+\max\{h_{i,1}-\delta_i,0\}.

 \end{array}
 $$
\end{proposition}
\begin{proof}
The number of degree $d$ corner-monomials of $E_P$ is of course $d+1-t_d$. By Definition \ref{degree_i+1_rel_gen_sequences} we have $b_{i+1}(E_P)=g_{i,1}$. The proof of the Proposition then follows directly from Lemma \ref{degree_i+1_rel_gen_formulas}  and Remark \ref{compact_formulas_for_gik_rik}. 
\end{proof}

\par
Recall from Definition \ref{special-partition-definition} that a partition $P$ of diagonal lengths $T$ is special if $\kappa(P)\neq \kappa(T)$. In other words, $P$ is special if $\kappa(P)$ does not have the minimum value $\kappa(T)$ possible for partitions of diagonal lengths $T$, from Equation \ref{kappaTeq}. The following immediate corollary of  Theorem~\ref{componentTheorem} gives a necessary and sufficient condition for a partition $P$ to be special. Recall that Corollary \ref{special1cor} specifies when  a single-block partition is special.

\begin{theorem}
[Component Theorem for $P$ special]\label{componentthm}
Assume that $T$ is a Hilbert function of height $d$ and socle degree ${\sf j}$,
 and that the partition $P$ of diagonal lengths $T$ decomposes into single-block partitions $P_d, \dots
, P_{\sf j}$. Then $P$ is special if and only if  $P_i$ is special, for some $i\in [d, {\sf j}]$.
\end{theorem}

\begin{proof}
By Theorem~\ref{componentTheorem}, the value of $\kappa(P)$ is minimum - equal to $\kappa(T)$ - if and only if $\kappa(P_i)$ is minimum for all $i\in [d,{\sf j}]$. Thus $P$ is non-special if and only if  at least one component $P_i$ is non-special for an integer $i\in [d,{\sf j}]$.
\end{proof}
\begin{figure}[t!]
\centering
\scalebox{0.6}{\begin{ytableau}
\textcolor{gray}{\bullet} & &&& \textcolor{gray}{\bullet} &\textcolor{gray}{\bullet} &&& \textcolor{gray}{\bullet}&&&&&*(black! 20)\textcolor{gray}{\bullet} &*(black! 50)\cr
$ $ & &&&&&&&&&&\cr
$ $ &&&\circ &&&\circ &\circ &&&&*(black! 20)\cr
\textcolor{gray}{\bullet} &&&&\textcolor{gray}{\bullet} & \textcolor{gray}{\bullet}&&&\textcolor{gray}{\bullet} && *(black! 20)&*(black! 50)\cr
$ $ & \bullet &   &   &   &   &   &   &   &*(black! 20)\bullet &*(black! 50) &*(black! 100)\cr
$ $  &\bullet&&& & &&&*(black! 20) &*(black! 50)\bullet&*(black! 100)\cr
$ $& &&&&&\cr
$ $& &&&&\cr
$ $& &&  \circ &&*(black! 20)\cr
$ $& && \circ  &*(black! 20)\cr
$ $& &\cr
$ $& &*(black! 20)\cr
\textcolor{gray}{\bullet}& *(black! 20)&*(black! 50)\cr
*(black! 20)&*(black! 50)\bullet&*(black! 100)\cr
\end{ytableau}}\caption{Ferrers diagram of Example \ref{gensEx}. Each labeled box represents a hook corner, it is labeled by a $\circ$ if its hand-degree is 13, with a \textcolor{gray}{$\bullet$} if the hand-degree is 14, and with a dark $\bullet$ when the hand-degree is 15.}\label{3blockdiag}
\end{figure}

\begin{example}\label{gensEx}
Consider the partition $P=(15, 12^4, 11, 7, 6^2, 5, 3^4)$ from Example \ref{1stgensEx}. Then $P$ has diagonal lengths $T=\left(1,2,\dots ,13,10_{13},6_{14},3_{15},0\right)$ and hook code (see Figure \ref{3blockdiag}) $$\mathfrak{Q}(P)=\left((3,1^2, 0)_{13},(5,4,1)_{14}, (2^2,1)_{15}\right).$$ Then, as we saw in Example \ref{1stgensEx}, $P$ can be decomposed into the following three single-block partitions. Partition $P_{13}=(7^2, 5, 4^2, 3, 1^2)$ of diagonal lengths $T_{13}=(1, \dots, 7, 4,0)$ and hook code $\mathfrak{Q}(P_{13})=(3, 1^2, 0)$,
partition $P_{14}=(8, 6^2,4, 3, 2^2)$ of diagonal lengths $T_{14}=(1, \dots, 7, 3,0)$ and hook code $\mathfrak{Q}(P_{14})=(5, 4, 1)$,
and partition $P_{15}=(6, 5^2, 4, 2^2)$ of diagonal lengths $T_{15}=(1, \dots, 6, 3,0)$ and hook code $\mathfrak{Q}(P_{15})=(2^2,1)$.

By Theorem~\ref{kappathm}, we have
\[
\kappa(P_{13})=6,\quad \kappa(P_{14})=5, \mbox{ and }
\kappa(P_{15})=5.
\]

Thus, by Theorem~\ref{componentTheorem}, we have $$\kappa(P)=6+5+5-(10-3)-2=7.$$

Let $I$ be a generic ideal in the cell $\mathbb{V}(E_P)$. Then a minimal set of generators for $I$ consists of seven generators. Of these seven generators, four have degree 13, two have degree 14, and one has degree 15. 
\end{example}

\subsubsection{Elementary and non-elementary Hilbert functions.}
A key aspect to understanding the family $\G_T$ is that, when the sequence $T$ of Equation \ref{Teq}
has a constant subsequence $(s,s,\ldots)$ with $s<d$, the order of $T$, then $G_T$ splits into the direct product
of simpler parameter spaces. We explain here briefly consequences for our analysis of generators for ideals in Jordan cells.\par
We say that a sequence $T$ satisfying Equation \eqref{Teq} is \emph{elementary} if there is no integer $i\in [d,{\sf j}]$ such that
$t_i=t_{i+1}<d$ \cite[\S 4Ai]{IY}; then we also say that $\G_T$ is elementary. It is well known (see \cite[\S 4B]{Ia}, \cite[Lemma 2.2]{Ia2}) that when a Hilbert function $T=H(R/I)$ satisfies $t_i=t_{i+1}=s<d$ then there is a form $f\in R_s$ such that 
\begin{equation}\label{Tconsteq} I_i=fR_{i-s} \text { and }I_{i+1}=f\cdot R_{i+1-s}.
\end{equation}
  It follows that 
$f| I_u$ for $u\le i+1$. This is usually shown using the properties of $\tau(V)=\dim_{\sf k} R_1V-\dim_{\sf k}V $ for vector subspaces $V\subset R_i$: this integer is the number of generators of an ``ancestor ideal'' $I=(V)\oplus_{u=1}^i V:R_u$, and $\tau (I_i)=1$ when $t_i=t_{i+1}$ (ibid.).\par
We will define implicitly in the next Theorem ``elementary factors'' $T(i)$ of Hilbert sequences $T$ which have constant subsequences of height $s<d$. These factors have no relation with the single-block components $T_i$ for each $T$, defined in Equation \eqref{HFdecomposition}, and a major topic for us. In fact 
if $T$ splits into elementary components $T(i)$ they are not usually single-block.
\begin{lemma}\label{elemlem}\cite[Lemma 4.2]{IY} There is a decomposition of $\G_T$ as a product 
$$\G_T=\prod_k G_{T(k)} \text { for $T(k)$ elementary}.$$
\end{lemma}
\begin{proof} Assume there is a single maximal consecutive subsequence
$t_i=t_{i+1}=\cdots =t_{i+k}=s$ with $k\ge 1$ and $s<d$.  Then consider $T(1)=(1,2,\ldots,s_{s-1},s,\ldots  ,s_{i+k},t_{i+k+1},\ldots ,t_{\sf j})$, and $T(2)$ defined by $T(2)_{u}=t_{u+s}-s$ for $u\le i-s$.  Let $p_I\in \G_T$ be a point parametrizing the graded ideal $I$ such that $A=R/I$ satisfies $H(A)=T$. Then we let $I(1)=(f_s,I)$. We have $I_{u+s}=f_sV_{u}$ for $0\le u\le i-s$: we define $I(2)_u=V_u$ for $u\in [0, i-s]$. Then the pair $(I(1),I(2))$ determines $I$ and conversely. This proves the Lemma for $k=2$, it is straightfoward to extend it to $k\ge 2$.
\end{proof}

\begin{remark}\label{kappa_cell1xcell_2_rem}\par
Let $T=(1, 2, \cdots , d, t_d, \cdots , t_{\sf j}, 0)$ be a Hilbert function as in Equation \ref{Teq} and $P$ a partition of diagonal lengths $T$. 
Suppose there is an integer $i\in\left[d,{\sf j}-1 \right]$ such that $t_i=t_{i+1}=\cdots =t_{i+k}$ with $k\geq 1$ and $s<d$.
Let  $I$ be a generic element of the cell ${\mathbb V}(E_P)$ and ${\mathcal G}(I)$ a minimal set of generators of $I$. Let 
${\mathcal G}(I)_1=\left\{ f\in {\mathcal G}(I), \text{degree}(f) \geq i+1 \right\}$ and ${\mathcal G}(I)_2=\left\{ f\in {\mathcal G}(I), \text{degree}(f) \leq i \right\}$.  Setting $m_1=\left|{\mathcal G}(I)_1\right|$ and $m_2=\left|{\mathcal G}(I)_2\right|$, we get  $\kappa(P)=m_1+m_2$. We know from Equation \eqref{Tconsteq} that there is a degree $s$ form $f_s$ such that $f_s$ divides each of the elements $f_1, \ldots , f_{m_2}$ of ${\mathcal G}(I)_2$. 
Let $I(1)=(f_s,I)$ and $I(2)=(I:f_s)$. Then $I(1)$ is a generic element of a cell ${\mathbb V}(E_{P(1)})$ and $I(2)$ is a generic element of a 
cell ${\mathbb V}(E_{P(2)})$. It is clear that $\kappa({P(1)})=m_1+1$ and $\kappa({P(2)})=m_2$, so $\kappa(P)=\kappa({P(1)})+\kappa({P(2)})-1$.
\end{remark}

\begin{proposition}\label{kappa_prod_elem_cells_prop}
Suppose that the variety $\G_T$ decomposes as $\displaystyle{\G_T=\prod_{k=1}^{k=r} G_{T(k)}}$ with each $T(k)$ elementary. 
Then any cell ${\mathbb V}(E_P)$ of $\G_T$ decomposes as 
$$\displaystyle{{\mathbb V}(E_P)=\prod_{k=1}^{k=r} {\mathbb V}(E_{P(k)})} \text { for $P(k)$ a partition of diagonal lengths $T(k)$}.$$ 
Also $\displaystyle{\kappa(P)=\sum_{k=1}^{k=r}\kappa(P(k))-r+1} $.
\end{proposition}

\begin{proof}
The Proposition follows from Remark \ref{kappa_cell1xcell_2_rem}.
\end{proof}

\begin{example}\label{flatHFexample}(See Figure \ref{5.8figure}.)
Let $P=\left(10^2,4,3^2,2^5\right)$ be the partition of diagonal lengths $$T=\left(1,2,\dots ,6_5,5_6,4_7,4_8,4_9,2_{10},0\right),$$ and difference-one hook code $$\mathfrak{Q}(P)=\left((0)_6,(1,0)_9,(2,1)_{10}\right). $$
Let $I$ be a generic element of $\mathbb{V}(E_P)$. 
The elementary components of $T$, explained in Lemma \ref{elemlem} and Remark \ref{kappa_cell1xcell_2_rem} are 
\begin{equation*}
T(1) = (1,2,3,4,4_4,4_5,\dots,4_9,2_{10},0),\quad \text{and}\hspace*{2mm} T(2)=(1,2,1).
\end{equation*}
As it is explained in Remark \ref{kappa_cell1xcell_2_rem} we let $I(1) = (f_4,I)$ be a generic element in the cell $\mathbb{V}(E_{P(1)})$ with the Hilbert function $T(1)$. Also $I(2)=(I\colon f_4)$ is a generic element in the cell $\mathbb{V}(E_{P(2)})$ with the Hilbert function $T(2)$. We have $P(1) = (10^2,2^8)$ and $P(2)=(2,1^2)$ which are subpartions of $P$ in different colors in Figure \ref{5.8figure}.
We easily see that $\kappa(P(1)) = \kappa(P(2))=3$ and therefore by Proposition \ref{kappa_prod_elem_cells_prop} we get $$\kappa(P) = \kappa(P(1)) + \kappa(P(2))-1 =5.$$

We could also compute $\kappa(P)$ by decomposition of $P$ and  $T$ into single-block components, see Equation \ref{HFdecomposition}.
Single-block component partitions $P_6,\dots ,P_{10}$ of diagonal lengths $T_{6},\dots ,T_{10}$ as follows,
$$
P_6=\left(2,1,1\right), \hspace*{0.2cm}P_7=(1), \hspace*{0.2cm}P_8=\left(0\right), \hspace*{0.2cm} P_9=\left(3,1,1\right)\hspace*{0.2cm}\text{and}\hspace*{0.2cm} P_{10}=\left(4,4,2,2\right),$$
$$
T_6=\left(1,2,1\right), \hspace*{0.2cm}T_7=(1,0), \hspace*{0.2cm}T_8=\left(0\right), \hspace*{0.2cm} T_9=\left(1,2,2,0\right)\hspace*{0.2cm}\text{and}\hspace*{0.2cm} T_{10}=\left(1,2,3,4,2,0\right).$$
The hook codes of $P_6$, $P_9$ and $P_{10}$ are  $\mathfrak{h}_6=(0)$, $\mathfrak{h}_9=(1,0)$ and $\mathfrak{h}_{10}=(2,1)$ respectively. 
Using Theorem \ref{kappathm} we get that 
$$\kappa(P_6)=\kappa(P_9)=\kappa(P_{10})=3,
$$
and for $P_7=\Delta_1$ and $P_8=\Delta_\emptyset$ by Remark \ref{empty_hook_code_block_partition}, we conclude that
$$
\hspace*{0.2cm}\kappa(P_7)=2,\hspace*{0.2cm} \kappa(P_8)=1.
$$
Therefore, Theorem \ref{componentthm} implies that

$$\kappa(P)=3+2+1+3+3-(4-2)-(10-6)=5$$

We also note that of these five generators, two generators have degree 6 and one generator has degree $7$ (corresponding to generators of $P_6$), and two have degree 10.\par
Note that $\dim \G_{T_{10}}=2(3)=6, \dim \G_{T_9}=(2)(1)=2$, and $\dim \G_T=8$, since $\G_T$ is fibred over $\mathbb P_4$ parametrizing the generator $f_4$ of $I_6$ by a Grassmannian $\Grass(2,4)$ parametrizing $I_{10}/f_4R_6$, a two-dimensional subspace of $R_{10}/f_4R_6$, which has dimension four. 
\par
\end{example}
\begin{figure}
\centering
\scalebox{0.6}{\begin{ytableau}
*(blue! 15) &*(blue! 15)&*(blue! 15)&*(blue! 15)&*(blue! 15)&*(blue! 15)&*(blue! 15)&*(blue! 15)&*(blue! 15)&*(blue! 15)\cr
*(blue! 15) &*(blue! 15)&*(blue! 15)&*(blue! 15)&*(blue! 15)&*(blue! 15)&*(blue! 15)&*(blue! 15)&*(blue! 15)&*(blue! 15)\cr
*(blue! 15)&*(blue! 15)&*(red! 20)&*(red! 20)\cr
*(blue! 15)&*(blue! 15)&*(red! 20)\cr
*(blue! 15)&*(blue! 15)&*(red! 20)\cr
*(blue! 15)&*(blue! 15)\cr
*(blue! 15)&*(blue! 15)\cr
*(blue! 15)&*(blue! 15)\cr
*(blue! 15)&*(blue! 15)\cr
*(blue! 15)&*(blue! 15)\cr
\end{ytableau}}\caption{Diagram of the partition in Example \ref{flatHFexample}. The 
elementary partitions $P(1)$ and $P(2)$ are colored blue and red, respectively.}\label{5.8figure}
\end{figure}

\section{Number of cells of special multiblock partitions.}\label{countMultiSection}
 Using Corollary \ref{singlecountcor} and Theorem \ref{componentTheorem} we are able to count the number of multiblock partitions with a given number of generators.

\begin{theorem}\label{multicountthm}
Assume that $T=(1, \dots, d, t_d, \dots, t_{\sf j},0)$ and for $d\leq i \leq \sf j$, let $T_i=(1, \dots,t_{i-1}-t_{i+1}, t_{i}-t_{i+1},0)$. Then for every positive integer $k$, the number of partitions $P$ of diagonal lengths $T$ and $\kappa(P)=k$, denoted by $\mu(T,k)$, satisfies
\begin{equation}\label{multicountEq}
\mu(T,k)=\sum_{(k_d,\dots,k_{\sf j})\in Q_k} \left(\prod_{i=d}^{\sf j}\mu(T_i,k_i)\right),
\end{equation}
where $Q_k=\{(k_d, \dots, k_{\sf j})\in \mathbb{Z}^{\mathsf{j}+1-d}\, |\, k_d+\dots+k_{\sf j}=k+(t_d-d)-(t_{\sf j}-{\sf j})\}.$ 
\end{theorem}

\begin{proof}
This is an immediate consequence of Theorem~\ref{componentTheorem}. Also recall that Corollary~\ref{singlecountcor} provides an explicit formula for $\mu(T_i,k_i)$, for every $d\leq i \leq \sf j$. 

\end{proof}
\begin{remark}
For each $i\in [d,\sf j]$, by Corollary~\ref{singlecountcor}, $\mu(T_i,k_i)$ is non-zero if and only if $\max\{t_i-t_{i+1}+1, t_{i-1}-t_i+1\}+1\leq k_i \leq t_{i-1}-t_{i+1}+1$. Thus in Equation~(\ref{multicountEq}) we are effectively taking the sum over the points in the hyperplane defined by $k_d+\dots+k_{\sf j}=k+(t_d-d)-(t_{\sf j}-\sf j)$ in the hyper cubes obtained by the product of line segments of the form  $[\max\{t_i-t_{i+1}, t_{i-1}-t_i\}+1, t_{i-1}-t_{i+1}+1]$ in $\mathbb{Z}^{\mathsf{j}+1-d}$.
\end{remark}
Recall that $\mathcal{P}(T)$ is the set of all partitions of diagonal lengths $T$. Denote by $A$ the cardinality of $\mathcal{P}(T)$. We have from \cite[Theorem 3.30]{IY}, or as a consequence of Equation \ref{count2eq} below, that
\begin{equation}\label{cardinalityPTeq}
A=\prod_{d\le i\le {\sf j}}\binom{t_{i-1}-t_{i+1}+1}{t_i-t_{i+1}}.
\end{equation}
A refinement, grading by the dimension of the cells, gives the Betti numbers of $\G_T$ \cite[Equation 3.34]{IY}.
Recall from Definition~\ref{special-partition-definition} that a partition $P$ of diagonal lengths $T$ is called special if $\kappa(P)>\kappa(T)$, and denote by $S$ the number of 
special partitions of diagonal lengths $T$.\par 
\noindent Using Definition \ref{Pidef1}, we decompose a partition $P$ of diagonal lengths $T=(1,2,\dots ,d, t_d,\dots , t_{\sf j},0)$ into ${\sf j}+1-d$ single-block partitions, $P_d,\dots ,P_{\sf j}$, where for each $d\leq i\leq \sf j $, the diagonal lengths of $P_i$ is the sequence \linebreak$T_i=\left(1,\dots ,t_{i-1}-t_{i+1},t_i-t_{i+1},0\right)$.
For each $d\leq i\leq \sf j$, we denote the total number of partitions of diagonal lengths $T_i$ by $A_i$ (see Lemma \ref{count1blockcor} and Equation \eqref{count2eq} below) and the number of special partitions of diagonal lengths $T_i$ by $S_i$.
 The number of special partitions is equal to $\sum_{k>\kappa(T)}{\mu(T,k)}$, where $\mu(T,k)$ is described in the above theorem.\par 
\noindent In the following, we provide the number of special partitions of diagonal lengths $T=\left(1,2,\dots ,d,t_d,\dots ,t_{\sf j},0\right)$, using the inclusion-exclusion principal. 
\begin{corollary}[Number of special partitions]\label{countspecialcor}
The number of special partitions of diagonal lengths $T=\left(1,2,\dots ,d,t_d,\dots ,t_{\sf j},0\right)$ is equal to
\begin{equation}\label{specialmultiblock}
S=\sum_{i=1}^{\mathsf{j}-d+1}(-1)^{i+1}\left(\sum_{\lambda\subseteq \{d,\dots ,\mathsf{j}\}, \vert\lambda\vert=i}S_\lambda A_{\{d,\dots ,{\sf {j}}\}\setminus \lambda}\right),
\end{equation}
where $S_\lambda=\prod_{i\in \lambda} S_i$ and $A_{\{d,\dots ,{\sf j}\}\setminus \lambda}=\prod_{i\in {\{d,\dots ,{\sf{j}}\}\setminus \lambda}}A_i$.
\end{corollary}
\begin{proof}
Theorem \ref{componentthm} implies that $P$ is special if and only if $P_i$ is special for some $i\in[d,\sf j]$. \\
Note that for each $i\in [d, \sf j]$ the number of partitions of diagonal lengths $T_i$ is equal to
\begin{equation}\label{count2eq} A_i=\binom{t_{i-1}-t_{i+1}+1}{t_i-t_{i+1}}.
\end{equation}
On the other hand, 
 Theorem \ref{countingthm} provides the number of special single-block partitions. Using Equation~\ref{numberspecial1eq} for each $d\leq i\leq \sf j$ we obtain the number of special partitions of diagonal lengths $T_i$ as the following
\begin{equation}\label{count3eq}
S_i=\binom{t_{i-1}-t_{i+1}+1}{t_i-t_{i+1}-\delta_i-1},
\end{equation}
where $\delta_i=\max\{2t_i-2t_{i+1}-t_{i-1}+t_{i+1},0\}=\max\{2t_i-t_{i+1}-t_{i-1},0\}$.\\
\noindent Now using the inclusion-exclusion principal we get the equality of Equation~\eqref{specialmultiblock}.
\end{proof}

As a consequence of the above Theorem, we recover a result of \cite[Theorem 3.7]{AIK} providing the number of complete intersection Jordan types $P\in \mathcal P(T)$. Recall that a complete intersection Jordan type of diagonal lengths $T$ is a partition $P$ of diagonal lengths $T$ such that $\kappa(P)=2$.
\begin{corollary}\label{CIJordantypecor}
\begin{itemize}
\item[$(a)$]The number of complete intersection Jordan types of diagonal lengths 
\begin{equation*}
T=\left(1_0,2_1,\dots ,(d-1)_{d-2},d_{d-1},(d-1)_{d},\dots , 2_{2d-3},1_{2d-2}\right)
\end{equation*}
is equal to $2^{d-1}$.
\item[$(b)$]The number of complete intersection Jordan types with diagonal lengths 
\small\footnotesize
\begin{equation*}
T=\left(1_0,2_1,\dots ,(d-1)_{d-2},d_{d-1}, \dots , d_{d+k-2},(d-1)_{d+k-1},\dots , 2_{2d-4+k},1_{2d-3+k}\right)
\end{equation*}
\normalsize
where $k\ge 2$ is equal to $2^{d}$.
\end{itemize}
 
\end{corollary}
\begin{proof}
\begin{itemize}

\item[$(a)$] In this case we have that $\mathsf{j}= 2d-2$ and the number of blocks in this case is equal to $d-1$, we also have $t_{d-1}=d, t_d=d-1,\dots ,t_{\sf j}=1$. For each $d\leq i\leq \sf j$ we have that $T_i=\left(1,2,1\right)$, and clearly $A_i=3$ and $S_i=1$. So the total number of partitions of diagonal lengths $T$ is $A=3^{d-1}$. On the other hand, using (\ref{specialmultiblock}), we obtain the number of  special partitions
\begin{align*}
S&=\sum_{i=1}^{d-1}(-1)^{i+1}\sum_{\lambda\subseteq\{d,\dots ,2d-2\}, \vert\lambda\vert=i}1^i \cdot 3^{d-1-i}\\
&=\sum_{i=1}^{d}(-1)^{i+1}\binom{d-2}{i}3^{d-1-i}\\
&=3^{d-1 }-2^{d-1}.
\end{align*}
Thus the number of complete intersection Jordan types with the Hilbert function in $(a)$ is equal to $A-S=2^{d-1}$.
\item[$(b)$] In this case we have that $\mathsf{j}=2d-3+k$ and $t_{d-1}=\cdots =t_{k+d-2}=d$, $t_{k+d-1}=d-1, \dots , t_{2d+k-3}=1$.  For each $i\in [d,d+k-3]$ we have $T_i=0$ and clearly  $A_i=1$ and $S_i=0$. We have $T_{d+k-2}=(1,1)$, so $A_{d+k-2}=2$ and $S_{d+k-2}=0$. There are $d-1$ more components for each $i\in [d+k-1, 2d+k-3]$ where $T_i=(1,2,1)$, $A_i=3$ and $S_i=1$, similar to the previous case. So the total number of  partitions in this case is $A=2\cdot 3^{d-1}$ \par
\noindent Using Equation (\ref{specialmultiblock}) we obtain the number of special partitions 
\begin{align*}
S&=\sum_{i=1}^{d+k-2}(-1)^{i+1}\sum_{\lambda\subseteq\{d,\dots ,2d-2\}, \vert\lambda\vert=i}1^i \cdot 3^{d-1-i}\cdot 2\\
&=2\sum_{i=1}^{d-1}(-1)^{i+1}\binom{d+k-3}{i}3^{d-1-i}\\
&=2\cdot 3^{d-1 }-2^{d}.
\end{align*}
Therefore the number of complete intersection Jordan types in this case is equal to $A-S=2^d$.
\end{itemize}
\end{proof}
\newpage
\listoffigures
\begin{ack} 
The first author was supported by the Swedish Research Council grant 
VR 2013-4545. The authors are grateful to the organizers of the conference ``Lefschetz Properties in Algebra, Geometry and Combinatorics'' at  Centro Internazionale per la Ricerca Matematica (CIRM) at Levico, Italy, in June 2018; and to the successor conference ``Lefschetz Properties in Algebra, Geometry and Combinatorics, II'' at
Centre International de Rencontres Math\'{e}matiques (CIRM) at Lumini, France in October 2019, where they participated in the working group on Jordan type. Helpful comments of a referee led us to focus the paper and improve clarity.\end{ack}

\end{document}